\documentclass[12pt]{amsart}
\usepackage{amsmath, amsthm, amssymb}
\usepackage{fullpage}
\usepackage{color}
\usepackage{hyperref}
\usepackage{enumitem}
\usepackage{ulem}

\newtheorem{theorem}{Theorem}
\newtheorem{lemma}{Lemma}
\newtheorem{proposition}[lemma]{Proposition}

\newtheorem{definition}[lemma]{Definition}
\newtheorem{remark}[lemma]{Remark}

\numberwithin{lemma}{section}

\newcommand{\half}{\frac{1}{2}}
\newcommand{\D}{\partial}
 \newcommand{\Ez}{E_0}
 \newcommand{\E}{E}
 \newcommand{\Elin}{E_{lin}}
 \newcommand{\Elind}{{E^{(2)}_{lin}}}

\newcommand{\cM}{{\mathcal M}}
\newcommand{\cG}{{\mathcal G}}
\newcommand{\cK}{{\mathcal K}}
\newcommand{\Gs}{\mathcal{G}^\sharp}
\newcommand{\Ks}{\mathcal{K}^\sharp}
\newcommand{\As}{{A^\sharp}}
\newcommand{\Ass}{{A^{\sharp\sharp}}}
\newcommand{\nfw}{{w_{NF}}}
\newcommand{\nfr}{{r_{NF}}}

\newcommand{\nfW}{{\W_{NF}}}
\newcommand{\nfR}{{R_{NF}}}

\newcommand{\hatw}{\hat w}
\newcommand{\hatr}{\hat r}

\numberwithin{equation}{section}

\newcommand{\R}{{\mathbb R}}

\newcommand{\Z}{{\mathbb Z}}

\renewcommand{\H}{{\mathcal H }}
\renewcommand{\AA}{\mathbf A}

\newcommand{\tw}{{\tilde w}}
\newcommand{\tr}{{\tilde r}}

\newcommand{\dH}{{\dot{\mathcal H} }}

\newcommand{\W}{{\mathbf W}}

\usepackage{xcolor}

\newcommand{\teal}[1]{\begingroup\color{black} #1\endgroup}

\begin{document}

\title{Two dimensional gravity waves at low regularity I: energy estimates}

\author{Albert Ai}
\address{Department of Mathematics, University of Wisconsin, Madison}
\email{aai@math.wisc.edu}
\author{Mihaela Ifrim}
 \address{Department of Mathematics, University of Wisconsin, Madison}
\email{ifrim@wisconsin.edu}
\author{ Daniel Tataru}
\address{Department of Mathematics, University of California at Berkeley}
\email{tataru@math.berkeley.edu}

\begin{abstract}
  This article represents the first installment of a series of papers concerned with low regularity solutions for the  water wave equations in two space dimensions. Our focus here is on sharp cubic energy estimates. Precisely, we introduce and develop the techniques to prove a new class of energy estimates, which we call \emph{balanced cubic estimates}.  This yields a key improvement over the earlier cubic estimates
  of Hunter-Ifrim-Tataru \cite{HIT}, while preserving  their scale invariant character and their position-velocity potential holomorphic coordinate formulation.
  
 Even without using any Strichartz estimates, these results
allow us to significantly lower the  Sobolev regularity threshold for local  well-posedness, drastically  improving 
earlier results obtained by Alazard-Burq-Zuily~\cite{abz,abz-str}, Hunter-Ifrim-Tataru~\cite{HIT} and
Ai~\cite{A}. 

\end{abstract}

\maketitle

\setcounter{tocdepth}{1}

\tableofcontents

\section{Introduction}
We consider the two dimensional water wave equations with infinite
depth and with gravity but without surface tension.  This is a free boundary problem which is governed by
the incompressible Euler's equations within the fluid domain, and with appropriate boundary conditions on the
water surface, which is the free boundary. Under the additional assumption that the flow is
irrotational, the two dimensional dynamics can be expressed in terms of a
one-dimensional evolution of the free boundary coupled with the trace of the velocity potential on the surface.  

The choice of the parametrization of the free
boundary plays an important role here, and can be viewed as a form of gauge freedom.  Historically there are three such choices of coordinates; the first two, namely  the Eulerian and Lagrangian coordinates, arise in the broader context  of fluid
dynamics. The third employs the so-called  conformal method, which is specific to two dimensional irrotational flows; this leads to what we call the holomorphic coordinates, which play a key role in the present paper. 

Our objective in this and subsequent papers is to improve, 
streamline, and simplify the analysis of the two dimensional
gravity wave equations. This is a challenging quasilinear, nonlocal,
non-diagonal system. We aim to develop its analysis in multiple ways,
including:

\begin{enumerate}
\item  prove better, scale invariant energy estimates,

\item improve the existing results on long time solutions, and

\item  refine the study of the dispersive properties and improve the low regularity theory.

\end{enumerate}

In the present article we carry out the first step of this program,
and obtain a new class of sharp, scale invariant energy estimates, which
we call \emph{balanced cubic energy estimates}.

As a consequence of our estimates, we are able to drastically lower 
the regularity threshold below the prior results
in \cite{HIT}, \cite{abz-str}, \cite{A}. This is despite the fact that
here we are using no Strichartz estimates, whose investigation
is left for a future installment; thus, this should be  seen 
as an improvement over \cite{HIT}, with further improvements 
yet to come. On the other hand, for reference we note that the second step of our program, namely the application of these results to the study of long time solutions has been subsequently carried out  in \cite{WW-global}.

We note that a family of cubic, scale invariant energy estimates has already
been obtained in \cite{HIT}. The key improvement here is that 
our new estimates are \emph{balanced}, as opposed to the 
unbalanced ones in \cite{HIT}; this will be further explained below.
We emphasize the fact that we also prove balanced cubic  estimates for the linearized equation; this is essential for the local well-posedness result.

The proof of these new energy estimates brings together and refines
the two main methods in the study of the long time dynamics for water 
waves, namely (i) the \emph{modified energy method} of the last two authors and collaborators \cite{BH}, \cite{HIT}, and (ii) the \emph{paradiagonalization method} of Alazard-Delort~\cite{ad}. This reflects our dual goal, which is to both prove the new 
energy estimates, and to reduce the study of the nonlinear equation
to the corresponding linear paradifferential equation. The last 
part, which here could simply be seen as an element of the proof
of the energy estimates, will become critical in the 
study of the dispersive properties in later work.

To conclude, we emphasize that the idea of balanced 
cubic energy estimates, first developed here, should not be seen as specific to water waves, or even to one dimensional flows,  but rather as a much more 
general principle. For instance, the reader may also 
look at our subsequent work \cite{MS}, where this idea is implemented for the hyperbolic minimal surface equation
in all dimensions.

\subsection{A brief history} 
The water wave equations in general, and the equations for gravity 
waves in particular, are fundamental in fluid dynamics and have 
received considerable attention over the years. While the following discussion unavoidably gives priority to the dynamical problem, we note that extensive work was also done toward understanding periodic and solitary waves. We refer the reader to \cite{groves,const}  for an overview of work in this direction.

 The first steps toward understanding the local theory for this problem are due to Nalimov and Ovsjannikov; see \cite{n,o} and references  therein.  Ovsjannikov primarily considered solutions to shallow water equations in spaces of analytic functions. Nalimov instead worked on the infinite depth problem, which is also the one considered here, and proved the first small data result in Sobolev spaces. His approach was later extended to the finite depth problem by Yosihara~\cite{y}, and to large data by Wu~\cite{wu2}.

One key observation in the study of water waves was
that one can fully describe the evolution in terms of the 
free boundary and the trace of the velocity potential on the free boundary. This is classical and possibly goes all the way back to Stokes. Further progress was made  in the work of Zakharov~\cite{zak}, where the Hamiltonian structure
of the problem was first uncovered; this eventually led to the current Eulerian formulation of the problem, which first appeared in \cite{CSS}. More recently, the  Eulerian setting was further exploited in the study of the local well-posedness; for further references we refer the reader to the more recent paper of Alazard-Burq-Zuily~\cite{abz} as well as to Lannes's book  
~\cite{L-book}.

The well-posedness result in ~\cite{abz} is based purely on energy estimates. In the two dimensional setting this was improved in 
\cite{HIT} using refined energy estimates, which exploit more 
of the structure of the problem; this required the use of holomorphic coordinates, which are discussed in the next subsection. Further 
improvements to the local theory in \cite{abz-str} were based 
on Strichartz estimates, following the model developed 
initially for the quasilinear 
wave equation by  Tataru~\cite{T, T2, T3} and 
Bahouri-Chemin~\cite{bc,bc2}. The Strichartz estimates, and implicitly the result of \cite{abz-str} were further improved by the first author in \cite{A1}, in all dimensions, and even further in \cite{A}
in two dimensions. This latter work represents in some sense the 
water wave analogue of the results of Smith-Tataru~\cite{ST}
in the nonlinear wave equation context.

\subsection{Water waves in holomorphic (conformal) coordinates}
The conformal method first appeared in the study of the 
stationary problem for solitary waves; one such example is in the work of Levi-Civita~\cite{LC}, but possibly also earlier. Ovsjannikov's work \cite{ov} is the  first that we are aware of where the conformal method is used in the study of the evolution problem. The complete evolution equations restricted to the boundary were independently written by Wu~\cite{wu2} and Dyachenko-Kuznetsov-Spector-Zakharov~\cite{zakharov} in slightly different forms. However it was only in Wu's paper ~\cite{wu2} that this formulation is fully exploited to prove local well-posedness in the large data problem. Finally, the formulation of gravity waves in holomorphic coordinates 
was revisited in \cite{HIT}, and it is that formulation that we will be using in the present paper.

Our results apply both in the case of the real line $\mathbb{R}$ and the periodic case $\mathbb{S}^1$.  Our equations are expressed in coordinates $(t,\alpha)$ where $\alpha$ corresponds to the holomorphic parametrization of the water domain by the lower half-plane restricted to the real line.  To avoid distracting technicalities we will 
do the analysis for the real line, and refer the reader to the 
discussion in \cite{HIT} concerning the (minor) changes in the periodic case.

To write the equations we use the Hilbert transform $H$, as well as the projection operator to negative frequencies,
\begin{equation*}
P= \frac12(I-iH).
\end{equation*}
Our main variables $(Z,Q)$ are functions of $(t,\alpha)$ which 
can be described as follows:
\begin{itemize}
    \item $Z$ denotes the trace of the conformal map on the real line, which is viewed as the boundary of the lower half-space. In other words, $Z$ represents
    the conformal parametrization of the free surface in complex notation.
    \item $Q$ is the holomorphic velocity potential, defined as  the trace on the real line of the holomorphic extension of the velocity potential $\phi$,  i.e.
    $\Re Q = \phi$. 
\end{itemize}

Both $Z$ and $Q$ are complex valued functions which satisfy an addition 
spectral condition. Precisely, both $Z-\alpha$ and $Q$  will be restricted to the closed subspace of holomorphic functions within various Sobolev spaces. Here we define holomorphic functions on $\mathbb{R}$  as those
whose Fourier transform is supported in $(-\infty,0]$; equivalently,
they admit a bounded holomorphic extension into the lower
half-space. This can be described by the relation $Pf
= f$.

There is a one dimensional degree of freedom in the choice of
$\alpha$, namely the horizontal translations. To fix this, in the real
case we consider waves which decay at infinity,
\[
\lim_{|\alpha \vert \to \infty} Z(\alpha) - \alpha = 0.
\]
In the periodic case we instead assume that $ Z(\alpha) - \alpha$ has
period $2\pi$ and purely imaginary average. We can also harmlessly
assume that $Q$ has real average.

In position-velocity potential holomorphic coordinates the water wave equations equations have the form
\begin{equation*}
\left\{
\begin{aligned}
& Z_t + F Z_\alpha = 0 \\
& Q_t + F Q_\alpha -i (Z-\alpha) + P\left[ \frac{|Q_\alpha|^2}{J}\right]  = 0, \\
\end{aligned}
\right.
\end{equation*}
where the gravity is normalized to $1$ and
\begin{equation*}
 F = P\left[ \frac{Q_\alpha - \bar Q_\alpha}{J}\right] , \qquad J = |Z_\alpha|^2.
\end{equation*}

For the derivation of the above equations, we refer the reader to the appendix of \cite{HIT}.
 Slightly different forms of these equations appeared in
Wu~\cite{wu2} and Dyachenko-Kuznetsov-Spector-Zakharov~\cite{zakharov}.
The system form of the equations, used here, is closer to \cite{zakharov},
but is in complex rather than real form, as in Wu's work \cite{wu2}.

It is convenient to remove the leading, linear part of $Z$ and work with a new variable, namely
\[
W = Z-\alpha .
\]
Then the  equations become
\begin{equation}
\label{ww2d1}
\left\{
\begin{aligned}
& W_t + F (1+W_\alpha) = 0 \\
& Q_t + F Q_\alpha -i W + P\left[ \frac{|Q_\alpha|^2}{J}\right]  = 0, 
\end{aligned}
\right.
\end{equation}
where
\begin{equation*}
 F = P\left[\frac{Q_\alpha - \bar Q_\alpha}{J}\right], \qquad J = |1+W_\alpha|^2.
 \end{equation*}
Note that $J$ represents the Jacobian of the conformal change of coordinates.

As the system \eqref{ww2d1} is fully nonlinear, a standard procedure
is to convert it into a quasilinear system by differentiating
it. Observing that almost no undifferentiated functions appear in
\eqref{ww2d1}, one sees that by differentiation we get a
self-contained first order quasilinear system for
$(W_\alpha,Q_\alpha)$.  However, this system turns out to be 
degenerate hyperbolic at leading order, so it is far better to diagonalize it.
This was done in \cite{HIT} using the operator
\begin{equation}
\AA(w,q) := (w,q - Rw), \qquad R := \frac{Q_\alpha}{1+W_\alpha}.
\label{defR}
\end{equation}
The factor $R$ above has an intrinsic meaning, namely it is the complex velocity 
on the water surface in the conformal parametrization. We also remark that
\[
\AA(W_\alpha,Q_\alpha) = (\W,R), \qquad \W : = W_\alpha.
\]
Thus, the pair $(\W,R)$ diagonalizes the differentiated
system. Indeed, a direct computation yields the self-contained system
\begin{equation} \label{ww2d-diff}
\left\{
\begin{aligned}
 &D_t \W + \frac{(1+\W) R_\alpha}{1+\bar \W}   =  (1+\W)M
\\
&D_t R = i\left(\frac{\W - a}{1+\W}\right).
\end{aligned}
\right.
\end{equation}
Here the {\em material derivative} $D_t$ is given by
\begin{align*}
D_t &= \D_t + b \D_\alpha,
\end{align*}
the {\em advection velocity} $b$ is
given by
\begin{equation*}
b = \Re F =  P \left[\frac{R}{1+\bar \W}\right] +  \bar P\left[\frac{\bar R}{1+\W}\right],
\end{equation*}
the real {\em Taylor coefficient} $1 + a$ is given by
\begin{equation}\label{a-def}
a := i\left(\bar P \left[\bar{R} R_\alpha\right]- P\left[R\bar{R}_\alpha\right]\right)
\end{equation}
(though we will often abuse terminology by referring to $a$ as the Taylor coefficient),
and the auxiliary function $M$, closely related to the material derivative of $a$,
is given by
\begin{equation}\label{M-def}
M :=  \frac{R_\alpha}{1+\bar \W}  + \frac{\bar R_\alpha}{1+ \W} -  b_\alpha =
\bar P [\bar R Y_\alpha- R_\alpha \bar Y]  + P[R \bar Y_\alpha - \bar R_\alpha Y].
\end{equation}
The function $Y$ above, given by
\begin{equation}\label{Y-def}
Y := \frac{\W}{1+\W},
\end{equation}
is introduced in order to avoid rational expressions above and in many
places in the sequel.  The system \eqref{ww2d-diff} governs an
evolution in the space of holomorphic functions, and will be used both
directly and in its projected version.

Observe that using these definitions, \eqref{ww2d1} may be expressed as
\begin{equation}
\label{ww2d1-b}
\left\{
\begin{aligned}
& D_t W + b = \bar R \\
& D_t Q -i W  = \bar P \left[ \frac{|Q_\alpha|^2}{J}\right],
\end{aligned}
\right.
\end{equation}
and \eqref{ww2d-diff} may be expressed as
\begin{equation}\label{ww2d-diff-Y}
\left\{
\begin{aligned}
 & D_t Y + |1 - Y|^2 R_\alpha  =  (1 - Y)M
\\
& D_t R - i(1 + a)Y = -ia.
\end{aligned}
\right.
\end{equation}

The functions $b$ and $a$ also play a fundamental role in the
linearized equations which are computed in the next section,
Section~\ref{s:linearized}. Indeed, these functions are essential to
understanding the quasilinear evolution, and appear in one form or
another in all works on the subject, see e.g. \cite{n,abz, L-book,wu}.
Our expressions here are similar in form to those in \cite{wu}.

\

We denote the linearized variables around a solution $(W, Q)$ of \eqref{ww2d1} by $(w,q)$ and, after the diagonalization, 
$$(w, r:=q-Rw).$$
The linearized equations (see \eqref{lin(wr)0}) have the form
\begin{equation*}
\left\{
\begin{aligned}
& D_t w  +  \frac{1}{1+\bar \W} r_\alpha
+  \frac{R_{\alpha} }{1+\bar \W} w  =\ (1+\W) (P \bar m + \bar P  m)
 \\
&D_t  r  - i  \frac{1+a}{1+\W} w  = \  \bar P n - P \bar n,
\end{aligned}
\right.
\end{equation*}
 where 
 \[ 
 m := \frac{r_\alpha +R_\alpha w}{J} + \frac{\bar R w_\alpha}{(1+\W)^2}, \qquad n := \frac{ \bar R(r_{\alpha}+R_\alpha w)}{1+\W}.
\]
In particular, the linearization of the system
\eqref{ww2d-diff} around the zero solution is
\begin{equation}\label{ww2d-0}
\left\{
\begin{aligned}
 & w_{ t} +  r_\alpha   =  0
\\
& r_t - i w = 0.
\end{aligned}
\right.
\end{equation}

An important role in our paper is played by the paradifferential
counterpart of the linearized equations, which is the linear flow
\begin{equation}\label{paralin(wr)-intro}
\left\{
\begin{aligned}
& T_{D_t} {w} + T_{1- \bar Y} \D_\alpha {r}
+ T_{(1 - \bar Y)R_\alpha} {w} = 0
 \\ 
& T_{D_t} {r} - i T_{1 - Y}T_{1 + a} {w} =  0,
\end{aligned}
\right.
\end{equation}  
where $T_f$ denotes the standard paraproduct operator 
and the operator $T_{D_t}$ which we will call the \emph{para-material derivative},
is given by
\[
T_{D_t} = \partial_t + T_b \partial_\alpha.
\]

In a similar vein, we will also consider the linearization of the equations \eqref{ww2d-diff} and its associated linear paradifferential flow, which we list here for reference in inhomogeneous form: 
\begin{equation} \label{paralin(hwhr)-intro}
\left\{
\begin{aligned}
 &T_{D_t} w  + T_{b_\alpha}  w + \D_\alpha T_{1 - \bar Y}T_{1 + W_\alpha} r = G \\
&T_{D_t}  r  + T_{b_\alpha} r  -iT_{(1 - Y)^2}T_{1 + a} w + T_M r = K.
\end{aligned}
\right.
\end{equation}

The  two sets of linearized equations are derived in Section~\ref{s:linearized}, where we also describe the connection
between them. The analysis of the linearized equations \eqref{lin(wr)}, carried out in Section~\ref{s:lin-est}, is a key component of this paper. One 
intermediate step in this analysis is the study of the 
above paradifferential equation \eqref{paralin(wr)-intro}.

The second paradifferential equation \eqref{paralin(hwhr)-intro}
will play a role in the study of the full differentiated equation \eqref{ww2d-diff}. The two paradifferential flows above will be shown to be equivalent in a suitable sense.

\subsection{Function spaces and control norms}
 The system \eqref{ww2d-0} is a well-posed linear evolution in the space $\H^0$ of holomorphic functions endowed with the $L^2 \times \dot H^{\frac12}$ norm. A conserved energy for this system is
\begin{equation}\label{E0}
\Ez (w,r) = \int \frac12 |w|^2 + \frac{1}{2i} (r \bar r_\alpha - \bar r r_\alpha)\, d\alpha.
\end{equation}
The nonlinear system \eqref{ww2d1} also admits a conserved energy, which has the form
\begin{equation}\label{ww-energy}
\E(W,Q) = \int \frac12 |W|^2 + \frac1{2i} (Q \bar Q_\alpha - \bar Q Q_\alpha)
- \frac{1}{4} (\bar W^2 W_\alpha + W^2 \bar W_\alpha)\, d\alpha.
\end{equation}
 As suggested by the above energy, our main function spaces for the 
 differentiated water wave system \eqref{ww2d-diff} are the
spaces $\H^s$ endowed with the norm 
\[
\| (\W,R) \|_{\H^s}^2 :=  \| \langle D \rangle^s (\W, R)\|_{ L^2 \times \dot H^\frac12}^2,
\]
where $s \in \R$. With these notations, the linear energy $\Ez$ is equivalent to (the square of) the $\H^0$ norm.

Almost all the estimates in this paper are scale invariant,
and to describe them it is very useful to also have homogeneous versions 
of the above spaces, namely the spaces $\dH^s$ endowed with the norm 
\[
\| (\W,R) \|_{\dH^s}^2 :=  \| |D|^s (\W, R)\|_{ L^2 \times \dot H^\frac12}^2.
\]
We caution the reader that, in order to streamline the exposition here, our 
notation for the energy spaces differs slightly from the notation used in \cite{HIT}.

The energy  estimates for  the solutions in \cite{HIT} were described in terms of the (time dependent) control norms $(A,B)$,
which are defined and redenoted here as
\begin{equation}\label{A-def}
A_0 = A := \|\W\|_{L^\infty}+\| Y\|_{L^\infty} + \||D|^\frac12 R\|_{L^\infty \cap B^{0}_{\infty, 2}},
\end{equation}
respectively
\begin{equation}\label{B-def}
A_{\frac12} =  B :=\||D|^\frac12 \W\|_{BMO} + \| R_\alpha\|_{BMO}.
\end{equation}

Instead, in this article the leading role will be played by
an intermediate control norm interpolating between $A_0$
and $A_{\frac12}$,
\begin{equation}\label{A14-def}
A_{\frac14} :=\| \W\|_{\dot B^{\frac14}_{\infty,2}} + \| R \|_{\dot B^{\frac34}_{\infty,2}}.
\end{equation}
Here the subscript of $A$ represents the difference in terms of derivatives between our control norm and scaling.
In particular  $A_{s}$ corresponds to and is controlled by the homogeneous $\dH^{\frac12+s}$ norm of $(\W, R)$, and $A_0$ is a scale invariant quantity. Concerning $A_{\frac14}$, we note the following inequality, 
\begin{equation}\label{A14-def+}
\||D|^\frac14 \W\|_{BMO} + \| |D|^\frac34 R \|_{BMO} \lesssim A_{\frac14}.
\end{equation}
With one notable exception, in all of the paper we use only this slightly weaker $BMO$ bound.

In addition to the pointwise scale invariant norm measured by $A$,
we will also need a stronger scale invariant Sobolev control norm $\As$
defined by 
\begin{equation}\label{Asharp-def}
\teal{\As :=\||D|^{\frac14}\W\|_{L^4} + \||D|^{3/4} R\|_{L^4}.}
\end{equation}
A number of implicit constants in our estimates will depend on $\As$.

\subsection{The main results}

For comparison purposes, we begin by recalling the energy estimates in \cite{HIT} 
for the differentiated system \eqref{ww2d-diff}.

\begin{theorem}\label{t:ee}
For each $n \geq 0$ there exists an energy functional $E_n$
associated to the differentiated equation \eqref{ww2d-diff}
with the following two properties:

(i) Energy equivalence:
\begin{equation}
E_n(\W,R) \approx_A \|(\W,R)\|_{\dH^n}^2    
\end{equation}

(ii) Cubic energy bound:
\begin{equation}
\frac{d}{dt} E_n(\W,R) \lesssim_A AB \|(\W,R)\|_{\dH^n}^2.    
\end{equation}
\end{theorem}

The above energy bounds are already scale invariant, so one cannot 
hope to simply lower the regularity of any of the factors. 
However, the two factors in the product $AB$ are unbalanced,
with the second factor $B = A_{\frac12}$ being the critical one 
in determining low regularity well-posedness thresholds.

A main idea in this article is that, by taking full advantage
of the structure of the water wave system, one can better rebalance the product $AB$ into $A_{\frac14}^2$.  Along the way, we also allow $n$ (now redenoted by $s$) to be non-integer:

\begin{theorem}\label{t:main}
For each $s \geq 0$ there exists an energy functional $E_s$
associated to the differentiated equation \eqref{ww2d-diff}
with the following two properties:

(i) Energy equivalence if $A \ll 1$:
\begin{equation}
E_s(\W,R) \approx \|(\W,R)\|_{\dH^s}^2    
\end{equation}

(ii) Balanced cubic energy bound:
\begin{equation}
\frac{d}{dt} E_s(\W,R) \lesssim_{\teal{\As}} A_{\frac14}^2 \|(\W,R)\|_{\dH^s}^2.  
\end{equation}

\end{theorem}
One should think of the energies $E_s$ as refined versions of the earlier energies $E_n$ in Theorem~\ref{t:ee}. In particular $E_s$ would satisfy the conclusion of Theorem~\ref{t:ee}, but not vice versa.

In the same vein, in this article we also prove a similar balanced  energy bound for the linearized equation \eqref{lin(wr)} in the space $\dH^{\frac14}$; see Theorem~\ref{t:balancedenergy} in Section~\ref{s:lin-est}. This improves a corresponding 
unbalanced bound in \cite{HIT} (done there in $\H^0$). A small price to pay is the use of the stronger control norm $\As$ instead of $A$ in the implicit constant; this has no other impact in either this paper or the following installments of our work.

As a consequence of the above energy bound and of the similar bound for the linearized equation, we will prove the following 
low regularity well-posedness result:

\begin{theorem}
\label{t:lwp}
 Let  $ s \geq s_0 = 3/4$.  The system \eqref{ww2d-diff} is locally well-posed for all initial data $(\W_0,R_0)$ in $\H^s(\mathbb{R})$
 (or $\mathbb T$) so that $A(0)= A(\W_0,R_0)$ is small.
  Further, the solution can be continued for as long as $A$  remains 
 small and $A_{\frac14} \in L^2_t$. 
\end{theorem}

We complete this theorem with two remarks which discuss different possible 
formulations of this result.

\begin{remark}\label{r:A}
The smallness assumption on $A$ is not essential here, but it does allow us to avoid 
a number of distracting technicalities and instead focus on the main new ideas of the
present paper. One should be able to replace this with a bound of the form 
$|1+\W_0| > c > 0$, which in a nutshell prevents the free surface from developing corners.
To see the potential effect of eliminating the smallness assumption on $A$ we refer the 
reader to the local well-posedness argument in \cite{HIT}.
\end{remark}

\begin{remark}
The $\dH^{\frac34}$ regularity of the data and of the solution is needed in this paper
in order to provide uniform control over $A_\frac14$. On the other hand, the 
low frequency regularity $\H^0$ is actually never used, but simply propagated in time.
With some extra work in the local well-posedness argument one could relax the low 
frequency assumption and prove well-posedness in the larger space $\dH^\frac14 \cap \dH^\frac34$, and possibly even better than that.
\end{remark}

It is revealing to compare this result with prior work. 
The direct comparison is with the results in \cite{abz} and \cite{HIT}, both of which are purely based on energy estimates.
These results correspond to $s = 1+\epsilon$, respectively $s = 1$.
Thus our improvement is by $1/4$ of a derivative, which is half-way 
to the absolute threshold $s = \frac12$ given by scaling.

One could also compare with the results of \cite{abz-str} 
with $s = 1-\frac{1}{24}+\epsilon$ or \cite{A} with $s = 1-\frac18 + \epsilon$, and our new result is still considerably better.
However, such a comparison would be biased because both \cite{abz-str}
and \cite{A} use Strichartz, while here we do not. Strichartz based improvements to the above result will be discussed in a subsequent paper, where another substantial improvement is obtained for the well-posedness range of $s$. 

A limitation of our exposition is that we work with the infinite 
depth problem, whereas the local well-posedness threshold should 
not depend on that. Indeed, we expect that this result easily carries over to the finite depth case.

One interesting consequence of the above energy estimates and well-posedness result is that we can also prove a low regularity cubic lifespan bound: 

\begin{theorem}
\label{t:cubic}
  Let $\epsilon \ll 1$,  $0 < \delta \leq \frac12$ and $c_0 \ll 1$.  Assume that the initial data for the equation   \eqref{ww2d-diff}  satisfies
\begin{equation}\label{small-data-14}
\|(\W(0), R(0))\|_{\dH^{\frac{3}{4}}} \leq \epsilon, 
\end{equation}
as well as
\begin{equation} \label{scale-inv}
\|(\W(0), R(0))\|_{ \dH^{\frac12-\delta}}^\frac14 \|(\W(0), R(0))\|_{ \dH^{\frac{3}{4}}}^\delta  \leq c_0^{\delta+\frac14}.
\end{equation}
Then the solution exists on an $\epsilon^{-2}$ sized time interval
$I_\epsilon = [0,T_\epsilon]$, and satisfies a similar bound. In
addition, the estimates
\[
\sup_{t \in I_\epsilon} \| (\W(t), R(t))\|_{\dH^s} \lesssim \| (\W(0), R(0))\|_{\dH^s}
\]
hold for all $s \geq 0$ whenever the right hand side is finite.
\end{theorem}

We note that in the current version  we have 
slightly adjusted the statement of this theorem in order to make it fully scale invariant and to more accurately reflect the exact consequences of our main result in Theorem~\ref{t:main}. This does not change the (very brief) proof in any significant way. This is in response to a discussion with S. Wu, in order to facilitate a direct comparison of this result 
with the more recent results in \cite{Wu-2020}.
The constant $c_0$ is universal, 
and comes from the smallness requirement on $A$ in Theorem~\ref{t:main} (see also Remark~\ref{r:A}). The power of $c_0$ is simply used for homogeneity purposes. See also \cite{BFP} for related results.

\medskip

Another key point of this paper, which is less important for the results stated above but will be critical in our follow-up work,
is that we establish a normal form based  equivalence between 
our differentiated variables $(\W,R)$ and normalized variables
$(\tilde \W, \tilde R)$, which in turn solve a paradifferential 
form of the linearized equations \eqref{lin(wr)}, with a favourable perturbative source term:

\begin{theorem}\label{t:nf}
There exists a normal form transformation 
\[
(\W,R) \to (\W_{NF},R_{NF}) 
\]
with the following properties: 

\begin{itemize}
    \item[(i)] \textbf{Regularity:}  The normal form transformation is smooth
    in $\H^s$ for all $s > \frac12$, with comparable control parameters $A$, $\As$, $A_{\frac14}$ for $(\W,R)$ and $(\W_{NF},R_{NF})$.

    \item[(ii)] \textbf{Uniform bounds:} For each $s \geq 0$ we have
\begin{equation}\label{WR-nf-bdd}
\| (\W,R) - (\W_{NF},R_{NF})\|_{\dH^s} \lesssim_{A} A \|(\W,R)\|_{\dH^s}.
\end{equation}

\item[(iii)] \textbf{Normalized equation:} The normal form variables $(\W_{NF},R_{NF}) $ solve an inhomogeneous paradifferential equation \eqref{paralin(hwhr)-intro} with sources $(G,K)$
satisfying
\begin{equation}\label{WR-nf-err}
\| (G, K)\|_{\dH^s} \lesssim_{A} A_{\frac14}^2  \|(\W,R)\|_{\dH^s}, \quad s \geq 0.
\end{equation}
\end{itemize}
\end{theorem}

From a normal forms perspective, one might ask whether
the results in this paper could be achieved directly 
using a normal form transformation. As far as the low frequency analysis is concerned, there is indeed a very simple normal form transformation which eliminates the quadratic terms in the equation \eqref{ww2d1}. As computed in \cite{HIT}, this has the form
\begin{equation}
\left\{
\begin{aligned}
\tilde W =& \  W - 2 P (\Re W  W_\alpha)
\\
\tilde Q =& \  Q - 2 P(\Re W Q_\alpha).
\end{aligned}
\right.
\label{nft1}
\end{equation}
Unfortunately this normal form transformation is unbounded at high frequency, which simply reflects the quasilinear character of the problem. 

To compound the difficulty,
if one attempts to differentiate this in order to obtain a normal form transformation for the differentiated equation \eqref{ww2d-diff}, this
will contain inverse derivatives so it will 
be unbounded also at low frequencies.

While we cannot use the above normal form directly,
we will nevertheless rely on it as a guide when computing partial normal form transformations, both 
in its differentiated form for the equation \eqref{ww2d-diff}, and in linearized form for the linearized equation \eqref{lin(wr)}.


\subsection{The structure of the paper} 
Much of the analysis in this paper is phrased in the language of  paradifferential calculus. To prepare for this, we  begin in Section~\ref{s:multilinear} with a review of notations and of some of the classical Coifman-Meyer paraproduct estimates, some further bounds from \cite{HIT}, as well as several new bounds. 

This is immediately followed and applied in Section~\ref{s:wwbounds} by a series of bounds which are specific to the water wave system. We successively consider the Taylor coefficient $a$, the advection velocity $b$, as well as
several other auxiliary functions.

Next we turn our attention to linearizations and para-linearizations.
Our goal in Section~\ref{s:linearized} is to derive 
the linearizations of the original equations \eqref{ww2d1}
and of the differentiated equations \eqref{ww2d-diff}, as 
well as the connection between them. To each of these two linearized equations we associate the respective paradifferential flows.

After that, in Section~\ref{s:lin-est}, we study the well-posedness
of the first linearized equations \eqref{lin(wr)} in the space
$\dH^{\frac14}$, as well as the corresponding balanced energy estimates. As an intermediate step, we prove a similar result 
for the corresponding paradifferential flow \eqref{paralin(wr)-intro}, and carry out a modified normal form type reduction.

In Section~\ref{s:ee} we consider the energy estimates in Sobolev spaces $\dH^s$, $s \geq 0$,  for  the solutions to the nonlinear equations. The principal part of these equations is closely related to the 
associated paradifferential linearized equations \eqref{paralin(hwhr)-intro}. The bulk of the work is to show that there exists a modified normal form type reduction of the full equations \eqref{ww2d-diff} to the paradifferential ones \eqref{paralin(hwhr)-intro}. The latter is then shown to be equivalent to the paradifferential flow \eqref{paralin(wr)-intro},
studied in the previous section.

Section~\ref{s:lwp} contains the proof of the local well-posedness
result in Theorem~\ref{t:lwp} for the equations \eqref{ww2d-diff}, as well as 
the proof of the cubic lifespan bounds in Theorem~\ref{t:cubic}.  We begin with more regular data, e.g. $\H^1$, for which well-posedness has already been proved in \cite{HIT}. The rough $\H^\frac34$ solutions are obtained as uniform limits of smooth solutions by using the estimates for the linearized equation.  The same construction yields their continuous dependence on the data.

\subsection{Acknowledgements}
The first author was partially supported by the Henry Luce Foundation as well as by the Simons Foundation. The second author was supported by a Luce Assistant Professorship, by the Sloan Foundation, and by an NSF CAREER grant DMS-1845037. The last author was supported by the NSF grant DMS-1800294 as well as by a Simons Investigator grant from the Simons Foundation. 

\section{Norms and multilinear  estimates} \label{s:multilinear}

Here we review some of the function spaces and estimates used later in the paper. We are primarily using the paradifferential calculus,
and our analysis will require a heavy dose of paraproduct type estimates. 
Many of these are relatively standard, like the classical Coifman-Meyer bounds
and some of their generalizations. Several others are from \cite{HIT}. 
To those we add a few more useful multilinear paraproduct type bounds.

\subsection{Function spaces}
We use a standard 
 Littlewood-Paley decomposition in frequency
\begin{equation*}
1=\sum_{k\in \mathbf{Z}}P_{k},
\end{equation*}
where the multipliers $P_k$ have smooth symbols localized at frequency $2^k$.

A good portion of our analysis happens at the level of homogeneous
Sobolev spaces $\dot{H}^{s}$, whose norm is given by
\begin{equation*}
\Vert f\Vert_{\dot{H}^{s}}\sim \Vert ( \sum_{k}\vert 2^{ks} P_{k}f\vert^2 )^{1/2}  \Vert _{L^2}=
\| 2^{ks} P_k f \|_{L^2_{\alpha} \ell^2_k}.
\end{equation*}
We will also use  the Littlewood-Paley
square function and its restricted version,
\begin{equation*}
\displaystyle S(f)(\alpha):=\bigg( \sum_{k\in {\mathbf Z}} |P_k(f)(\alpha)|^2\bigg)^\frac{1}{2}, \qquad S_{>k}(f)(\alpha) = \left( \sum_{j > k} |P_j f|^2 \right)^\frac12.
\end{equation*}
The Littlewood-Paley inequality  is recalled below
\begin{equation}\label{lp-square}
   \displaystyle \|S(f)\|_{L^p({\mathbf R})}\simeq_{p} \|f\|_{L^p({\mathbf R})}, \qquad  1<p<\infty.
   \end{equation}
By duality this also yields the estimate
\begin{equation}
\label{useful}
\Vert \sum_{k\in \mathbf{Z}}P_{k}f_{k}\Vert_{L^p}\lesssim
\Vert \sum_{k\in \mathbf{Z}}(\vert f_{k}\vert ^2)^{1/2}\Vert_{L^p}, \qquad 1 < p < \infty,
\end{equation}
for any sequence of functions $\left\{ f_k\right\}_k\in  L^p_{\alpha} l^2_k$.

The $p= 1$ version of the above estimate for the Hardy space $ H_1$ is
\begin{equation}\label{h1-square}
\Vert f\Vert_{H_1}\simeq \| S(f)  \|_{L^1_{\alpha} \ell^2_k},
\end{equation}
which by duality implies the BMO bound
\begin{equation}
\label{useful-bmo}
\Vert \sum_{k\in \mathbf{Z}}P_{k}f_{k}\Vert_{BMO}\lesssim
\Vert S(f)\Vert_{L^\infty}.
\end{equation}
The square function characterization of BMO is slightly different,
\begin{equation}\label{bmo-square}
\| u\|_{BMO}^2 \approx \sup_k \sup_{|Q|=2^{-k}} 2^k \int_Q |S_{>k} (u)|^2 \, dx.
\end{equation}
For real $s$ we define the homogeneous spaces\footnote{These are the same as the Triebel-Lizorkin spaces $F^{s}_{\infty,2}$.} $BMO^s$
with norm
\[
\| u\|_{BMO^s}  = \| |D|^s u\|_{BMO}.
\]
\subsection{Coifman-Meyer and and Moser type estimates}

In the context of bilinear estimates a standard tool is to consider a Littlewood-Paley 
paraproduct type decomposition of the product of two functions,
\[
 f g = \sum_{k \in \Z} f_{<k-4} g_k +  \sum_{k \in \Z} f_{k} g_{<k-4} + \sum_{|k-l| \leq 4}
f_k g_l := T_f g + T_g f + \Pi(f,g).
\]
We also define the restricted diagonal sum
\[
\Pi_{\geq k}(f,g) = \sum_{j \geq k} f_j g_j. 
\]
Here and below we use the notation $f_k = P_k f$, $f_{<k} = P_{<k} f$, etc.
By a slight abuse of notation, in the sequel we will omit the frequency separation 
from our notations in bilinear Littlewood-Paley decomposition; for instance instead of the 
above formula we will use the shorter expression
\[
  f g = \sum_{k \in \Z} f_{<k} g_k +  \sum_{k \in \Z} f_{k} g_{<k} + \sum_{k \in \Z} f_k g_k.
\]

Away from the exponents $1$ and $\infty$ one has a full set of estimates
\begin{equation}\label{CM}
\| T_f g\|_{L^r} + \| \Pi(f,g)\|_{L^r} \lesssim \|f\|_{L^p} \|g\|_{L^q}, \qquad \frac{1}r = \frac{1}{p} +
\frac{1}{q}, \qquad 1 < p,q,r < \infty.
\end{equation}
Corresponding to $q = \infty$ one also  has a BMO estimate 
\begin{equation}\label{CM-BMO}
\| T_f g\|_{L^p} + \| \Pi(f,g)\|_{L^p} \lesssim \|f\|_{L^p} \|g\|_{BMO},  \qquad 1 < p < \infty,
\end{equation}
while for the remaining product term we have the weaker bound
\begin{equation}\label{CM+}
\| T_g f\|_{L^p}  \lesssim \|f\|_{\dot W^{s,p}} \|g\|_{BMO^{-s}},  \qquad 1 < p < \infty, s > 0.
\end{equation}
These in turn lead to the commutator bound
\begin{equation}\label{CM-com}
\|[P, g] f \|_{L^p}  \lesssim \|f\|_{L^p} \|g\|_{BMO},  \qquad 1 < p < \infty.
\end{equation}

 We also need an extension of this, namely 
 
\begin{lemma}\label{l:com}
 The following commutator estimates hold 
 \teal{for $1 < p < \infty$:
  
\begin{equation}  \label{first-com}
\Vert |D|^s \left[ P,g\right] |D|^\sigma f \Vert _{L^p}\lesssim \Vert
  |D|^{\sigma+s} g\Vert_{BMO} \Vert f\Vert_{L^p}, \qquad \sigma \geq 0, \ \ s \geq 0,
\end{equation}
\begin{equation}\label{second-com}
\Vert |D|^s  \left[ P,g\right] |D|^\sigma f \Vert _{L^p}\lesssim \Vert
  |D|^{\sigma+s}  g\Vert_{L^p} \Vert f\Vert_{BMO}, \qquad  \sigma > 0, \ \ \, s \geq 0.
\end{equation}}
\end{lemma}
We remark that later this is applied to functions which are holomorphic/antiholomorphic, but that no such assumption is made above.

Next we consider some similar product type estimates involving $BMO$ and
$L^\infty$ norms.

\begin{proposition}\label{p:bmo}
a) The following estimates hold:
\begin{equation}\label{bmo-bmo}
\| \Pi(u,v) \|_{BMO} \lesssim \| u\|_{BMO} \|v\|_{BMO},
\end{equation}
\begin{equation}\label{bmo-bmo+}
\| P_{\leq k} \Pi_{\geq k}(u,v) \|_{L^\infty} \lesssim \| u\|_{BMO} \|v\|_{BMO},
\end{equation}
\begin{equation}\label{bmo-infty} 
\| T_u v\|_{BMO} \lesssim \| u\|_{L^\infty} \|v\|_{BMO},
\end{equation}
\begin{equation}\label{bmo>infty}
\| T_u v\|_{BMO} \lesssim \| u\|_{BMO^{-\sigma}} \|v\|_{BMO^\sigma}, \quad 
\qquad \sigma > 0.
\end{equation}

b) For $s > 0$ the space $L^\infty \cap BMO^{s}$ is an algebra,
\begin{equation}\label{bmo-alg}
\| uv\|_{BMO^{s}} \lesssim \| u\|_{L^\infty} \|v\|_{BMO^s}+
 \| v\|_{L^\infty} \|u\|_{BMO^s}.
\end{equation}

c) In addition, the following Moser estimate holds for a smooth function $F$ vanishing at $0$:
\begin{equation}\label{bmo-moser}
\| F(u)\|_{BMO^{s}} \lesssim_{\|u\|_{L^\infty}} \|u\|_{BMO^s}.
\end{equation}
\end{proposition}

This Proposition is a paradifferential reformulation of 
results from \cite{HIT}, with the exception of the estimate 
\eqref{bmo-bmo+}, which represents a mild strengthening of \eqref{bmo-bmo}
in the hi $\times$ hi $\to$ low scenario. Its proof is a direct consequence 
of the square function characterization of $BMO$ in \eqref{bmo-square}, and is left
for the reader.

A more standard algebra estimate and the corresponding Moser bound is as follows:
 \begin{lemma}
Let $\sigma > 0$. Then $\dot H^\sigma \cap L^\infty$ is an algebra, and
\begin{equation}
\| fg\|_{\dot H^\sigma} \lesssim \| f\| _{\dot H^\sigma} \|g\|_{L^\infty} +
\|f\|_{L^\infty}  \| g\| _{\dot H^\sigma}.
\end{equation}
In addition, the following Moser estimate holds for a smooth function $F$ vanishing at 0:
\begin{equation}\label{bmo-hs}
\| F(u)\|_{\dot H^\sigma } \lesssim_{\|u\|_{L^\infty}} \|u\|_{\dot H^\sigma }.
\end{equation}
\end{lemma}

\subsection{Paraproduct estimates}
Here we record the following para-commutator, paraproduct, and para-associativity lemmas: 

\begin{lemma}[Para-commutators]\label{l:para-com}
 Assume that $\gamma_1, \gamma_2 < 1$. Then we have
\begin{equation}
\| T_f T_g - T_g T_f \|_{\dot H^{s} \to \dot H^{s+\gamma_1+\gamma_2}} \lesssim 
\||D|^{\gamma_1}f \|_{BMO}\||D|^{\gamma_2}g\|_{BMO}
\end{equation}
\teal{and the $L^4$ based version
\begin{equation}
\| T_f T_g - T_g T_f \|_{\dot W^{s, 4} \to \dot W^{s+\gamma_1+\gamma_2, 4}} \lesssim 
\||D|^{\gamma_1}f \|_{BMO}\||D|^{\gamma_2}g\|_{BMO},
\end{equation}
as well as the unbalanced version
\begin{equation}
\| T_f T_g - T_g T_f \|_{ BMO^{s} \to \dot H^{s+\gamma_1+\gamma_2}} \lesssim 
\||D|^{\gamma_1}f \|_{L^2}\||D|^{\gamma_2}g\|_{BMO}.
\end{equation}
}
\end{lemma}
We note that the $BMO$ norms are not strictly speaking necessary; we only need
the Besov norms $\dot B^{\gamma}_{\infty,\infty}$. The limiting case $\gamma_j = 1$
is more delicate, as there one would need instead $\dot W^{1.\infty}$; see also \eqref{com-limit} below.

\begin{proof}
By orthogonality we can fix the frequency of the argument $u$; call it $\lambda$.
If $f$ is supported at frequency $\ll \lambda$, then $T_f u = fu$. Since 
multiplication is commutative, the only nontrivial contributions
are those where at least one of the two frequencies of 
$f$ and $g$ is comparable to $\lambda$. We consider two cases: 
\medskip

(i) Both $f$ and $g$ are at frequency $O(\lambda)$. Then we neglect 
the commutator and estimate the output directly using the $\dot B^{\gamma_j}_{\infty,\infty}$
norms.
\medskip

(ii) $f$ is at frequency $O(\lambda)$, and $g$ is at frequency $\ll \lambda$.
Then the commutator becomes
\[
 [T_{f_\lambda}, T_{g_{\ll \lambda}}] u_\lambda = [T_{f_\lambda}, g_{\ll \lambda}] u_\lambda = \lambda^{-1} L(f_\lambda,\partial_\alpha g_{\ll \lambda},u_\lambda),
\]
where $L$ denotes a translation invariant trilinear form with integrable kernel.
Hence we obtain
\[
\| [T_{f_{\lambda}}, T_{g_{\ll \lambda}} ] u_\lambda \|_{L^2} \lesssim \| f_\lambda\|_{L^\infty} \| \partial_\alpha g_{\ll \lambda}\|_{L^\infty} \|u_\lambda\|_{L^2}
\]
and the desired conclusion easily follows.





\end{proof}

\begin{lemma}[Para-products]\label{l:para-prod}
Assume that $\gamma_1, \gamma_2 < 1$, $\gamma_1+\gamma_2 \geq 0$. Then
\begin{equation}
\| T_f T_g - T_{fg} \|_{\dot H^{s} \to \dot H^{s+\gamma_1+\gamma_2}} \lesssim 
\||D|^{\gamma_1}f \|_{BMO}\||D|^{\gamma_2}g\|_{BMO}
\end{equation}
\teal{and the $L^4$ based version
\begin{equation}
\| T_f T_g - T_{fg}\|_{\dot W^{s, 4} \to \dot W^{s+\gamma_1+\gamma_2, 4}} \lesssim 
\||D|^{\gamma_1}f \|_{BMO}\||D|^{\gamma_2}g\|_{BMO}.
\end{equation}}

\end{lemma}
\teal{ One may also ask whether this extends also to the case $\gamma_1= \gamma_2=1$. As it turns out, there one would have to use $L^\infty$ norms instead,
 \begin{equation}\label{com-limit}
\| T_f T_g - T_{fg} \|_{\dot H^{s} \to \dot H^{s+2}} \lesssim \|\partial f \|_{L^\infty}\|\partial g\|_{L^\infty}.
\end{equation}
This is however not needed in the present paper.}

\begin{proof}
Reasoning as in the previous Lemma, we can fix the argument $u$'s frequency 
to $\lambda$, and full cancellation occurs unless at least one of the frequencies
of $f$ and $g$ is $\gtrsim \lambda$. Thus we distinguish three cases, the first two of which are the same as before. Case (i) is identical, and also case (ii) in view of the relation
\[
 [T_{f_\lambda} g_{\ll \lambda} - T_{f_\lambda g_{\ll \lambda}}] u_\lambda =\lambda^{-1} L(f_\lambda,\partial_\alpha g_{\ll \lambda},u_\lambda).
\]
Thus we consider the remaining case:

(iii) The frequency of $f$ is $\gg \lambda$. Then the frequency of $g$ must be comparable. We consider the worst case scenario, when $\gamma_1+\gamma_2 = 0$.
Then we have 
\[
(fg)_{< \lambda} = P_{< \lambda}\left(\sum_{\mu > \lambda} f_\mu g_\mu \right)
\]
and we can use \eqref{bmo-bmo+} to estimate 
\[
\| (fg)_{< \lambda} \|_{L^\infty} \lesssim \||D|^{\gamma_1}f \|_{BMO}\||D|^{\gamma_2}g\|_{BMO},
\]
which suffices.

\end{proof}

\begin{lemma}[Para-associativity]\label{l:para-assoc}
For $s + \gamma_2 \geq 0, s + \gamma_1 + \gamma_2  \geq 0$, and $\gamma_1 < 1$
we have
\begin{equation}
\| T_f \Pi(v, u) - \Pi(v, T_f u)\|_{\dot H^{s + \gamma_1+\gamma_2}} \lesssim 
\||D|^{\gamma_1}f \|_{BMO}\||D|^{\gamma_2}v\|_{BMO} \|u\|_{\dot H^{s}}
\end{equation}
\teal{
and the $L^4$ based version
\begin{equation}\label{para-ass4}
\| T_f \Pi(v, u) - \Pi(v, T_f u)\|_{\dot W^{s + \gamma_1+\gamma_2,4}} \lesssim 
\||D|^{\gamma_1}f \|_{BMO}\||D|^{\gamma_2}v\|_{BMO} \|u\|_{\dot W^{s,4}},
\end{equation}
as well as a version where $f$ is measured in an $L^2$ based norm
\begin{equation}\label{para-ass-lin}
\| T_f \Pi(v, u) - \Pi(v, T_f u)\|_{\dot H^{1/4}} \lesssim 
\|f \|_{\dot H^{1/4}}\||D|^{-1/4}v\|_{BMO} \||D|^{1/4}u\|_{BMO}.
\end{equation}}
\end{lemma}
We remark that one could easily expand the range of indices
in \eqref{para-ass-lin}, but for simplicity we have limited ourselved to this case, which is used in a single thread in this paper in the study of the linearized equation in Section~\ref{s:linearized}.
\begin{proof}
We remark that our convention in this paper will be that the $\Pi$ operator
also contains the projection $P$. This is immaterial for the purpose of this lemma: We can prove the result without $P$, and then harmlessly insert it, as it commutes with $T_f$. We denote by $\lambda$ the joint frequencies of $u$ and $v$.
Then only frequencies $\leq \lambda$ of $f$ will contribute. The output 
of frequencies $O(\lambda)$ in $f$ yields no cancellations, but it is easy to bound
the two terms separately using two Coifman-Meyer estimates. Thus in what follows
we consider only frequencies $\ll \lambda$ of $f$, in which case we have 
\[
(T_{f_{\ll \lambda}} u)_\lambda = (f_{\ll \lambda} u)_\lambda 
= f_{\ll \lambda} u_\lambda +   \lambda^{-1} L(\partial_\alpha f_{\ll \lambda},u_\lambda).
\]
The second term is easy to estimate directly in $L^2$;  this 
leads to the restriction $\gamma' < 1$. For the  contribution of the first term 
we use cancellations to write
\[
\sum_\lambda T_{f_{\ll \lambda}} (u_\lambda v_\lambda) - f_{\ll \lambda} u_\lambda v_\lambda = \sum_{\mu \ll \lambda}
f_\mu P_{\lesssim \mu} (u_\lambda v_\lambda) = 
\sum_{\lambda}\Pi(f_{\ll\lambda},u_\lambda v_\lambda)
+ T_{u_\lambda v_\lambda} f_{\ll \lambda}.
\]
Harmlessly reinserting frequencies of size $\lambda$ in $f$, we can replace the 
above expression by
\[
\Pi(f, \Pi(u,v)) + T_{\Pi(u,v)} f,
\]
at which point it suffices to  apply the Coifman-Meyer bounds twice. 

\teal{The proof of the $L^4$ type bounds \eqref{para-ass4} and of 
\eqref{para-ass-lin} is similar.}
\end{proof}

We finish with an alternative form of the para-associativity bound:
\begin{lemma}\label{l:para-assoc2}
We have, for $\gamma_1, \gamma_2 < 1$, $s + \gamma_1+\gamma_2 \geq 0, s + \gamma_2 \geq 0$, and $\bar v = \bar P \bar v$,
\begin{equation}
\| T_f P(\bar v u) - P (\bar v T_f u)\|_{\dot H^{s + \gamma_1+\gamma_2}} \lesssim 
\||D|^{\gamma_1}f \|_{BMO}\||D|^{\gamma_2}v\|_{BMO} \|u\|_{\dot H^{s}}
\end{equation}
\teal{
and the $L^4$ based version
\begin{equation}
\| T_f P(\bar v u) - P (\bar v T_f u)\|_{\dot W^{s + \gamma_1+\gamma_2,4}} \lesssim 
\||D|^{\gamma_1}f \|_{BMO}\||D|^{\gamma_2}v\|_{BMO} \|u\|_{\dot W^{s,4}},
\end{equation}
as well as
\begin{equation}
\| T_f P(\bar v u) - P (\bar v T_f u)\|_{\dot H^{1/4}} \lesssim 
\|f \|_{\dot H^{1/4}}\||D|^{-1/4}v\|_{BMO} \||D|^{1/4}u\|_{BMO}.
\end{equation}}


\end{lemma}

\begin{proof}
This is a consequence of Lemmas~\ref{l:para-com} and \ref{l:para-assoc}, using also the fact that $P$ eliminates the case where $\bar v$ is high frequency.

\end{proof}

\section{Water wave related bounds}
\label{s:wwbounds}

Here we consider estimates for objects related to the water wave equations,
primarily the Taylor coefficient $a$ and advection velocity $b$. We recall that these are given by
\[
a = 2 \Im P[R \bar R_\alpha], \qquad b = 2 \Re P \left[ \frac{R}{1+\bar \W}\right]
= 2 \Re ( P[(1 - \bar Y)R]).
\]
In addition to these, we will also consider several auxiliary functions,
namely 
\[
Y= \dfrac{\W}{1+\W}, \qquad X = T_{1-Y}W, \qquad M = 2 \Re P[R \bar Y_\alpha - \bar R_\alpha Y].
\]
The auxiliary variable $Y$ is often used as a coefficient in many of our computations. The function $X$ is a convenient leading order substitute for $\partial^{-1} Y$, and is very useful in some of our normal form computations. The function $M$ is essentially the material derivative of $\ln(1+a)$. 

We estimate these in terms of the control parameters $A$, $\As$, and $A_\frac14$ defined in \eqref{A-def}, \eqref{Asharp-def}, and \eqref{A14-def}, and in terms of the $H^s$ Sobolev norms of $\W$ and $R$.

\subsection{\texorpdfstring{$L^2$}{} and pointwise bounds}
We begin with the auxiliary variable $Y$, which
inherits its regularity from $\W$ due to \eqref{bmo-moser} and \eqref{bmo-hs}:

\begin{lemma}\label{l:Y}
The function $Y$ satisfies the $BMO$ bound 
\begin{equation}\label{est:Y}
\| |D|^\frac14 Y\|_{BMO} \lesssim_A A_{\frac14},
\end{equation}
and the $\dot H^\sigma$ bounds
\begin{equation}
\|Y\|_{ \dot H^\sigma} \lesssim_A \|\W\|_{ \dot H^\sigma},\qquad \lesssim_A \|\W\|_{\dot{W}^{\sigma, 4}}, \qquad \sigma \geq 0.
\end{equation}

\end{lemma}

We also have the following paradifferential identities relating $W, X,$ and $Y$:

\begin{lemma}\label{l:YtoW}
a) The function $Y$ satisfies
$$T_{1 + W_\alpha} Y = T_{1 - Y} W_\alpha - \Pi(Y, W_\alpha).$$

b) We have the relations
$$X_\alpha = T_{1 - Y} W_\alpha + \teal{Err_X} = T_{1 + W_\alpha} Y + \teal{Err_X},$$
where $\teal{Err_X}$ denotes a varying error term satisfying
\[
\|\teal{Err_X}\|_{\teal{\dot{W}^{\frac12, 4}}} \lesssim_{\As} A_{\frac14}, \quad \||D|^{\frac12}\teal{Err_X}\|_{BMO} \lesssim_A A_{\frac14}^2.
\]
\end{lemma}

\begin{proof}
a) It suffices to write the paraproduct expansion
\[
Y = (1 - Y)W_\alpha = T_{1 - Y} W_\alpha - T_{W_\alpha}Y - \Pi(Y, W_\alpha).
\]

\

b) For the first identity, by Lemma~\ref{l:Y}, we have
\[
-T_{Y_\alpha} W = \teal{Err_X}.
\]
Then apply (a) to obtain the second identity for $Y$, where we similarly have 
\[
\Pi(Y, W_\alpha) = \teal{Err_X}.
\]

\end{proof}

We continue with bounds for $a$.
\begin{proposition}\label{regularity for a}
 The Taylor coefficient $a$
 is nonnegative and satisfies the $BMO$ bound
\begin{equation}\label{a-bmo}
\| a\|_{BMO}    \lesssim \Vert R\Vert_{BMO^\frac12}^2,
\end{equation}
and the uniform bound
\begin{equation}\label{a-point}
\| a\|_{L^\infty} \lesssim \Vert R\Vert_{\dot B^{\frac12}_{\infty, 2}}^2.
\end{equation}
 Moreover,
 \begin{equation}\label{a-bmo+}
\Vert |D|^{\frac{1}{4}} a\Vert_{BMO}\lesssim AA_{\frac14}, \qquad \Vert |D|^{\frac{1}{2}} a\Vert_{BMO}\lesssim A_{\frac14}^2,
\end{equation}
and
\begin{equation}\label{a-Hs}
\Vert  a\Vert_{\dot H^s}\lesssim A\Vert  R \Vert_{\dot H^{s+\frac12}} \qquad s \geq 0,
\end{equation}
as well as
\begin{equation}\label{a-last}
    \teal{\||D|^{1/2}a\|_{L^4}} \lesssim \As A_\frac14, \qquad \teal{\|a\|_{\dot F^{\frac12}_{2, 1}} \lesssim \As^2.}
\end{equation}
\end{proposition}

\begin{proof}
The first part is from \cite{HIT}, so we only need to prove 
the bounds from \eqref{a-bmo+} on.
We recall that
\[
a =  i\left(\bar{P} \left[\bar{R} R_\alpha\right]- P\left[R\bar{R}_\alpha\right]\right).
\]
The bounds of \eqref{a-bmo+} are direct consequences of \eqref{bmo-bmo} and \eqref{bmo>infty}, by bounding the holomorphic and antiholomorphic parts of $a$ separately, and writing for instance
\[
P\left[R\bar{R}_\alpha\right] = T_{\bar R_\alpha} R + \Pi(\bar R_\alpha, R).
\]
We next see that \eqref{a-Hs} and the first part of \eqref{a-last} are direct consequences of the commutator estimates in Lemma~\ref{l:com}.

Finally, we consider the second part of \eqref{a-last}. If the two input frequencies of $R$ 
and the output frequency are all fixed and equal, then 
this is a straightforward $L^4 \times L^4 \to L^2$ bound.
It remains to account for the dyadic summation, which 
is now dealt with in a Triebel-Lizorkin fashion.
Precisely, for $R$ we have the Triebel-Lizorkin interpretation of the control norm $\As$ as
$\dot W^{\frac34,4} = \dot F^{\frac34}_{4,2}$ with an $\ell^2$
dyadic summation. Due to the off-diagonal decay in the 
bilinear estimate, this yields the corresponding $\ell^1$
 summability for the output.
 \end{proof}

\medskip 

Next we consider $b$, for which we have the following result.
\begin{lemma}\label{l:b}
Let $s > 0$. Then the transport coefficient $b$ satisfies 
\begin{equation}\label{b-bounds}
\||D|^s b\|_{BMO} \lesssim_A \| |D|^s R\|_{BMO}, \quad
\teal{\||D|^s b\|_{L^4} \lesssim_A \||D|^s R\|_{L^4},} \quad
\| b\|_{\dot{H}^s} \lesssim_A \|R\|_{\dot{H}^s}.
\end{equation}
In particular we have
\begin{equation}\label{b-bounds1}
\||D|^\frac12 b\|_{BMO} \lesssim_A A, \qquad
\||D|^\frac34 b\|_{BMO} \lesssim_A A_\frac14, \teal{\qquad  \||D|^\frac34 b\|_{L^4} \lesssim_A \As}.
\end{equation}
\end{lemma}
\begin{proof}
Recall that
\[
b = \Re P[R(1-\bar Y)] = R- P(R \bar Y).
\]
Hence, it remains to estimate $\partial_\alpha P( R \bar Y)$.
Consider first the $BMO$ bound.
As before, the role of $P$ is to restrict the bilinear
frequency interactions to high - low, in which case we can use
the bound \eqref{bmo-infty}, and the high-high case, where \eqref{bmo-bmo}
applies.

A direct argument, taking into account the same two cases, yields the $L^2$ \teal{and the $L^4$ bounds.}
\end{proof}

Next we consider the auxiliary expression $M$. For this we have

\begin{lemma}\label{l:M}
The function $M$ satisfies the pointwise bound 
\begin{equation}\label{M-infty}
\| M\|_{L^\infty} \lesssim_A A_{\frac14}^2,
\end{equation}
as well as the Sobolev bounds
\begin{equation}\label{M-L2}
\| M\|_{\dot H^{s-\frac12}} \lesssim_A A \|(\W,R)\|_{\dH^s}, \qquad s \geq 0.
\end{equation}
\end{lemma}
We remark that this is the only place in the paper where the Besov norm
in $A_{\frac14}$ is used. If instead we were to use only the weaker control norm \eqref{A14-def+}, then we would only get the slightly weaker bound
\begin{equation}\label{M-BMO}
\| M\|_{BMO} \lesssim_A A_{\frac14}^2.
\end{equation}

\begin{proof}
For the pointwise bound we claim that 
\begin{equation}\label{M-infty1}
\| M\|_{L^\infty} \lesssim \|R\|_{\dot B^{\frac34}_{\infty, 2}} \|Y\|_{\dot B^{\frac14}_{\infty, 2}}.
\end{equation}
To achieve this we write $M$ in two different ways,
\begin{equation}\label{M-rep}
M = \bar P[\bar R  Y_\alpha - R_\alpha \bar Y]+ P[R \bar Y_\alpha - \bar R_\alpha Y] = \partial_\alpha (\bar P [ \bar R Y] + P [R \bar Y])
- (\bar R_\alpha Y + R_\alpha \bar Y) .
\end{equation}
We apply a bilinear Littlewood-Paley decomposition and use the first expression 
above for the high-low interactions, and the second for the high-high interactions,
to write $M = M_1+M_2$ where
\[
\begin{split}
M_1 := & \ \sum_k [\bar R_k  Y_{<k,\alpha} - R_{<k,\alpha} \bar Y_k ]+ 
[R_k \bar Y_{<k,\alpha} - \bar R_{<k,\alpha} Y_k], 
\\
M_2 := & \ \sum_k
\partial_\alpha(\bar P [ \bar R_k Y_k] + P [R_k \bar Y_k])
- (\bar R_{k,\alpha} Y_k + R_{k,\alpha} \bar Y_k).
\end{split}
\]
We estimate the terms in $M_1$ separately; we show the argument for the first:
\[
\|   \sum_k \bar R_k  Y_{<k,\alpha}\|_{L^\infty} \lesssim 
\sum_{j \leq k} 2^{\frac34(j-k)} \|R_k\|_{L^{\infty}} \|Y_j\|_{L^{\infty}}
\lesssim \|R\|_{\dot B^{\frac34}_{\infty, 2}} \|Y\|_{\dot B^{\frac14}_{\infty, 2}}.
\]
For the first term in $M_2$ we note that the multiplier $\partial_\alpha P_{<k+4}  P$ 
has an $O(2^k)$  bound in $L^\infty$. Hence, we can estimate
\[
\|M_2\|_{L^\infty} \lesssim \sum_k 2^k \|R_k\|_{L^\infty} \|Y_k\|_{L^\infty} 
\lesssim \|R\|_{\dot B^{\frac34}_{\infty, 2}} \|Y\|_{\dot B^{\frac14}_{\infty, 2}}.
\]

The $\dot H^s$ bounds are a direct application of  the Coifman-Meyer bounds in \eqref{CM-BMO}, using only $A$ as in \eqref{A-def} for the lower frequency factor and the $\dH^s$ norm for the higher frequency factor. 

\end{proof}

\subsection{ Material and para-material derivative bounds.}
The material derivative $D_t = \D_t + b\D_\alpha$ plays a key role in the 
water wave equations. When applying this to products, we can use the well-known Leibniz rule. On the other hand, at the paradifferential level 
we replace the material derivative with the \emph{para-material derivative}
\[
T_{D_t} = \partial_t + T_b \partial_\alpha.
\]
Then it becomes interesting to know whether we have an analogue of Leibniz
rule for paraproducts. We consider this in the next Lemma, which we will refer to as 
the para-Leibniz rule. To state it we define four versions
of para-Leibniz errors. The first two apply in the unbalanced case, namely 
\[
E^{p}_L (u,v) = T_{D_t} T_u v - T_{T_{D_t} u} v - 
T_u T_{D_t} v,
\]
respectively 
\[
\tilde E^{p}_L (u,v) = T_{D_t} T_u v - T_{D_t u} v - T_u T_{D_t} v.
\]
Correspondingly, in the balanced case we set
\[
E^{\pi}_L (u,v) = T_{D_t} \Pi(u, v) - \Pi(T_{D_t} u, v) - \Pi(u, T_{D_t} v),
\]
respectively
\[
\tilde{E}^{\pi}_L (u,v) = T_{D_t} \Pi(u, v) - \Pi(D_t u, v) - \Pi(u, T_{D_t} v).
\]
With these notations, our para-Leibniz rule reads as follows:

\begin{lemma}\label{l:para-L}
a) For the unbalanced Leibniz rule error $E^{p}_L (u,v)$
we have the bound 
\begin{equation}\label{Leibniz-T0}
\|  E^{p}_L (u,v) \|_{\dot H^s} \lesssim  A_{\frac14} \| u\|_{BMO^{\frac14-\sigma}} \| v\|_{\dot H^{s+\sigma}}, \qquad \sigma > 0,  
\end{equation}
\teal{as well as }
\teal{\begin{equation}\label{Leibniz-T0+}
\| E^{p}_L (u,v) \|_{\dot H^s} \lesssim  \teal{A_{\frac14}} \| u\|_{\dot H^{\frac14-\sigma}} \| v\|_{BMO^{s+\sigma}}, \qquad \sigma > 0
\end{equation}}
whereas in the $\sigma = 0$ case 
 the same bounds \eqref{Leibniz-T0}, \eqref{Leibniz-T0+} 
hold for $\tilde E^{p}_L (u,v)$ with $\sigma = 0$.

b) For the balanced Leibniz rule error  $E^{\pi}_L (u,v)$
we have the bound
\begin{equation}\label{Leibniz-Pi}
\|  E^{\pi}_L (u,v) \|_{\dot H^s} \lesssim  \teal{A_{\frac14}} \| u\|_{BMO^{\frac14- \sigma}} \| v\|_{\dot H^{s+\sigma}}, \qquad 
\sigma \in \R, s \geq 0.
\end{equation}
In the $\sigma = 0$ case we also have the same bound
for $\tilde E^{\pi}_L (u,v)$.

c) We also have the $L^4$ versions of \eqref{Leibniz-T0} and 
\eqref{Leibniz-Pi}:
\begin{equation}\label{Leibniz-T04}
\|  E^{p}_L (u,v) \|_{\dot W^{s, 4}} \lesssim  A_{\frac14} \| u\|_{BMO^{\frac14-\sigma}} \| v\|_{\dot W^{s+\sigma, 4}}, \qquad \sigma > 0,  
\end{equation}
respectively
\begin{equation}\label{Leibniz-Pi4}
\|  E^{\pi}_L (u,v) \|_{\dot W^{s,4}} \lesssim  \teal{A_{\frac14}} \| u\|_{BMO^{\frac14- \sigma}} \| v\|_{\dot W^{s+\sigma,4}}, \qquad 
\sigma \in \R, s \geq 0.
\end{equation}
\end{lemma}

\begin{proof}
a) We begin with the case $\sigma > 0$, where 
\[
E^{p}_L (u,v) = T_{b} T_{\partial_\alpha u} v - T_{T_{b} \D_\alpha u} v
+ [T_b,T_u] \partial_\alpha v.
\]

In the first difference we have cancellation for the low frequencies of $b$,
and we are left with a dyadic sum of the form
\[
\sum_{j \leq l < k} \partial_\alpha u_j b_l v_k
\]
where we can rebalance derivatives. Precisely, if we bound the first 
two terms in $L^\infty$ and the third in $L^2$ then we have off-diagonal gain away from $j = l = k$. Hence it suffices to use the $\infty$ Besov norms for the first two factors.

We next consider the case $\sigma = 0$. We write 
\[
E^{p}_L (u,v) = T_{b} T_{\partial_\alpha u} v - T_{b \D_\alpha u} v
+ [T_b,T_u] \partial_\alpha v.
\]
For the first difference we use Lemma~\ref{l:para-prod}, while for the commutator term we use Lemma~\ref{l:para-com}. 

b) Here we start from
\[
E^{\pi}_L (u,v) =  T_b \Pi(\partial_\alpha u, v) - \Pi(T_b \partial_\alpha u, v)
+ T_b \Pi( u,\partial_\alpha v) - \Pi( u, T_b \partial_\alpha v).
\]
The two differences are identical, and in both cases we can use Lemma~\ref{l:para-assoc}.

For the separate estimate in the case $\sigma = 0$, we bound
\[
\Pi(T_{\D_\alpha u} b, v) + \Pi(\Pi(\D_\alpha u, b), v)
\]
using $\infty$ Besov norms and \eqref{bmo>infty}.

c) The proof of the $L^4$ type bounds is similar, using the 
corresponding $L^4$ bounds in Lemma~\ref{l:para-assoc}. 
\end{proof}

Finally, we rewrite the equations for $(\W, R)$ in a paradifferential form, and thereby capture identities and estimates for their para-material derivatives. Here, we also record the para-material derivatives of $X$ $Y$, and $a$.

\begin{lemma}\label{l:WR-source}
We have the following para-material derivatives:

a) Para-material derivative of $W$:
\begin{equation}
\begin{aligned}
T_{D_t} W = - T_{1 + W_\alpha} P[(1 - \bar Y)R] - P\Pi(W_\alpha, b).
\end{aligned}
\end{equation}

b) Para-material derivative of $(\W, R)$:
\begin{equation}\label{para(Wa)-full}
  \begin{aligned}
&T_{D_t} W_\alpha  = - T_{b_\alpha} W_\alpha - P\D_\alpha [ T_{1 + W_\alpha}[(1 - \bar Y)R] + \Pi(W_\alpha, b)], \\
&T_{D_t} R = - T_{R_\alpha} b - \Pi(R_\alpha, b) + i(1 + a)Y - ia.
  \end{aligned}  
\end{equation}

c) Leading term of the para-material derivative of $(\W, R)$:
\begin{equation}\label{para(Wa)}
\begin{aligned}
& T_{D_t}W_\alpha + T_{1+W_\alpha}T_{1- \bar Y} \D_\alpha R
 = G,
 \\
& T_{D_t} R - i T_{1 + a} Y =  K,
\end{aligned}
\end{equation} 
where the source terms $(G,K)$ in \eqref{para(Wa)} satisfy the BMO bounds
\begin{equation}\label{para(Wa)-source}
\| G \|_{BMO} + \| K \|_{BMO^\frac12} \lesssim_A A_{\frac14}^2  ,
\end{equation}
and 
\begin{equation}\label{para(Wa)-source+}
\| K \|_{BMO^\frac14} \lesssim_A A_{\frac14}.  
\end{equation}
\end{lemma}

\begin{proof}
For (a), we have
\[T_{D_t} W = - T_{1 + W_\alpha} b - \Pi(W_\alpha, b) + \bar R. \]
Observing that we may apply a projection $P$, which freely passes over low frequency paraproducts as well as $T_{D_t}$, we obtain the identity.

We proceed with the equation and estimates for $W_\alpha$. Differentiating the identity from (a),
\[
T_{D_t} W_\alpha  = - T_{b_\alpha} W_\alpha - P\D_\alpha [ T_{1 + W_\alpha}[(1 - \bar Y)R] + \Pi(W_\alpha, b)].
\]
Notice that all terms on the right hand side have a good balance of derivatives, with a derivative falling on a low frequency variable, with the exception of the case where $\D_\alpha$ falls on (the high frequency) $R$. Thus, we rewrite the right hand side as 
\[
- T_{1 + W_\alpha}T_{1 - \bar Y}R_\alpha + G,
\]
where
\[
G = T_{1 + W_\alpha}\Pi(\bar Y,R_\alpha) - T_{b_\alpha} W_\alpha - P[T_{ W_{\alpha\alpha}}[(1 - \bar Y)R] + T_{1 + W_\alpha}[\bar Y_\alpha R] + \D_\alpha\Pi(W_\alpha, b)].
\]
As observed, terms in $G$ have a good balance of derivatives and thus satisfy the desired estimate. For instance, repeated application of Proposition~\ref{p:bmo} yields
\[
\|T_{1 + W_\alpha}\Pi(\bar Y,R_\alpha)\|_{BMO} \lesssim \|W_\alpha\|_{L^\infty}\||D|^{1/4}\bar Y\|_{BMO}\||D|^{3/4}R\|_{BMO} \lesssim_A A_\frac14^2.
\]

\

We proceed with the equation for $R$. The formula in part (b) is straightforward. Then observe that the only term on the right hand side with a unfavorable balance of derivatives is $iT_{1 + a} Y$. For instance,
\[
\|i|D|^{1/2}\Pi(a, Y)\|_{BMO} \lesssim \|Y\|_{L^\infty} \||D|^{1/2}a\|_{BMO} \lesssim_A A_\frac14^2
\]
using Lemma~\ref{regularity for a} for the estimate on $a$.
The second $BMO$ bound for $K$ also follows by repeated application of Proposition~\ref{p:bmo}.
\end{proof}

Next, we record the para-material derivatives of $X$ and $Y$:

\begin{lemma}\label{l:XY-mat}
We have the following para-material derivatives:

a) Leading term of the para-material derivative of $Y$:
\begin{equation}\label{est:DtY}
T_{D_t} Y =- T_{|1 - Y|^2} R_\alpha  + G =: G' + G
\end{equation}
where the source term $G$ satisfies the pointwise bounds
\[
\| G \|_{BMO} \lesssim_A A_{\frac14}^2,
\]
and $G'$ satisfies
\[
\||D|^{-1/4} G' \|_{BMO} \lesssim_A A_{\frac14}.
\]

b) Para-material derivative of $X$:
\begin{equation}\label{est:DtX}
-T_{D_t} X = -T_{T_{1 - \bar Y} R_\alpha} X + P[(1 - \bar Y)R] + P\Pi(X_\alpha, b) + E_1,
\end{equation}
where, for $s + \frac34 \geq 0$,
\begin{equation}\label{est:DtX-err}
\|E_1\|_{\teal{\dot W^{\frac54, 4}}} \lesssim_{\As} A^2_{\frac14}, \qquad \|E_1\|_{\dot{H}^{s + \frac34}} \lesssim_{A} A_{\frac14}\|\W\|_{\dot{H}^s}.
\end{equation}

c) Leading term of the para-material derivative of $X$:
\[
T_{D_t} X + T_{1 - \bar Y}R = E_2,
\]
where
\[
\||D| E_2\|_{BMO} \lesssim_{\As} A^2_{\frac14}.
\]

d) Leading term of the para-material derivative of $X_\alpha$:
\[
T_{D_t} X_\alpha + T_{1 - \bar Y}R_\alpha = E_3, \qquad D_t X_\alpha + T_{1 - \bar Y}R_\alpha = E_3,
\]
where, for $s + \frac34 \geq 0$,
$$\| E_3\|_{BMO} \lesssim_{\As} A^2_{\frac14}, \qquad \|E_3\|_{\dot{H}^{s - \frac14}} \lesssim_{A} A_{\frac14}\|(\W, R)\|_{\dH^s}.$$

\end{lemma}

\begin{proof}
a) First we write down the identity
\[
T_{D_t} Y + T_{|1 - Y|^2} R_\alpha  = - (T_{R_\alpha} |1 - Y|^2 + \Pi(R_\alpha, |1 - Y|^2)) - (T_{Y_\alpha} b + \Pi(Y_\alpha, b)) + (1 - Y)M.
\]
Observe that the first three terms on the right hand side have a good balance of derivatives, and thus satisfy the estimate for $G$. For instance,
\[
\|\Pi(Y_\alpha, b)\|_{BMO} \lesssim \||D|^{1/4}Y\|_{BMO}\||D|^{3/4}b\|_{BMO} \lesssim_A A_{\frac14}^2.
\]
For the last term we  use \eqref{M-infty}.

 The estimate for $G'$ is straightforward, using again \eqref{bmo>infty}.
\bigskip

b) First we apply the para-Leibniz rule in Lemma~\ref{l:para-L},
\[
-T_{D_t} X = T_{D_t Y} W - T_{1 - Y} T_{D_t} W + E_1.
\]

When the material derivative falls on the low frequency $Y$, we use \eqref{ww2d-diff-Y} to write
\begin{align*}
T_{D_t Y}W = T_{-|1 - Y|^2 R_\alpha + (1 - Y)M} W = -T_{T_{1 - \bar Y} R_\alpha} X + E_1.
\end{align*}
Here for the term $T_{-|1 - Y|^2 R_\alpha} W$ we use Lemma~\ref{l:para-prod} to rewrite $W$ in terms of $X$, and Proposition~\ref{p:bmo} to peel off perturbative components where the (differentiated) $R_\alpha$ is lowest frequency. For the term 
$T_{(1 - Y)M} W$, on the other hand, we directly use the bound \eqref{M-infty}.

When the material derivative falls on the high frequency, we obtain
\begin{align*}
-T_{1 - Y}T_{D_t} W &= T_{1 - Y}(T_{1 + W_\alpha} P[(1 - \bar Y)R] + P\Pi(W_\alpha, b)).
\end{align*}
Applying Lemmas~\ref{l:para-assoc} and \ref{l:para-prod} and \ref{l:para-prod}, and Lemma~\ref{l:YtoW} on the last term, yields (b).

\

c) All terms from b) have a good balance of derivatives and may be absorbed into $E_2$, except
$$P((1 - \bar Y) R) = T_{1 - \bar Y} R - \Pi(\bar Y, R),$$
the latter of which may also be absorbed into $E_2$.

\

d) This is a straightforward consequence of (b) and (c).
\end{proof}

Lastly, we record the para-material derivative of $a$:

\begin{lemma}\label{l:a-mat}
We have the leading term of the material and para-material derivative of $a$,
\begin{equation}\label{para-aflow}
\begin{split}
D_t a = & \ - (1+a)M + E \\
T_{D_t} a = & \   - T_{1+a} M
+ E,
\end{split}
\end{equation}
where the errors $E$ satisfy
\begin{equation*}
\|E\|_{\teal{\dot{W}^{\frac14, 4} }\cap L^\infty}\lesssim_{\As} A_{\frac14}^2.
\end{equation*}
In particular we have
\begin{equation}\label{aflow}
\Vert D_t a \Vert_{L^\infty}\lesssim_{\As} A_{\frac14}^2.
\end{equation}
\end{lemma}

\begin{proof}
We reduce the first relation in \eqref{para-aflow} to the second by replacing products with para-products.  On the left, we need to estimate 
\begin{equation}\label{tri1}
\|\Pi(b,a_\alpha)\|_{\teal{\dot{W}^{\frac14, 4} }\cap L^\infty} + 
\|T_{a_\alpha} b\|_{\teal{\dot{W}^{\frac14, 4} }\cap L^\infty} \lesssim_{\As} \As A_{\frac14}^2,
\end{equation}
and on the right we need
\begin{equation}\label{tri2}
\|\Pi(M,a)\|_{\teal{\dot{W}^{\frac14, 4} }\cap L^\infty} + 
\|T_M a \|_{\teal{\dot{W}^{\frac14, 4} }\cap L^\infty} \lesssim_{\As} \As A_{\frac14}^2,
\end{equation}
where, slightly abusing notations, $\Pi$ does not include a projection $P$.

The \teal{$\dot{W}^{\frac14, 4}$} part of \eqref{tri1} follows from \eqref{a-bmo+} and 
the second part of \eqref{b-bounds} via \eqref{CM} and \eqref{CM+}.
The $L^\infty$ part of \eqref{tri1} follows from \eqref{b-bounds1} and \eqref{a-bmo} and \eqref{a-bmo+}.
Here we need to take advantage of the fact that we have off-diagonal decay, 
interpolating between the two pairs of bounds for both $a$ and $b$ in order to gain dyadic summability.
The bound \eqref{tri2} is a consequence of the $L^\infty$ bound \eqref{M-infty} for $M$ and \teal{the second part of \eqref{a-last} via the embedding $F^{1/2}_{2, 1} \subset F^0_{\infty,1}$.}

Now we turn our attention to the proof of second part of \eqref{para-aflow}. We first give a quick
proof for the \teal{$\dot W^{\frac14, 4}$} bound. 
There we use Lemmas~\ref{l:para-L}, \ref{l:WR-source}, and the para-associativity lemma, Lemma~\ref{l:para-assoc2} to expand the material derivative,
\begin{align*}
T_{D_t} P(\bar R_\alpha R) & =  P[(\D_\alpha T_{D_t} \bar R) R] + P[\bar R_\alpha T_{D_t} R] + E
\\
& = -i P[(\D_\alpha T_{1+a} \bar Y)R] + i P[\bar R_\alpha T_{1+a} Y] + E
\\
& = i T_{1+a} (- \bar Y_\alpha R + \bar R_\alpha  Y) + E
\end{align*}
with errors $E$ satisfying 
\[
\| E \|_{\teal{\dot{W}^{\frac14, 4}} } \lesssim_{\As} \As A_{\frac14}^2.
\]
We would like to repeat the same argument 
for the $L^\infty$ bound, except that we 
have the issue that $P$ is not bounded in $L^\infty$. To avoid this problem we 
use a trick we employed before, which is to take 
better advantage of the structure of both $M$
and $R$. Precisely, we can write $a$ in two ways
\[
a = -i ( P[\bar R_\alpha R] - \bar P[R_\alpha \bar R]) = - i \left( \partial_\alpha P[ \bar R R]
- R_\alpha \bar R\right),
\]
while $M$ can be represented in a matching 
fashion as
\[
M =  P[R \bar Y_\alpha - \bar R_\alpha Y] + \bar P[\bar R  Y_\alpha - R_\alpha \bar Y] = \partial_\alpha  P [R \bar Y - Y \bar R]  
+ \bar R  Y_\alpha - R_\alpha \bar Y.
\]
Here the first representation is good for unbalanced frequencies, as it selects 
the correct frequency balance. The second, on the other hand, is good for balanced interactions,
as it pairs each projection with a differentiation,
therefore improving the kernel decay.

Correspondingly we consider matching 
paradifferential decompositions of $a$ and $M$,
namely
\[
a = a^p + a^\Pi, \qquad M = M^p + M^\Pi,
\]
 where, using the two expressions above, we set
\[
a^p = -i ( T_{\bar R_\alpha} R  - T_{R_\alpha} \bar R), \qquad  a^\Pi = - i \left( \partial_\alpha P
\Pi(\bar R, R) - \Pi(R_\alpha, \bar R)\right) ,
\]
as well as the matching decomposition for $M$.
We will now prove that the desired $L^\infty$
bound holds separately for the pairs $a^p,M^p$
respectively $a^\Pi,M^\Pi$.

\bigskip
\textbf{ a) The unbalanced case, $a^p,M^p$.}
We consider the first terms in $a^p,M^p$,
for which we reproduce the computation above:
\begin{align*}
T_{D_t} T_{\bar R_\alpha} R  & = T_{\D_\alpha T_{D_t} \bar R } R  + T_{\bar R_\alpha} T_{D_t} R + E_1
\\
& = -i T_{\D_\alpha T_{1+a} \bar Y} R + i T_{\bar R_\alpha} T_{1+a} Y + E_1+E_2
\\
& = i T_{1+a} (- T_{\bar Y_\alpha} R + T_{\bar R_\alpha}  Y) + E_1+E_2+E_3.
\end{align*}
We successively consider the three errors.
$E_1$ is the para-associativity error,
which can be represented as a sum of 
two translation invariant trilinear forms
\[
E_1 =  L_{lmh}( \bar R_{\alpha\alpha}, b, R) +
L_{lmh} (b_\alpha,R_\alpha,R)
\]
with frequency localizations ordered as listed. 
Then we can apply twice the bound \eqref{bmo>infty}
to obtain
\begin{equation}\label{tri}
\|E_1\|_{BMO^{-\frac14}} \lesssim_A A^2 A_\frac14,
\qquad \|E_1\|_{BMO^{\frac14}} \lesssim_A  A_\frac14^3,
\end{equation}
after which the desired $L^\infty$ bound follows by interpolation.

The error $E_2$ can be written in the form
\[
E_2 = T_{\partial_\alpha \bar \cK} R + T_{\bar R_\alpha} \cK,
\]
where $\cK$ is from \eqref{para(Wa)}, which also satisfies a pair of unbalanced $BMO$ bounds in \eqref{para(Wa)-source} and \eqref{para(Wa)-source+}. Then a single 
application of the bound \eqref{bmo>infty}
yields the same bound as for $E_1$. 

Finally $E_3$ is also a para-associativity error,
but better than $E_1$ because it is quartic,
and it is treated in the same way.

\bigskip
\textbf{ b) The balanced case, $a^\Pi,M^\Pi$.}
We follow the same outline as before, considering the first terms in $a^{\Pi},M^{\Pi}$
and the corresponding computation:
\begin{align*}
T_{D_t} \partial_\alpha P\Pi(\bar R, R )  & =  \partial_\alpha P\Pi(T_{D_t} \bar R, R)  + \partial_\alpha P\Pi(\bar R, T_{D_t} R) + E_1
\\
& = -i \partial_\alpha P\Pi ( T_{1+a} \bar Y, R) + i \partial_\alpha P\Pi(\bar R, T_{1+a} Y) + E_1+E_2
\\
& = i T_{1+a} (- \partial_\alpha P\Pi(\bar Y, R) + \partial_\alpha P\Pi(\bar R,Y)) + E_1+E_2+E_3.
\end{align*}
As above  $E_1$ is the para-associativity error,
which can be represented as a sum of trilinear forms
\[
E_1 =  L_{lhh}( b_\alpha, \bar R_\alpha, R) +
L_{lhh}(b_\alpha,\bar R, R_\alpha),
\]
where the first frequency is smaller and the last 
two are balanced. Then we obtain the same bound 
as in \eqref{tri} by applying once  \eqref{bmo>infty} and once \eqref{bmo-bmo}.

The error $E_2$ can be written in the form
\[
E_2 = \partial_\alpha P\Pi( \bar \cK, R) + \partial_\alpha P\Pi ( \bar R, \cK),
\]
which can be treated using \eqref{para(Wa)-source}, \eqref{para(Wa)-source+} and   \eqref{bmo-bmo}.
Here more care is needed for the second term in 
$a^\Pi$, which yields an error $E_2$ of the form
\[
E_2= \Pi(\bar \cK,R_\alpha) + \Pi(\cK,\bar R_\alpha).
\]
This we can estimate by Cauchy-Schwarz,
\[
\| E_2\|_{L^\infty} \lesssim 
\| \cK \|_{\dot B^{\frac38}_{\infty,2}} 
\|R\|_{\dot B^{-\frac38}_{\infty,2}}.
\]
Then we bound the two norms on the right by interpolation between \eqref{para(Wa)-source}, \eqref{para(Wa)-source+} for the first factor on the right, respectively $A$ and $A_{\frac14}$ for the second factor.

Finally, the error $E_3$ is similar to $E_1$ but better.

\end{proof}


\section{The linearized equations}
\label{s:linearized}

The proofs of the main results in this paper take advantage of both the 
linearization of the original equation \eqref{ww2d1} for $(W,Q)$
and the linearization of the differentiated equation \eqref{ww2d-diff} for $(\W, R)$.
Both of these linearizations are used either directly or via their associated paradifferential flows. In this section we derive both 
of these linearized equations, as well as the relation between them.
We also write down the corresponding paradifferential flows.

\subsection{The linearization of the \texorpdfstring{$(W,Q)$}{} equation 
} We denote solutions for the linearized water wave equations \eqref{ww2d1} around a solution $(W, Q)$ by $(w,q)$. However, it is more convenient to
immediately switch to diagonal variables $(w,r)$, where
\[
r := q - Rw.
\]
The system of linearized equations for $(w,r)$ was computed 
in \cite{HIT}. We review here the key points.

The linearization of $R$ is
\begin{equation}\label{lin-R}
\delta R =
\dfrac{q_{\alpha}- Rw_{\alpha}}{1+\W}
= \dfrac{r_{\alpha}+ R_\alpha w}{1+\W},
\end{equation}
while the linearization of $F$ can be expressed in the form
\[
\delta F = P[ m - \bar m],
\]
where the auxiliary variable $m$ corresponds to differentiating
$F$ with respect to the holomorphic variables,
\[
m := \frac{q_\alpha - R w_\alpha}{J} + \frac{\bar R w_\alpha}{(1+\W)^2} =
 \frac{r_\alpha +R_\alpha w}{J} + \frac{\bar R w_\alpha}{(1+\W)^2}.
\]
Denoting also
\[
n := \bar R \delta R = \frac{ \bar R(r_{\alpha}+R_\alpha w)}{1+\W},
\]
the linearized water wave equations take the form
\begin{equation*}
\left\{
\begin{aligned}
&w_{t}+ F w_\alpha + (1+ \W) P[ m-\bar m] = 0 \\
&q_{t}+ F q_\alpha + Q_\alpha P[m-\bar m]  -i w +P\left[n+\bar n\right] =0.
\end{aligned}
\right.
\end{equation*}
Recalling that $b = F + \dfrac{\bar R}{1+\W}$, this becomes
\begin{equation*}
\left\{
\begin{aligned}
&(\partial_t + b \partial_\alpha) w + (1+ \W) P[ m-\bar m] =   \dfrac{\bar R w_\alpha}{1+\W} \\
& (\partial_t + b \partial_\alpha) q + Q_\alpha P[m-\bar m]  -i w +P\left[n+\bar n\right] = \dfrac{\bar R q_\alpha}{1+\W}.
\end{aligned}
\right.
\end{equation*}
Now, we can use the second equation in \eqref{ww2d-diff} to switch from $q$ to $r$ and obtain
\begin{equation*}
\left\{
\begin{aligned}
&(\partial_t + b \partial_\alpha) w + (1+ \W) P[ m-\bar m] =   \dfrac{\bar R w_\alpha}{1+\W}\\
& (\partial_t + b \partial_\alpha) r  -i \frac{1+a}{1+\W} w +P\left[n+\bar n\right] = \dfrac{\bar R (r_\alpha+ R_\alpha w)}{1+\W}.
\end{aligned}
\right.
\end{equation*}
Terms like $\bar P m$, $\bar Pn$ are lower order since the differentiated holomorphic 
variables must be at lower frequency. The same applies to their conjugates. Moving those terms to the right and taking advantage of algebraic cancellations we are left with
\begin{equation}\label{lin(wr)0}
\left\{
\begin{aligned}
& (\partial_t + b \partial_\alpha) w  +  \frac{1}{1+\bar \W} r_\alpha
+  \frac{R_{\alpha} }{1+\bar \W} w  = \mathcal{G}_0(w,r)
 \\
&(\partial_t + b \partial_\alpha)  r  - i  \frac{1+a}{1+\W} w  = \mathcal{K}_0(w,r),
\end{aligned}
\right.
\end{equation}
where
\begin{equation}\label{GK}
\begin{aligned}
\mathcal{G}_0(w,r) = \ (1+\W) (P \bar m + \bar P  m), \quad \mathcal{K}_0(w,r) =  \  \bar P n - P \bar n.
\end{aligned}
\end{equation}

We remark that while $(w,r)$ are holomorphic, it is not directly obvious
that the above evolution preserves the space of holomorphic states. To
remedy this one can also project the linearized equations onto the
space of holomorphic functions via the projection $P$.  Then we obtain
the equations
\begin{equation}\label{lin(wr)}
\left\{
\begin{aligned}
& \partial_t w + P[b \partial_\alpha w]  + P \left[ \frac{1}{1+\bar \W} r_\alpha\right]
+  P \left[ \frac{R_{\alpha} }{1+\bar \W} w \right] = P \mathcal{G}_0(  w, r)
 \\
&\partial_t r + P[b \partial_\alpha r]  - i P\left[ \frac{1+a}{1+\W} w\right]  =
 P \mathcal{K}_0( w,r).
\end{aligned}
\right.
\end{equation}
Since the original set of equations \eqref{ww2d1} is fully
holomorphic, it follows that the two sets of equations,
\eqref{lin(wr)0} and \eqref{lin(wr)}, are algebraically equivalent.
Later in this paper we will work with \eqref{lin(wr)}.

Based on the above form of the linearized equations, we can also write
the corresponding paradifferential flow:
\begin{equation}\label{paralin(wr)}
\left\{
\begin{aligned}
& T_{D_t} w + T_{1- \bar Y} \D_\alpha r
+ T_{(1 - \bar Y)R_\alpha} w = 0
 \\ 
& T_{D_t} r - i T_{1 - Y}T_{1 + a} w =  0.
\end{aligned}
\right.
\end{equation}  
Here we use the Weyl quantization for the paraproducts in order to ensure self-adjointness. We remark that the right hand side terms 
in \eqref{lin(wr)} have no paradifferential component (i.e. components such that the 
$(w,r)$ factor is the highest frequency). 

The $\dH^{1/4}$ well-posedness of the linearized equations~\eqref{lin(wr)} will be studied in Section~\ref{s:lin-est}. The $\dH^s$ well-posedness of the paradifferential equations~\eqref{paralin(wr)}, studied in Sections~\ref{s:para}, \ref{s:para-s}, will play a key role.

\subsection{The linearization of the \texorpdfstring{$(\W,R)$}{} equation 
} We denote solutions for the linearized water wave equations \eqref{ww2d-diff} around a solution $(\W, R)$ by 
$(\hatw,\hatr )$. Recalling \eqref{ww2d-diff}
\begin{equation*}
\left\{
\begin{aligned}
 &D_t \W + \frac{(1+\W) R_\alpha}{1+\bar \W}   =  (1+\W)M
\\
&D_t R = i\left(\frac{\W - a}{1+\W}\right),
\end{aligned}
\right.
\end{equation*}
we begin by computing the linearizations of $b$,
\[
\delta b = 2\Re P[(1 - \bar Y) \hat r - (1 - \bar Y)^2 R \bar{ \hat w}],
\]
and of $M$,
\[
\delta M = 2\Re P[\hat r \bar Y_\alpha -2 R(1 - \bar Y) \bar{\hat w} \bar Y_\alpha + R (1 - \bar Y)^2 \bar{\hat w}_\alpha - \bar{\hat r}_\alpha Y - \bar R_\alpha (1 - Y)^2 \hat w].
\]
Then we can write the linearized equation in the form
\begin{equation}\label{lin(hwhr)}
\left\{
\begin{aligned}
 & D_t \hat w + (1 - \bar Y)(1 + W_\alpha) \hat r_\alpha 
 + (1 - \bar Y) R_\alpha \hat w =  \  M \hat w  + (1 + W_\alpha) \delta M - W_{\alpha\alpha} \delta b 
 \\ & \qquad \qquad \qquad  
  + (1 - \bar Y)^2  (1+W_\alpha) R_\alpha \bar{\hat w}
\\
&D_t \hat r   -i(1 + a)(1 - Y)^2\hat w =  - R_\alpha \delta b 
 - (1 - Y) [\bar P(\bar{\hat r} R_\alpha) + \bar P(\bar R \hat r_\alpha) - P(\hat r \bar R_\alpha) - P(R \bar{\hat r}_\alpha)].
\end{aligned}
\right.
\end{equation}
Next we identify the corresponding paradifferential equation,
which contains only paraproduct type contributions where $\hatw$ and 
$\hatr$ terms are the highest frequency. Thus we can exclude all
$(\bar{\hatw},\bar{\hatr})$ terms, as well as all terms where 
$\hatw$ and $\hatr$ are inside a $\bar P$ projection. 
Using  the relation \eqref{M-def} to group terms, we are left with
\begin{equation} \label{paralin(hwhr)}
\left\{
\begin{aligned}
 &T_{D_t} \hat w  + T_{b_\alpha} \hat w + T_{1 - \bar Y}T_{1 + W_\alpha} \hat r_\alpha - T_{ \bar Y_\alpha}T_{1 + W_\alpha} \hat r + T_{1 - \bar Y} T_{W_{\alpha\alpha}} \hat r = 0 \\
&T_{D_t} \hat r  + T_{b_\alpha} \hat r  -iT_{(1 - Y)^2}T_{1 + a}\hat w + T_M \hat r = 0.
\end{aligned}
\right.
\end{equation}
This linearized equation will play a key role in the study of the differentiated nonlinear flow \eqref{ww2d-diff}.

\subsection{The relation between the two linearized flows}

To each solution $(w,r)$ to the linearized equation \eqref{lin(wr)} corresponds a solution $(\hatw,\hatr)$ to the linearized differentiated 
equation \eqref{lin(hwhr)}. In view of \eqref{lin-R}, the connection 
between the two is given by the relation
\begin{equation}\label{full-relation}
    (\hatw,\hatr) = \left(w_\alpha,\dfrac{r_{\alpha}+ R_\alpha w}{1+\W}\right).
\end{equation}
In Section~\ref{s:ee} we will use a paradifferential version of this relation in order to transfer the $\dH^s$ well-posedness results
from the paradifferential equation \eqref{paralin(wr)} to \eqref{paralin(hwhr)}.

\section{Bounds for the linearized equation}
\label{s:lin-est}

A key result of \cite{HIT} asserts that the linearized equation \eqref{lin(wr)} is  well-posed in $\H^0$ for as long  as $A$ and $\int B$ remain finite:

\begin{theorem}[\cite{HIT}]
The linearized equation~\eqref{lin(wr)} is well-posed in $\H^0$. Furthermore, there 
exists an energy functional $\Elin(w,r)$ so that we have

a) Norm equivalence:
\[
\Elin(w,r) \approx_A \| (w,r)\|_{\H^0}^2,
\]

b) Energy estimates:
\[
\frac{d}{dt} \Elin(w,r) \lesssim_A AB \| (w,r)\|_{\H^0}^2.
\]
\end{theorem}

We remark that this bound is scale invariant, and thus 
no improvement is possible where one lowers the regularity 
measured by either factor $A$ or $B$.
Instead, our goal in this section is to improve this result
by replacing the product $AB$ by a more balanced version,
namely $A_{1/4}^2$. We remark that $A_{1/4}^2 \lesssim AB$
but the two terms scale in the same way. The price to pay will be a shift in our function spaces, from $\H^0$ to $\dH^{\frac14}$:

\begin{theorem}\label{t:balancedenergy}
Assume $A \lesssim 1$ and $A_\frac14 \in L^2$. Then the linearized equation~\eqref{lin(wr)} is well-posed in $\dH^{\frac14}$. Furthermore, there 
exists an energy functional $\Elin^\frac14(w,r)$ so that we have

a) Norm equivalence:
\[
\Elin^\frac14(w,r) \approx_{\As} \| (w,r)\|_{\dH^{\frac14}}^2,
\]

b) Energy estimates:
\[
\frac{d}{dt} \Elin^\frac14(w,r) \lesssim_{\As} A_\frac14^2 \|(w,r)\|_{\dH^\frac14}^2.
\]
\end{theorem}

In the remainder of the section, we prove Theorem \ref{t:balancedenergy}.

\subsection{Reduction to the paradifferential equation}

 A key step in the proof of Theorem \ref{t:balancedenergy} is to reduce to the linearized paradifferential approximation \eqref{paralin(wr)} of the full linearized evolution \eqref{lin(wr)0}. Recall that this has the form
 \begin{equation*}
 \left\{
\begin{aligned}
& T_{D_t}w + T_{1- \bar Y} \D_\alpha r
+ T_{(1 - \bar Y)R_\alpha} w = 0
 \\
& T_{D_t} r - i T_{1 -Y}T_{1 + a} w =  0.
\end{aligned}
\right.
\end{equation*}  
Recall that throughout, we fix a self-adjoint quantization for $T$. For instance, we may use the Weyl quantization, or simply the average
\[
\half(T + T^*).
\]

The first step in the proof of Theorem \ref{t:balancedenergy}
is to show that the conclusion of the theorem holds for the evolution~\eqref{paralin(wr)}:

\begin{proposition} \label{p:paralin-wp}
Assume $A \lesssim 1$ and $A_\frac14 \in L^2$. Then the linearized paradifferential equation~\eqref{paralin(wr)} is well-posed in $\dH^{s}$
for every $s \in \R$. Furthermore, there  exist energy functionals $\Elin^{s.para}(w,r)$ so that we have

a) Norm equivalence:
\[
\Elin^{s,para}(w,r) \approx_{\teal{\As}} \| (w,r)\|_{\dH^{s}}^2,
\]

b) Energy estimates:
\[
\frac{d}{dt} \Elin^{s,para}(w,r) \lesssim_{\teal{\As}} A_\frac14^2 \|(w,r)\|_{\dH^s}^2.
\]
\end{proposition}
 
Here the regularity index $s$ plays no role. We will prove this result
 first in the easier special case $s = 0$ in Section~\ref{s:para},
  and then for all $s$ in Section~\ref{s:para-s}.

We remark that this result also allows us to treat 
linear perturbations of \eqref{paralin(wr)} of the form
 \begin{equation}\label{paralin(wr)FG}
 \left\{
\begin{aligned}
& T_{D_t} w + T_{1- \bar Y} \D_\alpha r
+ T_{(1 - \bar Y)R_\alpha} w = G
 \\
& T_{D_t} r - i T_{1 - Y}T_{1 + a} w = K,
\end{aligned}
\right.
\end{equation}  
provided that the source terms satisfy the bounds
\begin{equation}\label{GK-pert}
\| (G, K) \|_{\dH^{s}} \lesssim_A A_{1/4}^2 \| (w,r)\|_{\dH^{s}},
 \end{equation}
for some $s$. Then, by the variation of parameters formula,
the conclusion of the last proposition also applies 
to the equations \eqref{paralin(wr)FG} for the same $s$.

\
 
 Once we have dealt with the paradifferential equations, our next objective is to reduce the analysis of the full linearized equations to the paradifferential case. Toward this goal, we follow the clue above and rewrite the equations \eqref{lin(wr)0} as a paradifferential evolution with a source term,
\begin{equation}\label{paralin(wr)-full}
\left\{
\begin{aligned}
& T_{D_t} {w} + T_{1- \bar Y} \D_\alpha {r}
+ T_{(1 - \bar Y)R_\alpha} {w} = \Gs(w,r)
 \\
& T_{D_t} {r} - i T_{1 - Y}T_{1 + a} {w} = \Ks (w,r), 
\end{aligned}
\right.
\end{equation}
where the source terms $(\Gs,\Ks)$ have the form
\[
\Gs = P(\mathcal G_0 + \mathcal G_1), \qquad \Ks = P(\mathcal K_0 + \mathcal K_1)
\]
 with $(\mathcal{G}_0,\mathcal{K}_0)$ as in \eqref{GK} and 
 \begin{equation*}
 \begin{aligned}
& \mathcal G_1 = -(T_{w_\alpha} b + \Pi(w_\alpha, b)) + (T_{r_\alpha} \bar Y + \Pi(r_\alpha, \bar Y)) - (T_w (1 - \bar Y) R_\alpha + \Pi(w, (1 - \bar Y)R_\alpha)),
 \\
& \mathcal K_1 = -(T_{r_\alpha} b + \Pi(r_\alpha, b)) + i(T_{1 - Y}T_{w}a + T_{1 - Y}\Pi(w, a) - T_{(1 + a)w}Y - \Pi((1 + a)w, Y))
 \end{aligned}
 \end{equation*}
 denote the paradifferential truncations. Henceforth, it will be convenient to include an implicit projection $P$ in $\Pi$, so that $\Pi = P\Pi$. 
 
Unfortunately, our source terms $(\Gs,\Ks)$ do not satisfy 
the cubic bounds \eqref{GK-pert} for any $s$,
both because they contain quadratic contributions and because
 of unbalanced cubic contributions.  However, this still allows us to perturbatively discard portions of $(\Gs,\Ks)$ which do satisfy \eqref{GK-pert}. 
 
 To deal with the unfavourable part of $(\Gs,\Ks)$ we will rely on
 a more accurate paradifferential version of normal form analysis.
We also optimize the choice of $s$.  Our result reads as follows:
 
 \begin{proposition}\label{paralinearization}
Given $(w, r)$ satisfying \eqref{paralin(wr)-full}, there exist 
modified normal form linear variables $(\nfw, \nfr)$ satisfying \eqref{paralin(wr)FG} such that we have

(i) Invertibility: 
\begin{align*}
\|(\nfw, \nfr) - (w, r) \|_{\dH^\frac14} \lesssim_{A}  \teal{\As} \|(w,r)\|_{\dH^\frac14},
\end{align*}

(ii) Perturbative source term:
\begin{equation}\label{nf-err-bd-GK}
\|(G, K)\|_{\dH^{\frac14}} \lesssim_{\As} A_{\frac14}^2 \|(w, r)\|_{\dH^{\frac14}}.
\end{equation}
\end{proposition}

 To facilitate the normal form analysis in the proof of the above proposition, we 
 separately state a pair of  weaker quadratic bounds for $(\cG_0,\cK_0)$ and most of $(\cG_1,\cK_1)$, which will suffice in order to evaluate their impact as source terms for the material derivative of $(w, r)$. In the case of $(\cG_0,\cK_0)$ it will be useful to simultaneously extract the perturbative part of $(\cG_0,\cK_0)$, in preparation for the normal form analysis on the unfavorable part.

It will be convenient to define the unfavorable part in terms of the following variant of $w$ with a low frequency coefficient,
\[
x = T_{1 - Y} w.
\]
 
\begin{lemma}\label{l:GK} 
The source terms $(\cG_0,\cK_0)$ satisfy the bound
\begin{equation}\label{nf-err-bd2-GK}
\| ( \cG_0, \cK_0)\|_{\H^{0}} \lesssim_{\As} A_{\frac14} \| (w,r)\|_{\dH^{\frac14}}.
\end{equation}

Furthermore, they  have the representation
\[
P\cG_0 = \cG_{0, 0} + G, \qquad P\cK_0 = \cK_{0, 0} + K,
\]
where $(G,K)$ satisfy the bounds \eqref{nf-err-bd-GK}, whereas $(\cG_{0, 0},\cK_{0, 0})$ are given by 
\begin{align*}
\cG_{0, 0}(w, r) &= - P[T_{1 + W_\alpha}(T_{1 - \bar Y}\bar r_\alpha + T_{\bar x}\bar R_\alpha) Y] + P[T_{1 + W_\alpha}T_{1 - \bar Y} \bar x_\alpha R], \\
\cK_{0, 0}(w, r) &= - P[(T_{1 - \bar Y}\bar r_\alpha + T_{\bar x}\bar R_\alpha) R].
\end{align*}
 \end{lemma}
 We note that $(G,K)$ also satisfy \eqref{nf-err-bd2-GK}, but we will not prove it as this is not needed in the sequel.
 
We now turn our attention to $(\cG_1,\cK_1)$, where we have
 \begin{lemma}\label{l:GK1} 
The source terms $(\cG_1,\cK_1)$ have the representation
\begin{equation}
\begin{split}
P\cG_1 = & \ - T_{w} T_{1-\bar Y} R_\alpha + G, \\
P\cK_1 = & \ - i T_w T_{1+a} Y + K,
\end{split}    
\end{equation}
where $(G,K)$ satisfy the quadratic bound \eqref{nf-err-bd2-GK}.
\end{lemma}

Using these lemmas, we will obtain the following formula for the para-material derivatives of $(w, r)$:
\begin{equation}\label{wr-mat}
\left\{
\begin{aligned}
& T_{D_t} {w} = -T_{1- \bar Y} (r_\alpha + T_{w}R_\alpha) + G,
 \\
& T_{D_t} {r} = iT_{1 + a}( T_{1 - Y} {w} - T_w Y) + K,
\end{aligned}
\right.
\end{equation}
where 
$(G, K)$ satisfy the quadratic bound \eqref{nf-err-bd2-GK}. We will also obtain a similar formula for $x$.

 We will first prove Lemmas~\ref{l:GK},~\ref{l:GK1}, and the para-material derivative formula \eqref{wr-mat} in Section~\ref{s:GK}. Next, Proposition~\ref{paralinearization} is proved in two steps in Sections~\ref{s:nf-1} and  \ref{s:nf-0}, where we construct normal form corrections for the terms $(\mathcal G_1,\mathcal K_1)$ and $(\mathcal G_0,\mathcal K_0)$ respectively.

\subsection{\texorpdfstring{$\H^0$}{} bounds for the paradifferential equation} 
\label{s:para}

In this section we consider the $\H^0$ well-posedness of \eqref{paralin(wr)}, which we reproduce here for convenience with general source term  $(G, K)$:
\begin{equation}\label{lin-para}
\left\{
\begin{aligned}
& T_{D_t} w  + T_{1 - \bar Y} r_\alpha
+ T_{(1 - \bar Y)R_{\alpha}} w  = G,
 \\
&T_{D_t} r  - i T_{1 - Y} T_{1 + a}  w  = K.
\end{aligned}
\right.
\end{equation}

In the study of \eqref{lin(wr)} in \cite{HIT}, the following associated positive definite linear energy was used:
\[
\Elind(w,r) = \int_{\R} (1+a) |w|^2 + \Im ( r \bar  r_\alpha) \,
d\alpha .
\]
Here we slightly alter this by considering its paradifferential version,
\[
\Elin^{0,para}(w,r) = \int_{\R} T_{1 + a} w \cdot \bar w + \Im ( r \bar  r_\alpha) \,
d\alpha .
\]

With this notation, we will prove that

\begin{proposition}\label{prop:plinwp}
The linear equation \eqref{lin-para} is well posed in $\H^0$, and the following estimate holds:
\begin{equation}
\frac{d}{dt}\Elin^{0,para}(w,r) = 2 \Re \int_{\R} T_{1 + a}\bar w \, G - i \bar r_\alpha \, K \ d\alpha + O_\As(A_{\frac14}^2) \Elind(w, r).
\end{equation}
\end{proposition}

\begin{proof}
A direct computation yields
\begin{align*}
\frac{d}{dt} \int T_{1 + a}\bar{w} \cdot w + \Im(r \bar r_\alpha) \, d\alpha &= 2 \Re \int T_{1 + a}\bar w \cdot w_t \, d\alpha + 2 \Im \int \bar r_\alpha r_t \, d\alpha \\
&\quad + \int T_{a_t} \bar w \cdot w \, d\alpha.
\end{align*}
Next, we augment the time derivatives $w_t$ and $r_t$ on the right hand side into material derivatives. Integrating by parts, write
\begin{align*}
2 \Re \int T_{1 + a} \bar w \cdot T_b\D_\alpha w \, d\alpha &= -\int T_{((1 + a)b)_\alpha} \bar w \cdot w \, d\alpha\\
&\quad -\int (T_{b_\alpha} T_a  + T_b T_{a_\alpha} - T_{(ab)_\alpha})\bar w \cdot w \, d\alpha - \int (T_b T_a - T_a T_b)\bar w_\alpha \cdot w \, d\alpha.
\end{align*}
and similarly,
\begin{align*}
2 \Im \int \bar r_\alpha \cdot T_b\D_\alpha r \, d\alpha &= 0.
\end{align*}

Adding the above three identities and using the equations \eqref{lin-para}, the quadratic
$\Re (w \bar r_\alpha)$ term cancels:
\begin{align*}
\frac{d}{dt} \Elin^{0,para}(w,r) &= 2 \Re \int (1 + T_a)\bar w \cdot G - i \bar r_\alpha K \, d\alpha\\
&\quad - 2 \Re \int (1 + T_a) \bar w \cdot T_{\frac{R_\alpha}{1 + \W}} w \, d\alpha  \\
&\quad + \int (T_{b_\alpha} T_a  + T_b T_{a_\alpha} - T_{(ab)_\alpha})\bar w \cdot w \, d\alpha + \int (T_b T_a - T_a T_b)\bar w_\alpha \cdot w \, d\alpha \\
&\quad + \int T_{a_t + b a_\alpha + (1 + a)b_\alpha} \bar w \cdot w \, d \alpha.
\end{align*}
We can write two of the error terms (precisely, the second and fourth lines on the right hand side) in terms of the auxiliary function $M$ defined in \eqref{M-def}, modulo another paraproduct error term:
\begin{align*}
\frac{d}{dt} \Elin^{0,para}(w,r) &= 2 \Re \int (1 + T_a)\bar w \cdot G - i \bar r_\alpha K \, d\alpha \, d\alpha + \teal{Err_M} + \teal{Err_{com}}
\end{align*}
where
\begin{align*}
\teal{Err_M} &= \int T_{D_t a} \bar w \cdot w - (1 + T_a) \bar w \cdot T_M w \, d \alpha, \\
Err_{com} &= \int (T_{b_\alpha} T_a  + T_b T_{a_\alpha} - T_{(ab)_\alpha})\bar w \cdot w \, d\alpha + \int (T_b T_a - T_a T_b)\bar w_\alpha \cdot w \, d\alpha \\
&\quad + \int (T_{ab_\alpha} - T_{ a} T_{b_\alpha}) \bar w \cdot w \, d\alpha.
\end{align*}
For the first term in $\teal{Err_M}$, by Lemma~\ref{l:a-mat}, we have 
\[
\|D_ta\|_{L^\infty} \lesssim_{\As} A_{\frac14}^2.
\]
For the second term we combine the pointwise bounds $\vert a\vert \lesssim A^2$ in Proposition~\ref{regularity for a} together with 
\[
\|M\|_{L^\infty} \lesssim_A A_{\frac14}^2
\]
in Lemma~\ref{l:M}.

All terms in $Err_{com}$ may be estimated using the product and commutator estimates, Lemmas \ref{l:para-com} and \ref{l:para-prod}.
\end{proof}

\subsection{\texorpdfstring{$\dH^s$}{} bounds for the paradifferential equation}
\label{s:para-s}

Our aim here is to prove $\dH^s$ well-posedness for 
the linearized paradifferential equation \eqref{paralin(wr)},
and to identify a suitable energy functional.
Our main result is as follows:

\begin{proposition}\label{prop:cubiccommutator}
Let $s \in \R$. Given $(w, r)$ solving the homogeneous linearized equation \eqref{paralin(wr)}, there exist
linearized, normalized variables $(\tw^s, \tr^s)$ solving
\begin{equation}\label{eq:para-linearized7}
\left\{
\begin{aligned}
& T_{D_t} \tw^s + T_{1- \bar Y} \D_\alpha \tr^s
+ T_{(1 - \bar Y)R_\alpha} \tw^s = G_s \\
& T_{D_t} \tr^s - i T_{1 - Y}T_{1 + a} \tw^s =  K_s,
\end{aligned}
\right.
\end{equation}
and such that
\[
\|(\tw^s, \tr^s) - D^s(w, r) \|_{\H^{0}} \lesssim_A A \|(w, r)\|_{\dH^s},
\]
and
\begin{align}\label{GsKs}
\|(G_s, K_s)\|_{\H^0} &\lesssim_A A_{\frac14}^2 
\|(w, r)\|_{\dH^{s}}.
\end{align}
\end{proposition}

The $\dH^s$ well-posedness for \eqref{lin-para} follows by a fixed point argument by applying 
Proposition~\ref{prop:plinwp}  to $(\tw^s,\tr^s)$ given by the above 
proposition. The $\dH^s$ energy in Proposition~\ref{p:paralin-wp}
will then be given by
\[
\Elin^{s,para}(w,r) = \Elin^{0,para}(\tw^s,\tr^s).
\]

\begin{proof}
The natural approach to this problem is to consider 
the equations for the variables
\[
(w^s,r^s) = (D^s w,D^s r).
\]
This can be written in the form
\begin{equation}\label{eq:para-linearized7a}
\left\{
\begin{aligned}
 & T_{D_t} w^s + T_{1- \bar Y} \D_\alpha r^s
+ T_{(1 - \bar Y)R_\alpha} w^s = \cG^s_0(w^s,r^s) \\
 & T_{D_t} r^s - i T_{1 - Y}T_{1 + a} w^s = \cK^s_0(w^s,r^s),
\end{aligned}
\right.
\end{equation}
where the source terms $(\cG^s_0,\cK^s_0)$ have the form 
\begin{equation}\label{eq:para-s-source}
\begin{aligned}
 \cG^s_0(w^s,r^s) = & L(b_\alpha,w^s)  - L(\bar Y_\alpha,r^s)
 + L([(1-\bar Y)R_\alpha]_\alpha, \partial^{-1}_\alpha w^s),  \\
 \cK^s_0(w^s,r^s) = & \ L(b_\alpha,r^s)+ i L(Y_\alpha, \partial^{-1}_\alpha T_{1 + a} w^s ).
\end{aligned}
\end{equation}
Here $L$ denotes the order zero paradifferential bilinear form
\[
L(f_\alpha,u) = - [D^s,T_f] \partial_\alpha D^{-s} u
\]
which has leading term
\[
L(g,u) = - s T_g u.
\]
We need to study the evolution \eqref{eq:para-linearized7a} in $\H^0$.
The difficulty is that the source terms $(\cG^s_0,\cK^s_0)$ cannot be treated perturbatively, both because they have 
a quadratic component and because their cubic and higher part is unbalanced\footnote{This would be less of an issue if all we wanted was an $AB$ bound.}.  

 To facilitate the next discussion we note that we can organize the source terms in $(\cG,\cK)$ by order, i.e. by the relative Sobolev regularity of the leading factors $(w,r)$. For instance in \eqref{eq:para-s-source} the first terms on the right are leading order $0$, the second are  order $-\frac12$, and the third are order $-1$. The lower the order the better; this is most readily 
 seen in the normal form analysis. 

The natural attempt would be to try to apply a normal form transformation in order to eliminate all the bad terms, and replace them  by perturbative terms. Such a strategy would work for terms of order  $\leq -\frac12$. Precisely, all nonperturbative terms of order $\leq -\frac12$ can be viewed as quadratic expressions, possibly  with a lower frequency coefficient. Then the normal form computation
works directly, and  applies to the quadratic part, while simply retaining the low frequency coefficient.
 
On the other hand, for the zero order terms the above procedure 
no longer works.
A normal form would indeed replace the quadratic terms by cubic ones, but those 
will still be partly unbalanced.  One could reiterate and eventually 
obtain an infinite series of corrections. Instead of this, we will borrow 
an idea from \cite{HIT}, where it was observed that there is an appropriate exponential conjugation (which can be thought of as the sum of the 
infinite series of corrections alluded to above) which does the job
modulo ``lower order terms", which can be dealt with directly by normal forms.

The conjugation introduced in \cite{HIT}, only for integer $s$, was 
purely algebraic, by a factor
\[
 \Phi = \phi(W_\alpha) = J^{-s}.
\]
Here we will also allow for a noninteger $s$, and use the conjugation
in a paradifferential fashion. We denote the conjugated variables by
\[
(\tw^s_1,\tr^s_1) = (T_\Phi w^s, T_{\Phi} r^s)
\]
and compute the equations for $(\tw^s_1,\tr^s_1)$. Along the way  we will cheerfully discard good cubic terms
which can be placed in $(G_s,K_s)$.  By good terms, 
we mean terms which satisfy the good bounds
\eqref{GsKs}. From here on, we shall simply 
denote by $(G_s,K_s)$ any such expressions.

We first use the para-Leibniz rule in Lemma~\ref{l:para-L} 
to write
\[
T_{D_t} \tw^s_1 = T_\Phi T_{D_t} w^s + T_{D_t \Phi} w^s + G_s,
\]
where we have used the bound
\[
\| D^\frac14 \Phi\|_{BMO} \lesssim_A A_\frac14.
\]

We compute
\[
D_t \Phi = \partial_{W_\alpha} \phi\cdot D_t W_\alpha + \partial_{\bar W_\alpha} \phi\cdot D_t \bar W_\alpha
= s \Phi [ (1-\bar Y) R_\alpha + (1-Y) \bar R_\alpha - 2M] = s \Phi (b_\alpha - M).
\]
We may discard the $M$ term by Lemma~\ref{l:M}, arriving at
\[
T_{D_t} \tw^s_1 = T_\Phi T_{D_t} w^s + s  T_{\Phi b_\alpha} w^s + G_s.
\]
Finally, applying Lemma~\ref{l:para-prod} for the $\Phi b_\alpha$ paraproduct  we conclude that
\[
T_{D_t} \tw^s_1 = T_\Phi T_{D_t} w^s + s  T_{b_\alpha} \tw^s_1 + G_s.
\]
Similarly we have 
\[
T_{D_t} \tr^s_1 = T_\Phi T_{D_t} r^s + s  T_{b_\alpha} \tr^s_1 + K_s.
\]

For the other terms on the left hand side of \eqref{eq:para-linearized7a}, using again Lemmas~\ref{l:para-com}, \ref{l:para-prod},
we have
\begin{align*}
T_{1-\bar Y} \D_\alpha \tr^s_1 &= T_\Phi T_{1-\bar Y} \D_\alpha r^s -s T_{1-\bar Y} T_{(1 - Y)W_{\alpha\alpha} + (1 - \bar Y)\bar W_{\alpha\alpha}} \tr^s_1 + G_s \\
& =  T_\Phi T_{1-\bar Y} \D_\alpha r^s -s (T_{1 - \bar Y}T_{(1 + W_\alpha)Y_\alpha} + T_{\bar Y_{\alpha}}) \tr^s_1 + G_s,
\end{align*}
and 
$$T_{(1 - \bar Y)R_\alpha} \tw^s_1 = T_{\Phi} T_{(1 - \bar Y)R_\alpha} w^s + G_s$$
in the first equation, and
\[
T_{1 - Y}T_{1 + a} \tw^s_1 = T_{\Phi}T_{1 - Y}T_{1 + a} w^s + K_s
\]
in the second.
Finally, by Lemma~\ref{l:para-com},  $T_\Phi$ commutes with the coefficients on the right side of \eqref{eq:para-linearized7a} modulo good errors, so for $(\tw^s_1,\tr^s_1)$ we obtain the system
\begin{equation}\label{eq:para-linearized7b}
\left\{
\begin{aligned}
& T_{D_t} \tw^s_1 + T_{1- \bar Y} \D_\alpha \tr^s_1
+ T_{(1 - \bar Y)R_\alpha} \tw^s_1 = \cG^s_1  + G_s \\
& T_{D_t} \tr^s_1 - iT_{1 - Y}T_{1 + a} \tw^s_1 = \cK^s_1  + K_s
\end{aligned}
\right.
\end{equation}
with the nonperturbative source terms
\[
\begin{split}
\cG^s_1 = & \ L(b_\alpha,\tw^s_1) + s T_{b_\alpha} \tw^s_1   - L(\bar Y_\alpha, \tr^s_1)
- s T_{\bar Y_\alpha} \tr^s_1+ L([(1-\bar Y)R_\alpha]_\alpha, \partial^{-1}_\alpha \tw^s_1)\\
&\quad - s T_{1 - \bar Y}T_{(1 + W_\alpha) Y_{\alpha}} \tr^s_1 ,\\
\cK^s_1  = & \  L(b_\alpha,\tr^s_1) + s T_{b_\alpha} \tr^s_1  + i L(Y_\alpha, \partial^{-1}_\alpha T_{1 + a} \tw^s_1 ).
\end{split}
\]
The main point here is that there is cancellation in the first two leading terms in both $G_1$ and $K_1$,
leaving us with only terms of order $-\half$ or lower. Precisely, we may write
\[
\begin{split}
\cG^s_1 = & \ L_1(b_{\alpha\alpha}, \D_\alpha^{-1} \tw^s_1) - L_1(\bar Y_{\alpha\alpha}, \D_\alpha^{-1} \tr^s_1)+ L([(1-\bar Y)R_\alpha]_\alpha, \partial^{-1}_\alpha \tw^s_1) \\
&\quad  - s T_{1 - \bar Y}T_{(1 + W_\alpha) Y_{\alpha}} \tr^s_1, \\
\cK^s_1  = & \  L_1(b_{\alpha\alpha}, \D_\alpha^{-1}\tr^s_1) + i L(Y_\alpha, \partial^{-1}_\alpha T_{1 + a}\tw^s_1 ),
\end{split}
\]
where $L_1$ denotes 
\[
L_1(f_{\alpha\alpha}, \D_\alpha^{-1} u) = L(f_\alpha,u) + s T_{f_\alpha}u
\]
which, like $L$, has leading term
\[
L_1(g, u) = -s(s-1)T_g u.
\]
We also observe that $(\cG_1^s, \cK_1^s)$ 
satisfy bilinear estimates of the form\footnote{
Similar bounds are also satisfied by $(G_s,K_s)$, but will not be needed in the sequel because we have \eqref{GsKs}.}
\begin{equation}\label{eq:para-linearized7best}
\||D|^{-1/4}(\cG_1^s, \cK_1^s)\|_{\H^0} 
\lesssim_A A_{\frac14}
\|(w, r)\|_{\dH^{s}}
\end{equation}
which will suffice to control the secondary impact
 of the source terms in the normal form analysis below.

\
We proceed with the normal form corrections. We recall that their 
purpose is two-fold:
\begin{enumerate}[label= (\roman*)]
\item to turn bilinear source terms into trilinear ones,
and
\item to replace trilinear and higher unbalanced interactions
into balanced ones. 
\end{enumerate}
To achieve the first purpose it suffices to 
consider classical normal forms for bilinear interactions. However,
for the second goal more care is needed. 

Precisely, we first identify the leading order bilinear interaction, and think of the remaining factors (if any) as low frequency paradifferential coefficients.
For the leading order part we compute its associated bilinear normal form correction. But after this we follow up with a second step, which is to add low frequency coefficients in order to fully match the low frequency coefficients in the source terms.

One favourable feature in this process is that the secondary outcome
of the source terms $(\cG^s_1 ,\cK^s_1)$ (as they appear in the material derivative of the normal form corrections) is perturbative. This allows us to divide and conquer,
i.e. to compute the normal form corrections separately for different parts of the source terms.

We begin with a correction consisting of $L_1$ bilinear forms, which will have the twofold effect on $(\cG^s_1 ,\cK^s_1)$ of canceling the $L_1$ terms, and exchanging the $L$ terms with their leading terms. This is easily computed to have the form
\[
\left\{
\begin{aligned}
\tw^s_2 &= \D_\alpha T_{1 - Y} L_1(W_{\alpha\alpha} \D_\alpha^{-2} \tw^s_1) + T_{1 - \bar Y} L_1(\bar W_{\alpha\alpha}, \D_\alpha^{-1} \tw^s_1) =: \tw^s_{2,1}+\tw^s_{2,2} \\
\tr^s_2 &= T_{1 - Y} L_1(W_{\alpha\alpha}, \D_\alpha^{-1} \tr^s_1)+ T_{1 - \bar Y} L_1(\bar W_{\alpha\alpha}, \D_\alpha^{-1} \tr^s_1) =: \tr^s_{2,1}+\tr^s_{2,2},
\end{aligned}
\right.
\]
where the $L_1$ part is the leading bilinear expression and the preceding paraproducts should be thought of as coefficients. We claim that this correction has the following effect when inserted in \eqref{eq:para-linearized7a}:
\begin{equation}\label{eq:para-linearized7c}
\left\{
\begin{aligned}
& T_{D_t} \tw^s_2 + T_{1- \bar Y} \D_\alpha \tr^s_2
+ T_{(1 - \bar Y)R_\alpha} \tw^s_2 = \cG^s_2 - \cG^s_1 + G_s \\
& T_{D_t} \tr^s_2 - iT_{1 - Y}T_{1 + a}\tw^s_2 = \cK^s_2 - \cK^s_1 + K_s,
\end{aligned}
\right.
\end{equation}
where
\[
\begin{split}
\cG^s_2 = &\  - s T_{1 - \bar Y}T_{(1 + W_\alpha) Y_{\alpha}} \tr^s_1
-sT_{(1-\bar Y)R_{\alpha\alpha}} \partial^{-1}_\alpha \tw^s_1,
\\
\cK^s_2 = &\ - is T_{Y_\alpha} \partial^{-1}_\alpha T_{1 + a}\tw^s_1.
\end{split}
\]

To prove this we begin with the first equation of the system, and only the contributions of the bilinear holomorphic correction $\tw^s_{2,1}$. When applying $T_{D_t}$, we commute it with differentiation while using Lemma~\ref{l:para-L} to distribute it to (para)products. The commutators with $\partial_\alpha$ involve $T_{b_\alpha}$, and their contributions can all be placed in $G_s$ using \eqref{b-bounds}. Then we have
\begin{equation*}
\begin{aligned}
T_{D_t} \tw^{s}_{2,1} & =  -\D_\alpha T_{ T_{D_t} Y} L_1(W_{\alpha\alpha}, \D_\alpha^{-2} \tw^s_1) 
+ \D_\alpha T_{1 - Y} L_1(\partial_\alpha T_{D_t}W_\alpha, \D_\alpha^{-2} \tw^s_1) 
\\ & \quad \, + \D_\alpha T_{1 - Y} L_1(W_{\alpha\alpha}, \D_\alpha^{-2} T_{D_t} \tw^s_1)  + G_s.
\end{aligned}
\end{equation*}
The first term on the right is perturbative in view of \eqref{est:DtY}.
For $T_{D_t} W_\alpha$ we use the equation \eqref{para(Wa)} along with the source term bound 
\eqref{para(Wa)-source}, whereas for the source term arising 
from $T_{D_t} \tw^s_1$  we use \eqref{eq:para-linearized7b} with the source term bounds \eqref{eq:para-linearized7best} and \eqref{GsKs}. These bounds allow us to  estimate the corresponding $L_1$ contributions by taking advantage of the fact that $L_1$ has a paraproduct structure. Then we arrive at
\begin{equation*}
\begin{aligned}
T_{D_t} \tw^{s}_{2,1}
& = - \D_\alpha T_{1 - Y} L_1(T_{1 + W_\alpha} T_{1 - \bar Y} R_{\alpha\alpha}, \D_\alpha^{-2} \tw^s_1)
 - \D_\alpha T_{1 - Y} L_1(W_{\alpha\alpha}, \D_\alpha^{-2} T_{1 - \bar Y} \D_\alpha \tr^s_1)  + G_s\\
&= - L_1(T_{1 - \bar Y}R_{\alpha\alpha\alpha}, \D_\alpha^{-2} \tw^s_1) - L_1(T_{1 - \bar Y}R_{\alpha\alpha}, \D_\alpha^{-1} \tw^s_1)\\
&\quad - \D_\alpha T_{1 - Y} L_1(W_{\alpha\alpha}, T_{1 - \bar Y} \D_\alpha^{-1} \tr^s_1) + G_s.
\end{aligned}
\end{equation*}
For the remaining holomorphic terms on the left hand side of the first equation in \eqref{eq:para-linearized7b} we similarly have
\begin{equation*}
\left\{
\begin{aligned}
&T_{1 - \bar Y}\D_\alpha  \tr^s_{2,1} = \D_\alpha T_{1 - Y} L_1(W_{\alpha\alpha}, T_{1 - \bar Y}\D_\alpha^{-1} \tr^s_1) + G_s \\
&T_{(1 - \bar Y)R_\alpha} \tw^s_{2,1} = G_s. 
\end{aligned}
\right.
\end{equation*}
Throughout the above computation we have made repeated use of two facts in the course of collecting error terms into $G_s$. First, we repeatedly apply the product and commutator estimates, Lemmas \ref{l:para-com} and \ref{l:para-prod}. Second, we observe that when derivatives fall on the coefficients, we have a good balance of derivatives. More generally, any cubic paradifferential terms where the lowest frequency variable is differentiated to an order higher (in fact strictly higher due to integer mismatch) than $A_\frac14$ may be absorbed into $G_s$ (this property is combined with the low order $\leq -\half$ of the high frequency variable to pigeonhole sufficient derivatives onto the middle frequency variable; we then obtain two contributions of $A_\frac14$ from the low and middle frequencies, possibly after redistributing derivatives from the lowest frequency variable upward).

We similarly consider the contribution of the bilinear corrections with one antiholomorphic input (staying with the first equation of the system):
\begin{equation*}
\begin{aligned}
T_{D_t} \tw^s_{2,2}
&= - L_1(T_{1 - Y}\bar R_{\alpha\alpha}, \D_\alpha^{-1} \tw^s_1) - L_1(\bar Y_{\alpha}, \tr^s_1)  + G_s,\\
T_{1 - \bar Y}\D_\alpha \tr^s_{2,2} &= L_1(\bar Y_{\alpha\alpha}, \D_\alpha^{-1} \tr^s_1) + L_1(\bar Y_{\alpha}, \tr^s_1) + G_s, \\
T_{(1 - \bar Y)R_\alpha} \tw^s_{2,2} &= G_s. 
\end{aligned}
\end{equation*}
Here we used the same two facts as above in collecting $G_s$. In addition, we also used Lemma \ref{l:YtoW} to exchange instances of $W_\alpha$ with $Y$.

After collecting these contributions, observing cancellations, and extracting perturbative paradifferential terms from $\cG^s_1$ into $G_s$ using the two facts as discussed above, we see that all $L_1$ terms of $\cG^s_1$ cancel (note that we need to expand $b$ into its real and imaginary parts for both an $R$ and $\bar R$ component) and all $L$ terms are exchanged with their leading order terms, thus proving our claim  for the first equation in \eqref{eq:para-linearized7c}.

We repeat this computation for the second equation of the system \eqref{eq:para-linearized7c} (here it is convenient to compute the entire correction at once): 
\begin{equation*}
\begin{aligned}
T_{D_t} \tr^s_{2} &= L_1(2\Re (T_{1 - Y}\partial_\alpha T_{D_t} W_\alpha), \D_\alpha^{-1} \tr^s_1) 
 + L_1(2\Re (T_{1 - Y} W_{\alpha\alpha}), \D_\alpha^{-1} T_{D_t}\tr^s_1) + K_s \\ 
&= -L_1(2\Re (T_{1 - Y}T_{1 + W_\alpha}T_{1 - \bar Y}R_{\alpha\alpha}), \D_\alpha^{-1} \tr^s_1)  + L_1(2\Re (T_{1 - Y} W_{\alpha\alpha}), \D_\alpha^{-1} iT_{1 - Y}T_{1 + a}\tw^s_1) + K_s \\
&= -L_1(b_{\alpha\alpha}, \D_\alpha^{-1} \tr^s_1) 
 +  iT_{1 - Y}L_1(2\Re( T_{1 - Y} W_{\alpha\alpha}), \D_\alpha^{-1}T_{1 + a}\tw^s_1) + K_s,
 \end{aligned}
\end{equation*}
 as well as
 \begin{equation*}
\begin{aligned}
-iT_{1 - Y}T_{1 + a} \tw^s_{1,2} & 
= -iT_{1-Y} \partial_\alpha L_1( T_{1-Y} W_{\alpha\alpha}, \D_\alpha^{-2} T_{1+a} \tw^s_1) + K_s,
\\
-iT_{1 - Y}T_{1 + a} \tw^s_{2,2} & 
= -iT_{1-Y}  L_1( T_{1-\bar Y} \bar W_{\alpha\alpha}, \D_\alpha^{-1} T_{1+a} \tw^s_1) + K_s.
\end{aligned}
\end{equation*}
Thus, the $L_1$ and $L$ terms in $\cK^s_1$ are canceled or exchanged for their leading terms, respectively. This proves our claim  for the second equation in \eqref{eq:para-linearized7c}.



We are now left to deal with the source terms $(\cG^s_2,\cK^s_2)$.
We address these remaining source terms with a second normal form correction. Write
\begin{align*}
\tw^s_3 &= - sT_{(1 + W_\alpha)Y_\alpha} \D^{-1}_\alpha \tw^s_1, \\
\tr^s_3 &= 0.
\end{align*}
This is easily checked to cancel $(\cG^s_2,\cK^s_2)$, when inserted into \eqref{eq:para-linearized7a}.
Then setting
\begin{align*}
\tw^s &= \tw^s_1 + \tw^s_2 + \tw^s_3, \\
\tr^s &= \tr^s_1 + \tr^s_2 + \tr^s_3
\end{align*}
concludes the proof of Proposition \ref{prop:cubiccommutator}.
\end{proof}


 \subsection{The quadratic bound for the paradifferential source term}
 \label{s:GK}
 Here we prove Lemmas~\ref{l:GK}, \ref{l:GK1}. In addition, in Lemmas~\ref{l:Dtwr},~\ref{l:Dtx} we derive 
 formulas for the para-material derivatives of $w$, $r$ and $x$.
 
 \begin{proof}[Proof of Lemma~\ref{l:GK}]
 We recall our objective, which is to isolate the nonperturbative part of $(\cG_0, \cK_0)$ with bilinear estimates, provided in Lemma~\ref{l:GK}. For convenience, we recall that
\begin{equation*}
\begin{aligned}
\mathcal{G}_0(w, r) &= (1 + W_\alpha)(P\bar m + \bar P m), \\
\mathcal{K}_0(w, r) &= \bar P n - P \bar n,
\end{aligned}
\end{equation*}
where
$$m = |1 - Y|^2(r_\alpha + R_\alpha w) + (1 - Y)^2 \bar R w_\alpha, \quad n = (1 - Y)\bar R (r_\alpha + R_\alpha w).$$

We consider $\mathcal{G}_0$ first. We may absorb
\[
P((1 + W_\alpha)\bar P m) = P(W_\alpha \bar P m)
\]
into $G$. To see this, we begin by applying Coifman-Meyer 
bounds to write
\[
\| P(W_\alpha \bar P m)\|_{\dot H^\frac14}
\lesssim \| D^\frac14 W_\alpha \|_{BMO}
\|\bar P m\|_{L^2}.
\]
Then it remains to show that 
\begin{equation}\label{m-est}
    \| \bar P m\|_{L^2} \lesssim_A A_{\frac14} \|(w,r)\|_{\dH^\frac14}.
\end{equation}
Incidentally, this estimate also works toward proving that $\cG_0$
satisfies the bound in  \eqref{nf-err-bd2-GK} (we will prove \eqref{m-est} on each term
of $\cG_0$).

To prove \eqref{m-est} we  observe that the differentiation in the terms  of $m$ always is applied to one of the holomorphic variables, which has to be at or below the frequency of the undifferentiated antiholomorphic variable. 
 For instance, a typical term is
\[
\|\bar P(\bar Y r_\alpha)\|_{L^2} \lesssim \|r\|_{{\dot{H}^\frac34}} \||D|^{1/4} \bar Y\|_{BMO}
\lesssim_A A_{\frac14} \|r\|_{{\dot{H}^\frac34}}.
\]
(Here, also note that in the case of $\Pi(\bar Y, r_\alpha)$, derivatives cannot be shifted to the high frequency $W_\alpha$. We have addressed this with our choice of function space $\dH^{\frac14}$.) We also note the slightly different case where $w$ is not
differentiated, e.g.
\[
\|\bar P(\bar Y R_\alpha w)\|_{L^2} \lesssim
\| w\|_{L^4} \| \bar P(\bar Y R_\alpha)\|_{L^4}
\lesssim \| w\|_{\dot H^\frac14} \| D^\frac14 Y\|_{BMO}
\teal{\| |D|^{3/4} R \|_{L^4}} \lesssim_A \As A_\frac14 \| w\|_{\dot H^\frac14}.
\]
All other terms are similar to one of the two considered above.

Similarly, we have
\[
(1 + W_\alpha) P\bar m = T_{1 + W_\alpha} P \bar m + G,
\]
where the $G$ term is perturbative.
Then expand $P\bar m$, with terms
$$P [|1 - Y|^2(\bar r_\alpha + \bar R_\alpha \bar w)] = P [(1 - \bar Y)(\bar r_\alpha + \bar R_\alpha \bar w)Y], \qquad P[(1 - \bar Y)^2 \bar w_\alpha R].$$
We can peel off frequency components, using the observation that cubic terms where the lowest frequency variable is differentiated may be absorbed into $G$, similar to the above estimates. In particular, here $Y$ and $R$ respectively cannot be the lowest frequency due to $P$, and the cases where $\bar r_\alpha, \bar R_\alpha$ and $\bar w_\alpha$ are the lowest frequency, respectively, are perturbative. The remaining cases form
\begin{align*}
 P [T_{1 - \bar Y}(\bar r_\alpha + T_{\bar w}\bar R_\alpha)Y] &=  P [(T_{1 - \bar Y}\bar r_\alpha + T_{\bar x}\bar R_\alpha)Y] + G, \\
 \qquad P[T_{(1 - \bar Y)^2} \bar w_\alpha R] &= P[T_{1 - \bar Y} \bar x_\alpha R] + G.
 \end{align*}
After applying para-associativity Lemma~\ref{l:para-assoc2} to factor in $T_{1 + W_\alpha}$, we obtain $\cG_{0, 0}$.


The analysis for $P\cK_0 = P \bar n$ is similar, based on proving the bound
\begin{equation}\label{n-est}
    \| P \bar n\|_{\dot H^\frac12} \lesssim_A A_{\frac14} \|(w,r)\|_{\dH^\frac14}.
\end{equation}
\end{proof}

\begin{proof}[Proof of Lemma~\ref{l:GK1}]
We recall that 
 \begin{equation*}
 \begin{aligned}
& \mathcal G_1 = -(T_{w_\alpha} b + \Pi(w_\alpha, b)) + (T_{r_\alpha} \bar Y + \Pi(r_\alpha, \bar Y)) - (T_w (1 - \bar Y) R_\alpha + \Pi(w, (1 - \bar Y)R_\alpha)),
 \\
& \mathcal K_1 = -(T_{r_\alpha} b + \Pi(r_\alpha, b)) + i(T_{1 - Y}T_{w}a + T_{1 - Y}\Pi(w, a) - T_{(1 + a)w}Y - \Pi((1 + a)w, Y)).
 \end{aligned}
 \end{equation*}
  Putting aside the nonperturbative terms on the right hand side of Lemma \ref{l:GK1}, and those antiholomorphic terms eliminated by $P$, we observe that it remains to bound
 \begin{equation*}
 \begin{aligned}
& -(T_{w_\alpha} b + \Pi(w_\alpha, b)) + \Pi(r_\alpha, \bar Y)) - (T_w \Pi(- \bar Y, R_\alpha) + \Pi(w, (1 - \bar Y)R_\alpha)),
  \end{aligned}
 \end{equation*}
 respectively
 \begin{equation*}
 \begin{aligned}
& -(T_{r_\alpha} b + \Pi(r_\alpha, b)) + i(T_{1 - Y}T_{w}a + T_{1 - Y}\Pi(w, a) - \Pi((1 + a)w, Y)).
 \end{aligned}
 \end{equation*}

We consider terms from $\cG_1$ first. We observe that all the terms are balanced, with derivatives on low frequencies, so that we have the desired estimate. For instance, 
\[
\|T_{w_\alpha} b\|_{L^2} \lesssim \|w\|_{\dot{H}^{1/4}} \||D|^{3/4}b\|_{BMO} \lesssim_A \|w\|_{\dot{H}^{1/4}} \||D|^{3/4}R\|_{BMO}.
\]

A similar analysis holds for most of the terms from $\cK_1$, except
\[
T_{1 - Y}T_{w}a + T_{1 - Y}\Pi(w, a).
\]
For these terms, we use
\begin{align*}
\|T_{1 - Y}T_{w}a\|_{\dot{H}^{1/2}} &\lesssim\|Y\|_{L^\infty} \teal{\|w\|_{L^4}\||D|^{1/2} a\|_{L^4}}\teal{\lesssim_{\As} A_{\frac14}\| w\|_{\dot{H}^{1/4}}},
\end{align*}
using here Lemma \ref{regularity for a} to estimate $a$.

 \end{proof}

Putting these lemmas together, we obtain the following formulas, with estimates, for the para-material derivatives of $(w, r)$:

\begin{lemma}\label{l:Dtwr}
Given $(w, r)$ satisfying \eqref{paralin(wr)-full}, we have the representations
\begin{equation*}
\left\{
\begin{aligned}
& T_{D_t} {w} = -T_{1- \bar Y} (r_\alpha + T_{w}R_\alpha) + G_2 =: G_1 + G_2
 \\
& T_{D_t} {r} = iT_{1 + a}( T_{1 - Y} {w} - T_w Y) + K_2 =: K_1 + K_2,
\end{aligned}
\right.
\end{equation*}
and likewise with $D_t (w, r)$ in place of $T_{D_t}(w, r)$,
where $(G_2,K_2)$ satisfies the quadratic bound
\begin{equation*}
\begin{aligned}
\| ( G_2, K_2)\|_{\H^{0}} &\lesssim_{\As} A_{\frac14} \| (w,r)\|_{\dH^{\frac14}},
\end{aligned}
\end{equation*}
and $(G_1,K_1)$ satisfies the linear bound
\begin{equation*}
\||D|^{-1/4} (G_1, K_1)\|_{\H^{0}} \lesssim_{\As} \| (w,r)\|_{\dH^{\frac14}}.
\end{equation*}
\end{lemma}

\begin{proof}
Apply Lemmas~\ref{l:GK} and \ref{l:GK1}. Then observe that
\begin{equation*}
\|T_{(1 - \bar Y)R_\alpha} {w}\|_{L^2} \lesssim \|\bar Y\|_{L^\infty} \||D|^{3/4} R\|_{BMO} \|w\|_{\dot{H}^{1/4}},
\end{equation*}
so that this remaining term from the left hand side of \eqref{paralin(wr)-full} may be absorbed in $G_2$. Note that the para-commutator of Lemma~\ref{l:para-com} allows us to reorder paraproducts freely.

To estimate $(G_1, K_1)$, the terms 
\[
-T_{1- \bar Y} r_\alpha,\qquad i T_{1 + a}T_{1 - Y} {w}
\]
are straightforward, with a sufficient balance of derivatives on $(w, r)$ already. For the remaining two terms, apply also Sobolev embedding. For instance,
\[
\||D|^{-1/4} T_{1-\bar Y} T_{w} R_\alpha\|_{L^2} \lesssim \|\bar Y\|_{L^\infty}\teal{ \|w\|_{L^4} \||D|^{3/4}R\|_{L^4}} \lesssim_{A} A\As \|w\|_{\dot{H}^{1/4}}.
\]
\end{proof}

We likewise have an expansion for the para-material derivative of $x$:
\begin{lemma}\label{l:Dtx}
Given $(w, r)$ satisfying \eqref{paralin(wr)-full}, we have the representation
\begin{equation*}
\begin{aligned}
& T_{D_t} {x} =  -T_{1- \bar Y} (T_{1 - Y} r_\alpha + T_{x}R_\alpha) + G_2 =: G_1 + G_2
\end{aligned}
\end{equation*}
and likewise with $D_tx$ in place of $T_{D_t}x$,
where $G_2$ satisfies the quadratic bound
\begin{equation*}
\begin{aligned}
\|G_2\|_{L^2} &\lesssim_{\As} A_{\frac14} \| (w,r)\|_{\dH^{\frac14}},
\end{aligned}
\end{equation*}
and $G_1$ satisfies the linear bound
\begin{equation*}
\||D|^{-1/4} G_1\|_{L^2} \lesssim_{\As} \| (w,r)\|_{\dH^{\frac14}}.
\end{equation*}
\end{lemma}
Even though this is not needed later, we remark that $G_2$
also satisfies the last bound.
\begin{proof}
Using the para-Leibniz rule Lemma \ref{l:para-L}, we may write (collecting all terms satisfying the appropriate estimate into $G_2$)
\begin{equation}
T_{D_t} x = T_{-T_{D_t} Y}w + T_{1 - Y}T_{D_t}w + G_2.
\end{equation}
To the first term, we apply Lemma~\ref{l:XY-mat} to see that it is perturbative. Apply Lemma~\ref{l:Dtwr} to the second term, to obtain (rewriting using the para-commutator Lemma~\ref{l:para-com}, and dropping perturbative components as usual)
\[
-T_{1 - Y}T_{1- \bar Y} (r_\alpha + T_{w}R_\alpha) + G_2 = -T_{1- \bar Y} (T_{1 - Y} r_\alpha + T_{x}R_\alpha) + G_2.
\]

\end{proof}

 \subsection{Normal form analysis for \texorpdfstring{$(\mathcal G_1,\mathcal K_1)$}{}}
 \label{s:nf-1}
 
 In this section we introduce normal form corrections to $(w, r)$ which will cancel the leading order 
 contributions of the paradifferential truncations $(\mathcal G_1,\mathcal K_1)$. Precisely, 
 we will show the following:

\begin{proposition}
 Assume that $(w,r)$ solve \eqref{paralin(wr)-full}. Then there exists a linear normal form correction
\[
(\tw,\tr)(t)= NF((w,r)(t))
\]
solving an equation of the form
 \begin{equation}
 \label{paralin(wr)GK1}
 \left\{
\begin{aligned}
& T_{D_t} \tilde{w} + T_{1- \bar Y} \D_\alpha \tilde{r}
+ T_{(1 - \bar Y)R_\alpha} \tilde{w} = -P\cG_1(w, r) + G
 \\
& T_{D_t} \tilde{r} - i T_{1 - Y}T_{1 + a} \tilde{w} = 
-P\cK_1(w, r) + K,
\end{aligned}
\right.
 \end{equation}
with the following properties:

(i) Quadratic correction bound: 
\begin{equation}\label{nf-bd}
\| NF(w, r)\|_{\dH^{1/4}} \lesssim_A \teal{\As} 
\| (w,r)\|_{\dH^{1/4}},
\end{equation}

(ii) Secondary correction bound:
\begin{equation}\label{nf-bd2}
\| NF(g, k)\|_{\dH^{1/4}} \lesssim_A A_{\frac14} \|(g, k)\|_{\H^{0}},
\end{equation}

(iii) Cubic error bound:
\begin{equation}\label{nf-err-bd}
\|( G, K)\|_{\dH^{1/4}} \lesssim_{\teal{\As}} A_{\frac14}^2 \| (w,r)\|_{\dH^{1/4}}.
\end{equation}


 \end{proposition}

The motivation for the secondary bound is that the normal form corrections will be constructed not in terms of solutions for the homogeneous 
problem \eqref{paralin(wr)}, but in terms of solutions 
for the inhomogeneous problem \eqref{paralin(wr)-full} (which recall originally takes the form \eqref{lin(wr)0}). As such, we need to be able to estimate the effect of source 
terms via (ii). For this purpose 
we use the $\H^0$ norm, which is suitable in conjunction with the source terms $(G_2, K_2)$ of Lemma~\ref{l:Dtwr}.

\

We will build $(\tilde{w}, \tilde{r})$ in three steps. In the first step, we apply a correction to generate the quadratic paradifferential truncations $(\mathcal G_1,\mathcal K_1)$
  except temporarily putting aside the cubic truncation $T_w a$ corresponding to the Taylor coefficient $a$, which is itself quadratic.

By viewing the linearized variable $w$ as a low frequency coefficient, $T_w a$ may be viewed as quadratic in the nonlinear variable $R$. More generally, the correction of the first step misses unbalanced cubic source terms which are quadratic in the nonlinear variables $(W, R)$ with low frequency $w$ coefficients. We address most of these source terms via a correction in the second step, other than those quadratic in $R$, which we resolve in a third step. 

\

Before defining the corrections $(\tw, \tr)$, we enumerate notation that we will use in this subsection. Corresponding to the three steps above, we will define
$$(\tw, \tr) = (\tw_1, \tr_1) + (\tw_2, \tr_2) + (\tw_3, \tr_3),$$
where $(\tw_i, \tr_i)$ will be defined below. We denote the source terms of the paradifferential $(\tilde{w}_i, \tilde r_i)$ equation by $(\tilde{\cG}_i,\tilde{\cK}_i)$, i.e.
\begin{equation}\label{eq:para-linearized}
\left\{
\begin{aligned}
& T_{D_t} \tilde{w}_i + T_{1- \bar Y} \D_\alpha \tilde{r}_i
+ T_{(1 - \bar Y)R_\alpha} \tilde{w}_i = \tilde{\cG}_i
 \\
& T_{D_t} \tilde{r}_i - i T_{1 - Y}T_{1 + a} \tilde{w}_i = \tilde{\cK}_i.
\end{aligned}
\right.
\end{equation}

Throughout, our normal form corrections reflect two cases:
\begin{enumerate}[label= (\roman*)]

\item High-low interactions,

\item  Balanced interactions with low frequency output.

\end{enumerate}
Regarding the second case, it is convenient to let $\Pi = P\Pi$ throughout, thus including an implicit projection $P$.

Lastly, it will be convenient to state our corrections in terms of the following intermediate between $W$ and $\D_\alpha^{-1}Y$:
$$X := T_{1 - Y} W.$$

\

We proceed to define our first set of corrections,
\begin{align*}
\tilde{w}_1 &= - \D_\alpha (T_w X + \Pi(w, X)) - \Pi(w_\alpha, \bar X), \\
\tilde{r}_1 &= - (T_{r_\alpha}X + \Pi(r_\alpha, X)) - \Pi(r_\alpha, \bar X).
\end{align*}
We remark that at the quadratic level, this 
correction is nothing but the 
corresponding (i.e. non-paradifferential) holomorphic
(i.e. linear in $(w,r)$) part of the linearization of the normal form transformation \eqref{nft1}.

We first verify that this correction satisfies \eqref{nf-bd}. Indeed, by Sobolev embeddings 
and paraproduct bounds, 
\[
\|\D_\alpha (T_w X)\|_{\dot{H}^{1/4}} \lesssim\|w\|_{L^4} \||D|^{1/4} X_\alpha\|_{L^4} \lesssim_A \|w\|_{\dot{H}^{1/4}} \|W_\alpha\|_{\dot{W}^{1/4,4}}.
\]
Similarly,
\teal{
\[
\|T_{r_\alpha} X\|_{\dot{H}^{3/4}} \lesssim\||D|^{-1/4}r_\alpha\|_{L^2} \|X_\alpha\|_{BMO} \lesssim_A  \|r\|_{\dot{H}^{3/4}} \|W_\alpha\|_{L^\infty}.
\]}

The balanced frequency terms are straightforward; for instance, 
$$\|\D_\alpha \Pi(w, X)\|_{\dot{H}^{1/4}} \lesssim_A \|w\|_{\dot{H}^{1/4}}\|W_\alpha\|_{L^\infty}.$$

It is also straightforward to verify that this correction satisfies the secondary bound \eqref{nf-bd2}; for instance,
$$\|\D_\alpha (T_{\teal{g}} X)\|_{\dot{H}^{1/4}} \lesssim \||D|^{1/4}W_\alpha\|_{BMO}\|\teal{g}\|_{L^2}.$$

Next, we turn to the source terms, and prove the following:

\begin{lemma}\label{lem:cor1}
We have the representation 
\[
\tilde{\cG}_1 = -P\cG_1 + \cG_{1, 1} + G, \qquad \tilde{\cK}_1 = -P\cK_1 + \cK_{1, 1} + K
\]
where $(G,K)$ satisfy the bound \eqref{nf-err-bd}, whereas $(\cG_{1, 1},\cK_{1, 1})$ are given by 
\begin{align*}
\cG_{1, 1}(w, r) &= T_w(T_{1 - \bar Y} \D_\alpha^2 \Pi(X, R) + \Pi(Y_\alpha, \bar R) + T_{1 - \bar Y}\Pi(R_{\alpha\alpha}, \bar X) - PM), \\
\cK_{1, 1}(w, r) &= iT_w(T_{1 + a} \D_\alpha \Pi(Y, X)  + T_{1 + a}\Pi(Y_\alpha, \bar X) + T_{1 - Y}Pa).
\end{align*}
\end{lemma}

In other words, our first normal form correction $\tilde w_1$ does not quite allow us to dispense with the paradifferential truncations $(\cG_1,\cK_1)$. Instead, it allows us to 
replace it with the milder source term  $(\cG_{1, 1},\cK_{1, 1})$ (modulo perturbative errors $(G, K)$). These residual source terms are milder in that they are cubic, though unbalanced and hence not quite perturbative.

\

Throughout the computations below, we use the observation that cubic and higher order terms such that the lowest frequency variable is ``fully differentiated'' may be absorbed into $G$ or $K$, where 
$$(|D|^{5/4}W, |D|^{3/4}R) \in BMO, \quad (|D|^{1/4}w, |D|^{3/4}r) \in L^2.$$
For instance, for the following typical term appearing in the first equation, the lowest frequency variable is either $Y$ or $r$, and both are fully differentiated (and in fact over-differentiated):
$$\|T_{Y_\alpha} T_{r_\alpha} W\|_{\dot{H}^{\frac14}} \lesssim \||D|^{-3/4}Y_\alpha\|_{BMO} \||D|^{-1/4} r_\alpha\|_{L^2} \||D|^{5/4}W\|_{BMO} \lesssim_A A_{\frac14}^2 \|r\|_{\dot{H}^{\frac34}},$$
which verifies \eqref{nf-err-bd}. 

For quartic and higher order terms, terms such that the lowest frequency variable is a linearized variable $w$ (even with no derivatives) may also be absorbed into $G$ or $K$, by measuring $w$ in $L^4$ and applying a Sobolev embedding. For instance, for the following typical term appearing in the first equation, the lowest frequency variable may also be $w$:
$$\|T_{Y_\alpha R_\alpha} T_{w} W\|_{\dot{H}^{\frac14}} \lesssim \||D|^{-3/4}Y_\alpha\|_{BMO}\||D|^{-1/4} R_\alpha\|_{BMO}\|w\|_{L^4}\| D^\frac14 W_\alpha\|_{L^4} \lesssim_{\As} A_{\frac14}^2 \|w\|_{\dot{H}^{\frac14}},$$
which verifies \eqref{nf-err-bd}. 

\

\begin{proof}

We compute the contributions from $(\tw_1, \tr_1)$ when inserted in the left hand side of \eqref{eq:para-linearized}, separating the output into the three components as in the statement of the lemma. It is instructive to compare the leading order bilinear interactions with the normal form computations in the proof of Proposition \ref{prop:cubiccommutator}, though the situation here is more delicate as we cannot apply an exponential conjugation to lower the order of our source terms and corrections.

\

1) We begin by computing the first equation in \eqref{eq:para-linearized}, starting with the high-low corrections in $(\tw_1, \tr_1)$. After commuting $D_t$ with $\D_\alpha$,
\begin{align*}
-T_{D_t} \D_\alpha T_w X &= -\D_\alpha T_{D_t} T_w X + T_{b_\alpha} \D_\alpha T_w X
= -\D_\alpha T_{D_t} T_w X + T_{b_\alpha} T_w X_\alpha+ G,
\end{align*}
we have the following four contributions to the first equation, using the para-Leibniz rule 
in Lemma~\ref{l:para-L}, (precisely the $\sigma = 0$ case of the bound \eqref{Leibniz-T0+}), then part (b) of Lemma \ref{l:XY-mat} for the para-material derivative of $X$, and Lemma~\ref{l:Dtwr} for the para-material derivative of $(w, r)$ (combined with the secondary bound \eqref{nf-bd2} on $(\tw_1, \tr_1)$ to justify absorbing the source terms into $(G, K)$):
\begin{align*}
-\D_\alpha T_{D_t} T_w X &= -\D_\alpha ( T_{D_t w} X + T_w T_{D_t} X) + G \\
& = \D_\alpha (T_{T_{(1 - \bar{Y})}(r_\alpha + T_w R_\alpha)} X \!+\! T_w(- T_{T_{1 - \bar Y}R_\alpha} X\! +\! P((1 - \bar Y)R) + \Pi(X_\alpha, b))) + G \\ 
&= \D_\alpha T_{T_{1 - \bar{Y}} r_\alpha} X +  \D_\alpha T_w P[(1 - \bar Y)R] +  T_w \partial_\alpha \Pi(X_\alpha, b) + G\\
T_{b_\alpha}  T_w X_\alpha &= T_w T_{T_{1 - \bar Y} R_\alpha + T_{1 - Y}\bar R_\alpha}  X_\alpha + G \\
-T_{1 - \bar{Y}} \D_\alpha T_{r_\alpha} X &= -\D_\alpha T_{T_{1 - \bar Y} r_{\alpha}} X + G\\
-T_{(1 - \bar{Y})R_\alpha}\D_\alpha T_w X &= -T_w T_{T_{1 - \bar{Y}} R_\alpha} X_\alpha + G.
\end{align*}

Consider the middle term of the first contribution. Applying the product rule, we have
\[
\D_\alpha T_w P[(1 - \bar Y)R] = T_{w_\alpha}P[(1 - \bar Y)R] + T_w P [(1 - \bar Y) R_\alpha] - T_w P[\bar Y_\alpha R].
\]
The first two terms (note $T_{w_\alpha}P [(1 - \bar Y)R] = P T_{w_\alpha} b$) are two of the high-low paradifferential truncations in $\cG_1$. 

Summing all remaining nonperturbative terms and observing cancellations, we obtain
the secondary error terms
\begin{align}\label{eq:corrections1}
 T_w P\left[ - \bar Y_\alpha R  + \D_\alpha  \Pi(X_\alpha, b) + T_{T_{1 - Y}\bar R_\alpha}   X_\alpha \right]
\end{align}
which we collect in $\cG_{1, 1}$.

\

2) Continuing with the first equation in \eqref{eq:para-linearized}, we have analogous computations for the frequency balanced corrections with holomorphic variables, with the following differences:
\begin{enumerate}[label= (\roman*)]
\item Several terms have better derivative balance and may be absorbed into $G$ directly. Indeed, all of the errors corresponding to those in \eqref{eq:corrections1} are perturbative.
\item On the other hand, the following cancellation no longer applies, as the right hand side is in fact perturbative:
\[
\partial_\alpha \Pi (T_{1 - \bar Y} T_w R_\alpha, X) \neq \partial_\alpha 
\Pi(w, T_{T_{1 - \bar Y}R_\alpha} X) = G.
\]

\end{enumerate}
As a result, we obtain the desired balanced interaction truncations analogous to before, but with a different residual for $\cG_{1, 1}$,
\begin{align}\label{eq:corrections2}
 T_w T_{1 - \bar Y} \D_\alpha \Pi(R_\alpha, X),
\end{align}
where we have used Lemma~\ref{l:para-assoc} to factor out the coefficients. 

\

3) Lastly for the first equation in \eqref{eq:para-linearized}, we have similar computations for the frequency balanced corrections with one anti-holomorphic variable. The difference with 2) is that we include only the correction where $\D_\alpha$ falls on $w$. Using Lemmas~\ref{l:para-L}, \ref{l:XY-mat} we have
\begin{align*}
-T_{D_t} \Pi(w_\alpha, \bar X) &= \Pi(\partial_\alpha T_{D_t} w, \bar X)
+  \Pi(w_\alpha, T_{D_t} \bar X) + G \\
 & = T_{1 - \bar Y} \Pi(r_{\alpha \alpha} + T_w R_{\alpha \alpha}, \bar X) +\Pi(w_\alpha, T_{1 - Y}\bar R) + G \\
-T_{1 - \bar{Y}} \D_\alpha \Pi(r_\alpha, \bar X) &= - T_{1 - \bar Y} \Pi(r_{\alpha\alpha}, \bar X) - \Pi(r_\alpha, T_{1 - \bar Y}\bar X_\alpha) + G.
\end{align*}
The second term of the first contribution,
$$\Pi(w_\alpha, T_{1 - Y}\bar R) = \Pi(w_\alpha, P((1 - Y)\bar R)) + G,$$
is the other desired balanced interaction truncation in $\cG_1$ from $\Pi(w_\alpha, b)$.

Consider the second term of the second contribution. Applying Lemma \ref{l:YtoW} to express $X$ in terms of $Y$, we have
$$- \Pi(r_\alpha, T_{1 - \bar Y}\bar X_\alpha) = - \Pi(r_\alpha, \bar Y) + G.$$
This is the last desired paradifferential truncation in $\cG_1$.

After cancellations, the only residual term for $\cG_{1, 1}$ is
\begin{align}\label{eq:corrections2.1}
T_{1 - \bar Y} \Pi(T_w R_{\alpha\alpha}, \bar X) = T_w T_{1 - \bar Y} \Pi(R_{\alpha\alpha}, \bar X) + G,
\end{align}
where we have used Lemma~\ref{l:para-assoc} to factor out the coefficients.

Summing \eqref{eq:corrections1}, \eqref{eq:corrections2} and \eqref{eq:corrections2.1}, we have
the residual terms
\begin{align*}
&- T_w P\left[\bar Y_\alpha R  + \D_\alpha \Pi(X_\alpha, b) + T_{T_{1 - Y}\bar R_\alpha}   X_\alpha 
 + \D_\alpha T_{1 - \bar Y} \Pi(R_\alpha, X) + T_{1 - \bar Y} \Pi( R_{\alpha\alpha}, \bar X)\right].
\end{align*}
Applying perturbative modifications, and in particular Lemma \ref{l:YtoW}, we 
obtain  exactly $\cG_{1, 1}$ modulo perturbative terms $G$.

\

4) Next, we compute the second equation in \eqref{eq:para-linearized}. We begin with the high-low corrections, Applying the bound \eqref{Leibniz-T0+} in Lemma~\ref{l:para-L} and part (c) of Lemma \ref{l:XY-mat} for the first identity and Lemma \ref{l:YtoW} for the second, we have
\begin{align*}
-T_{D_t} T_{r_\alpha} X &= T_{\D_\alpha T_{D_t} r} X + T_{r_\alpha} T_{D_t} X + K
\\
&= - T_{i T_{1 + a}T_{1 - Y}w_\alpha  - T_{1+a} T_w Y_\alpha} X + T_{r_\alpha}T_{1 - \bar Y} R + K \\
&= - iT_{1 + a} T_{((1 - Y)w)_\alpha} X + T_{r_\alpha} T_{1 - \bar Y} R + K \\
i T_{1 - Y}T_{1 + a}\D_\alpha T_w X &= i T_{1 - Y}T_{1 + a} T_{((1 - Y)w)_\alpha} W + i T_{1 - Y}T_{1 + a} T_w T_{1 - Y}W_\alpha + K \\
&= iT_{1 + a} T_{((1 - Y)w)_\alpha} X + iT_{1 + a} T_w Y + i T_{1 - Y}T_{1 + a} T_w\Pi(Y, W_\alpha) + K.
\end{align*} 
After summing the two contributions with cancellations and applying Lemma \ref{l:YtoW}, we obtain
\[
T_{r_\alpha} T_{1 - \bar Y}R + iT_{1 + a} T_w Y + iT_{1 + a}T_w \Pi(Y, X_\alpha) + K.
\]
The first two terms (note that $T_{r_\alpha} T_{1 - \bar Y}R = P T_{r_\alpha} b + K$) are two of the high-low paradifferential truncations in $\cK_1$, leaving to $\cK_{1, 1}$ only the third term,
\begin{equation}\label{eq:corrections4.0}
iT_{1 + a}T_w \Pi(Y, X_\alpha).
\end{equation}

\

5) Continuing with the second equation in \eqref{eq:para-linearized}, we have analogous computations for the frequency balanced corrections with holomorphic variables, with the following differences, similar to step 2):
\begin{enumerate}[label= (\roman*)]
\item The error term in \eqref{eq:corrections4.0} becomes
$$iT_{1 + a}\Pi(w, \Pi(Y, X_\alpha))$$
which may be absorbed into $K$ directly.
\item On the other hand, the following cancellation no longer applies, as the right hand side is in fact perturbative:
$$iT_{1 + a}\Pi(Y_\alpha w, X) \neq i T_{1 - Y}T_{1 + a} \Pi(w, T_{Y_\alpha}W) = K.$$
\end{enumerate}
As a result, we obtain the desired balanced interaction truncations analogous to before, but with a different residual for $\cK_{1, 1}$,
\begin{align}\label{eq:corrections4}
iT_{1+a}T_w \Pi(Y_\alpha, X).
\end{align}

\

6) Lastly for the second equation in \eqref{eq:para-linearized}, we have similar computations for the frequency balanced corrections with one anti-holomorphic variable. The only difference with 5) is that we lack the term analogous to
$$i T_{1 - Y}T_{1 + a} \Pi(w, X_\alpha)$$
(contributed by $\tw_1$) which would be
$$i T_{1 - Y}T_{1 + a} \Pi(w, \bar X_\alpha),$$
but accordingly, this is absent as a paradifferential truncation in $\cK_1$. We thus obtain only the residual term analogous to that in 5):
\begin{align}\label{eq:corrections4.1}
iT_{1 + a}T_w \Pi(Y_\alpha, \bar X).
\end{align}

Summing \eqref{eq:corrections4.0}, \eqref{eq:corrections4}, and \eqref{eq:corrections4.1}, and observing that in $\cK_1$, only the high-low paradifferential error from the Taylor coefficient $a$ remains, we obtain $\cK_{1, 1}$. 

\end{proof}


We proceed with the second step in the construction of the normal form correction, which cancels those errors in $(\cG_{1, 1}, \cK_{1, 1})$ which can be seen as quadratic in the nonlinear variables $(W, R)$ and with $w$ as a low frequency coefficient. The corresponding correction addresses the quadratic part, while $w$ is simply retained as a coefficient.
We define this correction as\footnote{This is obtained again by a standard normal form analysis, but taking care to also match the low frequency coefficients.}

\begin{align*}
\tilde{w}_2 &= T_w(-T_{\frac{1}{1 + a}}Pa + \half \D_\alpha^2 \Pi(X, X) + \Pi(X_{\alpha\alpha}, \bar X)),\\
\tilde{r}_2 &= 0.
\end{align*}
We first verify that this correction satisfies \eqref{nf-bd}. By Sobolev embedding and Proposition~\ref{regularity for a},
\[
\|T_wT_{\frac{1}{1 + a}}a\|_{\dot{H}^{\frac14}} \lesssim \teal{\|w\|_{L^4}(1+ \|a\|_{L^\infty}) \||D|^{1/4} a\|_{L^4}}  \lesssim_A 
 \teal{\As^2} \|w\|_{\dot{H}^{1/4}}.
\]

The balanced terms are then easily estimated.

As with $(\tw_1, \tr_1)$, the secondary bound \eqref{nf-bd2} is also easily verified. For instance, 
\[
\|T_fT_{\frac{1}{1 + a}}a\|_{\dot{H}^{\frac14}} \lesssim \|f\|_{L^2} (1+\|a\|_{L^\infty}) \||D|^{1/4} a\|_{BMO}  \lesssim_A \teal{A A_{\frac14}}\|f\|_{L^2}.
\]
Next we turn our attention to the source terms:

\begin{lemma}\label{lem:cor2}

We have the representation
$$\tilde \cG_2 = -\cG_{1, 1} + \cG_{1, 2} + G, \qquad \tilde \cK_2 = -\cK_{1, 1} + \cK_{1, 2} + K$$
where $(G,K)$ satisfy the bound \eqref{nf-err-bd}, whereas $(\cG_{1, 2},\cK_{1, 2})$ are given by 
\begin{align*}
\cG_{1, 2}(w,r) = 0, \qquad \cK_{1, 2}(w, r) &= 2iT_{x}Pa.
\end{align*}
\end{lemma}

We see this nearly dispenses with $ (\cG_{1, 1}, \cK_{1, 1})$, other than a residual Taylor coefficient truncation in $\cK_{1, 2}$, which we will leave to the third step of our normal form construction.


\begin{proof}

We begin by considering the Taylor coefficient correction term in $\tilde{w}_2$, starting with the contribution to the first equation of \eqref{eq:para-linearized}. This is done as in the proof of the previous lemma, using Lemma~\ref{l:para-L} 
to distribute the para-material derivative:
\begin{align*}
- T_{D_t}T_w T_{\frac{1}{1 + a}}a & = T_{T_{D_t} w} T_{\frac{1}{1 + a}}a
+ T_w T_{T_{D_t} \frac{1}{1+a}} a + T_w T_{\frac{1}{1 + a}} T_{D_t} a + G.
\end{align*}
The first term goes into $G$ using the second estimate of Lemma \ref{l:Dtwr} for the material derivative of $w$ along with the bounds on $a$ in Lemma \ref{regularity for a}:
\[
\| T_{T_{D_t} w} T_{\frac{1}{1 + a}}a \|_{\dot H^\frac14} \lesssim_A \| T_{D_t} w\|_{\dot H^{-\frac14}} \| |D|^\frac12 a\|_{BMO} \lesssim_{\As} A_{\frac14}^2 \|((w,r)\|_{\dH^\frac14}. 
\]
The second term is also perturbative, as we can use \eqref{a-bmo+} to complete
$T_{D_t}$ to $D_t$ and then finish with \eqref{aflow}. Lastly, we have by \eqref{para-aflow}, 
\begin{equation*}
T_{D_t}P a = -T_{1+a}PM + E,
\end{equation*}
where $E$ is perturbative via the error bound
\begin{equation*}
\teal{\|E\|_{\dot{W}^{\frac14, 4}}} \lesssim_{\As} A_{\frac14}^2.
\end{equation*}

Summarizing, we have proved that
\begin{align*}
-T_{D_t}T_w T_{\frac{1}{1 + a}}Pa &= T_w PM + G
\end{align*}
which matches the last two terms of $\mathcal{G}_{1, 1}$.
To the second equation, this correction term generates
\begin{align*}
i T_{1 - Y}T_{1 + a}T_w T_{\frac{1}{1 + a}}Pa &= iT_xPa + K.
\end{align*}
This contributes to $\mathcal{K}_{1, 2}$ a second copy of the Taylor coefficient term, where we have rewritten
\[
iT_wT_{1 - Y}Pa = iT_xPa + K,
\] 
perturbatively peeling off the case where $w$ is lowest frequency, as this component is quartic.

\

We proceed with the remaining correction terms, which are somewhat simpler:

\

1) We compute the first equation in \eqref{eq:para-linearized}, starting with the term with holomorphic variables and using Lemma~\ref{l:para-L} and Lemma~\ref{l:XY-mat}:
\begin{align*}
 \half T_{D_t} T_w \D_\alpha^2 \Pi(X, X) &= -T_w T_{1 - \bar Y} \D_\alpha^2\Pi(X, R) + G
\end{align*}
which is the first term of $\mathcal{G}_{1, 1}$. 

\

2) Continuing with the first equation, we address the term with one anti-holomorphic variable
in the same manner:
\begin{align*}
T_{D_t} T_w \Pi(X_{\alpha\alpha}, \bar X) &= -T_w(T_{1 - \bar Y}\Pi(R_{\alpha\alpha}, \bar X) + T_{1 - Y}\Pi(X_{\alpha\alpha}, \bar R)) + G.
\end{align*}
These are the remaining terms of $\mathcal{G}_{1, 1}$, using Lemma \ref{l:YtoW} on the second term.

\

3) We compute the second equation in \eqref{eq:para-linearized}, starting with the term with holomorphic variables:
\begin{align*}
-\half i T_{1 - Y}T_{1 + a}T_w\D_\alpha^2 \Pi(X, X) &= -i T_{1 - Y}T_{1 + a}T_w \D_\alpha\Pi(X_{\alpha}, X) + K
\end{align*}
which is the first term of $\mathcal{K}_{1, 1}$, after applying Lemma \ref{l:YtoW}.

\

4) Continuing with the second equation, we address the term with one anti-holomorphic variable:
\begin{align*}
-i T_{1 - Y}T_{1 + a} T_w \Pi(X_\alpha, \bar X) &= 
-iT_{1 + a} T_w \Pi(Y_\alpha, \bar X) + G.
\end{align*}
This is the remaining term of $\mathcal{K}_{1, 1}$.

\end{proof}

It remains to cancel the paradifferential truncation from the Taylor coefficient. We define the correction
\begin{align*}
\tilde{w}_3 &= -P((T_w \bar X_\alpha)X_\alpha) - T_{\frac{1}{1 + a}} T_w Pa, \\
\tilde{r}_3 &= -P((T_w \bar X_\alpha) R).
\end{align*}
As with $(\tw_2, \tr_2)$, it is straightforward to verify that this correction satisfies \eqref{nf-bd} and \eqref{nf-bd2}.

\begin{lemma}\label{lem:cor3}
We have
$$\tilde \cG_3 = -\cG_{1, 2} + G = G, \qquad \tilde \cK_2 = -\cK_{1, 2} + K = -2iT_{x}P a + K,$$
where $(G,K)$ satisfy the bound \eqref{nf-err-bd}.
\end{lemma}

\begin{proof}
1) We begin by checking that no nonperturbative errors are contributed to the first equation in \eqref{eq:para-linearized}. For this we repeatedly use Lemma~\ref{l:para-L}, noting that the terms 
involving the para-material derivative $T_{D_t}$ of the "coefficients" $w$ and $\frac{1}{1+a}$,
as well as those involving  $[T_{D_t},\partial_\alpha]$, are all perturbative. We also repeatedly apply the para-associativity lemma, Lemma~\ref{l:para-assoc2} and use \eqref{para-aflow} to compute the para-material derivative of $a$:
\begin{align*}
- T_{D_t}P ((T_w \bar X_\alpha)X_\alpha) &=  P[(T_w \D_\alpha T_{D_t}\bar X)X_\alpha] +
P[(T_w \bar X_\alpha)\D_\alpha  T_{D_t} X] + G\\
& = P[(T_w T_{1 - Y} \bar R_\alpha)X_\alpha] + P[(T_w \bar X_\alpha)T_{1 - \bar Y}R_\alpha] + G \\
& = T_wP[\bar R_\alpha Y] + P[(T_w \bar X_\alpha)T_{1 - \bar Y}R_\alpha] + G ,\\
-T_{D_t} T_{\frac{1}{1 + a}}T_w Pa &= -T_{\frac{1}{1 + a}} T_w T_{D_t} Pa + G  = T_w PM + G \\
&= T_w P[R \bar Y_\alpha - \bar R_\alpha Y] + G ,\\
- T_{1 - \bar Y}\D_\alpha ((T_w \bar X_\alpha) R) &= - T_{1 - \bar Y}[(T_{w} \bar X_{\alpha\alpha}) R] - T_{1 - \bar Y}[(T_w \bar X_\alpha) R_\alpha] + G ,\\
 T_{(1 - \bar Y) R_\alpha}( - (T_w \bar X_\alpha)X_\alpha - T_{\frac{1}{1 + a}} T_wP a) &= G.
\end{align*}
Summing and using Lemma \ref{l:YtoW}, all the terms cancel up to $G$ errors.

\

2) For the second equation of \eqref{eq:para-linearized}, we compute
\begin{align*}
- T_{D_t} P[(T_w \bar X_\alpha) R] &= P[(T_{w} T_{1 - Y}\bar R_\alpha) R] - iP[(T_w \bar X_\alpha) T_{1 + a}Y] + K, \\
- i T_{1 - Y}T_{1 + a}( -P[(T_w \bar X_\alpha)X_\alpha] - T_{\frac{1}{1 + a}} T_w Pa) &=  iT_{1 + a}P[(T_w \bar X_\alpha) Y] + iT_xPa + K.
\end{align*}
Summing, we see that two terms cancel, leaving two copies of the desired Taylor coefficient term.
\end{proof}


\subsection{Normal form analysis for \texorpdfstring{$(\mathcal G_0,\mathcal K_0)$}{}}
\label{s:nf-0}

 In this section we introduce normal form corrections to $(w, r)$ which will cancel the leading order contributions of the source terms $(\mathcal G_0,\mathcal K_0)$ of the linearized equations, defined in \eqref{GK}. Precisely, 
 we will show the following:
 
 \begin{proposition}
 Assume that $(w,r)$ solve \eqref{paralin(wr)-full}. Then there exists a linear normal form correction
\[
(\tw,\tr)(t)= NF((w,r)(t))
\]
solving an equation of the form
 \begin{equation}
 \label{paralin(wr)GK0}
 \left\{
\begin{aligned}
& T_{D_t} \tilde{w} + T_{1- \bar Y} \D_\alpha \tilde{r}
+ T_{(1 - \bar Y)R_\alpha} \tilde{w} = -P\cG_0 + G
 \\
& T_{D_t} \tilde{r} - i T_{1 - Y}T_{1 + a} \tilde{w} = 
-P\cK_0 + K,
\end{aligned}
\right.
 \end{equation}
with the following properties:

(i) Quadratic correction bound: 
\[
\| NF(w, r)\|_{\dH^{1/4}} \lesssim_A  \teal{\As} \| (w,r)\|_{\dH^{1/4}},
\]

(ii) Secondary correction bound:
\begin{equation*}
\| NF(g, k)\|_{\dH^{1/4}} \lesssim_A A_{\frac14} \|(g, k)\|_{\dH^{0}},
\end{equation*}

(iii) Cubic error bound: 
\[
\|( G, K)\|_{\dH^{1/4}} \lesssim_{\teal{\As}} A_{\frac14}^2 \| (w,r)\|_{\dH^{1/4}}.
\]
 \end{proposition}

Similar to the normal form correction for the paradifferential truncations $(\cG_1, \cK_1)$, we will build the correction in three steps, denoting the output of the three steps by, respectively,
$$(\tw, \tr) = (\tw_1, \tr_1) + (\tw_2, \tr_2) + (\tw_3, \tr_3),$$
along with their source terms $(\cG_i, \cK_i)$ in \eqref{eq:para-linearized}.

In the first step, we construct a correction to generate the quadratic terms, which are linear in the linearized variables $(w, r)$. Second, we construct a correction to generate the terms which are quadratic in the nonlinear variables $(W, R)$ and their conjugates (with $\bar w$ appearing as a low frequency coefficient). Lastly, we have a correction to address terms quadratic specifically in $R$.

\

We define our first set of corrections, which rectify the quadratic source terms which are linear in 
$(w, r)$ (unlike the last two corrections, which 
are primarily quadratic in $(W,R)$ and with $(w,r)$
playing the role of coefficients):
\begin{align*}
\tilde{w}_1 &= - P(\bar x W_\alpha), \\
\tilde{r}_1 &= - P(\bar x T_{1 + W_\alpha}R).
\end{align*}
We remark that at the quadratic level, this 
correction is nothing but the 
corresponding (i.e. non-paradifferential) antiholomorphic
(i.e. linear in $(\bar w,\bar r)$) part of the linearization of the normal form transformation \eqref{nft1}.

These corrections satisfy the bound \eqref{nf-bd}, as is easily verified with Sobolev embedding, as well as \eqref{nf-bd2}, similar to $\tw_1$ in Section \ref{s:GK}.

We will prove the following, where  we recall that 
$(\cG_{0, 0}, \cK_{0, 0})$ is defined in Lemma \ref{l:GK}. 

\begin{lemma}\label{lem:cor4}

We have the representation 
\[
\tilde{\cG}_1 = -\cG_{0, 0} + \cG_{0, 1} + G, \qquad \tilde{\cK}_1 = -\cK_{0, 0} + \cK_{0, 1} + K
\]
where $(G,K)$ satisfy the bound \eqref{nf-err-bd}, whereas $(\cG_{0, 1},\cK_{0, 1})$ are given by 
\begin{align*}
\cG_{0, 1}(w, r) &= T_{\bar x}(\Pi(W_{\alpha\alpha},b) + \Pi(T_{1 - \bar Y} W_\alpha - T_{1 + W_\alpha} \bar Y, R_\alpha ) - T_{1 + W_\alpha}PM), \\
\cK_{0, 1}(w, r) &= T_{\bar x}(iPa + 2T_{1 + W_\alpha}T_{1 - \bar Y}T_{R_\alpha}R + \Pi(T_{1 + W_\alpha}R_{ \alpha}, b) + iT_{1 + a} \Pi(Y, W_\alpha)).
\end{align*}
\end{lemma}

\begin{proof}

1) We compute the first equation in \eqref{eq:para-linearized}. For this we use Lemma~\ref{l:para-L} to distribute the para-material derivative
and then Lemma~\ref{l:Dtx} to 
capture the leading part of $D_t \bar x$. On the other hand, for $T_{D_t} W_\alpha$ we need to use the full expression \eqref{para(Wa)-full} of Lemma \ref{l:WR-source}. We will also use Lemma~\ref{l:YtoW} to relate $Y$ and $W_\alpha$ paraproducts.
\begin{align*}
-T_{D_t} P(\bar x W_\alpha) &= - P[(D_t \bar x) W_\alpha] - P[\bar x T_{D_t}W_\alpha] + G
\\
& = P[T_{1 - Y}(T_{1 - \bar Y}\bar r_\alpha + T_{\bar x}\bar R_\alpha) W_\alpha] \\
&\quad P[\bar x( \D_\alpha T_{1 + W_\alpha} P[(1 - \bar Y)R]+ T_{b_\alpha} W_\alpha  + \D_\alpha \Pi(W_\alpha, b))] + G \\
& = P[T_{1 + W_\alpha}(T_{1 - \bar Y}\bar r_\alpha + T_{\bar x}\bar R_\alpha)Y] \\
&\quad + P[\bar x ( T_{W_{\alpha\alpha}} T_{1 - \bar Y}R + T_{1 + W_\alpha} P[- \bar Y_\alpha R] + T_{1 + W_\alpha} P[(1 - \bar Y)R_\alpha])] \\
&\quad + P[\bar x(T_{b_\alpha} W_\alpha + \Pi(W_{\alpha\alpha}, b) + \Pi(W_\alpha, b_\alpha))] + G, \\
-T_{1 - \bar Y}\D_\alpha P(\bar x T_{1 + W_\alpha}R) &= -P[ T_{1 + W_\alpha}T_{1 - \bar Y}\bar x_\alpha R] \\
&\quad - P[\bar x T_{W_{\alpha\alpha}}T_{1- \bar Y}R]  \\
&\quad - P[\bar x T_{1 + W_\alpha} T_{1 - \bar Y}R_\alpha] + G, \\
-T_{(1 - \bar Y)R_\alpha} P[\bar x W_\alpha] &= -P[\bar x T_{(1 - \bar Y)R_\alpha} W_\alpha] + G.
\end{align*}

Summing up and observing cancellations, we obtain the following contributions to the 
source terms in the first equation in \eqref{eq:para-linearized}:
\begin{equation}\label{eq:RHSerror1}
\begin{aligned}
&P[T_{1 - \bar Y}T_{1 + W_\alpha}(\bar r_\alpha + T_{\bar w}\bar R_\alpha)Y]  -P[ T_{1 + W_\alpha}T_{1 - \bar Y}\bar x_\alpha R]\\
&\quad + P[\bar x(-T_{1 + W_\alpha} P[\bar Y_\alpha R] - T_{1 + W_\alpha} \Pi(\bar Y,R_\alpha))] \\
&\quad + P[\bar x(T_{(1 - Y)\bar R_\alpha} W_\alpha + \Pi(W_{\alpha\alpha}, b) + \Pi(W_\alpha, b_\alpha))] + G.
\end{aligned}
\end{equation}
The first two terms form $\mathcal{G}_{0, 0}$, as desired. The remaining terms are precisely $\mathcal{G}_{0, 1}$, after perturbative rearrangements (noting that the case where $\bar x$ is balanced with the holomorphic terms is perturbative) and applying Lemma~\ref{l:YtoW}.

\

2) We compute the second equation similarly, by applying Lemma~\ref{l:para-L} to distribute $T_{D_t}$ and Lemma~\ref{l:Dtx} for the leading part of $D_t \bar x$. We use \eqref{para(Wa)-full} of Lemma \ref{l:WR-source} for $T_{D_t}R$, but \eqref{para(Wa)} suffices for $T_{D_t} W_\alpha$. Thus we obtain 
\begin{align*}
-T_{D_t} P[\bar x T_{1 + W_\alpha} R] &= - P[(D_t \bar x) T_{1 + W_\alpha} R] - P[\bar x T_{D_t W_\alpha} R] -  P[\bar x T_{1 + W_\alpha} T_{D_t} R] \teal{+ K} \\ 
& = P[T_{1 - Y}(T_{1 - \bar Y}\bar r_\alpha + T_{\bar x}\bar R_\alpha) T_{1 + W_\alpha} R]
 + P[\bar x T_{T_{1 - \bar Y}T_{1 + W_\alpha}R_\alpha} R] \\
&\quad + P[\bar x T_{1 + W_\alpha} (i(1 + a)(1 - Y) + T_{R_\alpha}b + \Pi(R_{ \alpha}, b))]+ K, \\
iT_{1 - Y}T_{1 + a}P[ \bar x W_\alpha] &= iT_{1 + a}P[\bar x (T_{1 + W_\alpha }Y + \Pi(Y, W_\alpha))] + K.
\end{align*}
Summing up and peeling away perturbative terms using para-commutator and paraproduct estimates, we obtain the following contributions to the source term in the second equation: 
\begin{equation}\label{eq:RHSerror2}
\begin{aligned}
&P[(T_{1 - \bar Y}\bar r_\alpha + T_{\bar x}\bar R_\alpha) R]\\
&\quad + P[\bar x(ia + 2T_{1 + W_\alpha}T_{1 - \bar Y}T_{R_\alpha}R + \Pi(T_{1 + W_\alpha}R_{ \alpha}, b) + iT_{1 + a} \Pi(Y, W_\alpha))]+ K.
\end{aligned}
\end{equation}
The first term matches $\mathcal{K}_{0, 0}$, as desired. The remaining terms are precisely $\mathcal{K}_{0, 1}$ (again noting that the case where $\bar x$ is balanced with the holomorphic terms is perturbative).

\end{proof}


In the second step, we introduce the correction which generate the source terms which are quadratic in the nonlinear variables appearing in $(\cG_{0, 1}, \cK_{0, 1})$: 
\begin{align*}
\tilde{w}_2 &= T_{\bar x}(\Pi(W_{\alpha\alpha}, \Re X) - \Pi(W_\alpha, \Re X_\alpha) - \Pi(W_\alpha, X_\alpha) - T_{\frac{1}{1 + a}}T_{1 + W_\alpha}P a), \\
\tilde{r}_2 &= T_{\bar x}(\Pi(T_{1 + W_\alpha}R_\alpha, \Re X) - \Pi(\Re T_{1 + W_\alpha} R, X_\alpha)).
\end{align*}
Here we have treated $T_{\bar x}$ as a low frequency coefficient which simply carries over to the corrections.
These corrections satisfy the bound \eqref{nf-bd}, as it is easily verified with Sobolev embeddings, as well as \eqref{nf-bd2}, which is similar to the case of $\tw_2$ in Section \ref{s:GK}. The source terms generated by these corrections in \eqref{eq:para-linearized}
are as follows:

\begin{lemma}\label{lem:cor5}
We have the representation 
\[
\tilde{\cG}_2 = -\cG_{0, 1} + \cG_{0, 2} + G, \qquad \tilde{\cK}_1 = -\cK_{0, 1} + \cK_{0, 2} + K,
\]
where $(G,K)$ satisfy the bound \eqref{nf-err-bd}, whereas $(\cG_{0, 2},\cK_{0, 2})$ are given by 
\begin{align*}
\cG_{0, 2}(w, r) &= 0, \\
\cK_{0, 2}(w, r) &= T_{\bar x}(2iPa +  2T_{1 + W_\alpha}T_{1 - \bar Y}T_{R_\alpha}R + \half \D_\alpha \Pi(T_{1 + W_\alpha} R, b) ).
\end{align*}
\end{lemma}

\begin{proof}

We begin by considering the Taylor coefficient correction in $\tilde{w}_2$. Its contribution to the source term $\tilde{\cG}_2$ of the first equation in \eqref{eq:para-linearized} is computed using Lemma~\ref{l:para-L}, and then \eqref{para-aflow} to compute the para-material derivative of $a$. 
Here, as well as in all computations from here on, the para-material derivative
 $T_{D_t} \bar x$ of $\bar x$, given by Lemma~\ref{l:Dtx}, plays a perturbative role.
Note that the only nonperturbative case is when the derivative falls on the high frequency $a$:
\begin{align*}
-T_{D_t} T_{\bar x}T_{\frac{1}{1 + a}}T_{1 + W_\alpha}P a &= T_{\bar x}T_{1 + W_\alpha}PM + G,
\end{align*}
matching the last two two terms in $\cG_{0, 1}$. 

On the other hand its contribution to the second equation is
\[
-iT_{1 - Y}T_{1 + a}(-T_{\bar x}T_{\frac{1}{1 + a}}T_{1 + W_\alpha}Pa) = iT_{\bar x} Pa.
\]
This is the extra Taylor coefficient term $iT_{\bar x}Pa$ in $\cK_{0, 2}$.

\

It is now convenient to compute the derivatives of the term
\[
-T_{\bar x} \Pi(W_\alpha, X_\alpha)
\]
in $\tw_2$. We have in the first equation,
\[
T_{D_t}T_{\bar x} \Pi(W_\alpha, X_\alpha) = -2T_{\bar x}T_{1 - \bar Y} \Pi(W_\alpha, R_\alpha) + G,
\]
and in the second equation,
\[
-iT_{1 - Y}T_{1 + a}T_{\bar x} \Pi(W_\alpha, X_\alpha) = -iT_{\bar x} T_{1 + a}\Pi(W_\alpha, Y) + G.
\]
Thus, after $(G, K)$-perturbative modifications, it remains to match 
\begin{equation}\label{eq:para-linearized4.1}
\begin{aligned}
\mathcal{G}_{0, 1}'(w, r) &= T_{\bar x}(\Pi(W_{\alpha\alpha}, 2\Re T_{1 - \bar Y} R) - T_{1 + W_\alpha}T_{1 - \bar Y}\Pi(2\Re X_\alpha, R_\alpha )), \\
\mathcal{K}_{0, 1}'(w, r) &= T_{\bar x}(ia + 2T_{1 + W_\alpha}T_{1 - \bar Y}T_{R_\alpha}R + \Pi(T_{1 + W_\alpha}R_{ \alpha}, 2\Re T_{1 - \bar Y} R)),
\end{aligned}
\end{equation}
instead of $(\mathcal{G}_{0, 1},\mathcal{K}_{0, 1})$, with correction terms
\begin{align*}
\tilde{w}_2' &= T_{\bar x}(\Pi(W_{\alpha\alpha}, \Re X) - \Pi(W_\alpha, \Re X_\alpha))\\
\tilde{r}_2' &= T_{\bar x}(\Pi(T_{1 + W_\alpha}R_\alpha, \Re X) - \Pi(\Re T_{1 + W_\alpha} R, X_\alpha)).
\end{align*}

We substitute this in the paralinearized equation, and apply our para-Leibniz 
rule in Lemma~\ref{l:para-L}. Whenever the para-material derivative or a regular derivative applies to the coefficient $T_{\bar x}$ (or any other low frequency coefficient) we get perturbative contributions which we can neglect. This greatly simplifies the computations below.

\medskip 
1) We compute the contributions in the first equation in \eqref{eq:para-linearized}:
\begin{align*}
T_{D_t} \tilde{w}_2' &= T_{\bar x}(-\Pi(T_{1 - \bar Y}T_{1 + W_\alpha}R_{\alpha\alpha}, \Re X) -\Pi(W_{\alpha\alpha}, \Re T_{1 - \bar Y} R) \\
&\quad +\Pi(T_{1 - \bar Y}T_{1 + W_\alpha}R_{\alpha}, \Re X_\alpha)  +\Pi(W_\alpha,\Re T_{1 - \bar Y} R_{\alpha})) + G ,\\
T_{1 - \bar Y} \D_\alpha \tr_2' &= T_{\bar x}(\Pi(T_{1 - \bar Y}T_{1 + W_\alpha}R_{\alpha\alpha}, \Re X)  + \Pi(T_{1 - \bar Y}T_{1 + W_\alpha}R_{\alpha}, \Re X_\alpha) \\
&\quad - \Pi(T_{1 - \bar Y}\Re T_{1 + W_\alpha}R_{\alpha}, X_\alpha) - \Pi(T_{1 - \bar Y}\Re T_{1 + W_\alpha}R, X_{\alpha\alpha})) + G\\
&= T_{\bar x}(\Pi(T_{1 - \bar Y}T_{1 + W_\alpha}R_{\alpha\alpha}, \Re X) 
+ \Pi(T_{1 - \bar Y}T_{1 + W_\alpha}R_{\alpha}, \Re X_\alpha) \\
&\quad - \Pi(\Re T_{1 - \bar Y}R_{\alpha}, W_\alpha)- \Pi(\Re T_{1 - \bar Y}R, W_{\alpha\alpha})) + G.
\end{align*}
Summing the two contributions and observing cancellations, we obtain the source terms
\begin{align*}
& T_{\bar x}(\Pi(T_{1 - \bar Y}T_{1 + W_\alpha}R_{\alpha}, 2\Re X_\alpha) -\Pi(W_{\alpha\alpha}, 2\Re T_{1 - \bar Y} R)) + G
 \end{align*}
which match the remaining terms in $\mathcal{G}_{0, 1}$, recorded in $\mathcal{G}_{0, 1}'$, \eqref{eq:para-linearized4.1}.

\

2) We compute the second equation in \eqref{eq:para-linearized}:
\begin{align*}
T_{D_t}\tr_2' &= T_{\bar x}(iT_{1 + a}\Pi(T_{1 + W_\alpha}Y_\alpha, \Re X) - \Pi(T_{1 + W_\alpha} R_\alpha, \Re T_{1 - \bar Y}R) \\
&\quad -  iT_{1 + a}\Pi(\Re T_{1 + W_\alpha} Y, X_\alpha) 
  + \Pi(T_{1 + W_\alpha}R, \Re T_{1 - \bar Y}R_\alpha)) + K \\
-iT_{1 - Y}T_{1 + a} \tw_2' &= - iT_{\bar x}T_{1 + a}\Pi(T_{1 + W_\alpha} Y_{\alpha}, \Re X) +  i T_{\bar x}T_{1 + a}\Pi(T_{1 + W_\alpha}Y, \Re X_\alpha) + K, \\
&= - iT_{\bar x}T_{1 + a}\Pi(T_{1 + W_\alpha} Y_{\alpha}, \Re X) +  i T_{\bar x}T_{1 + a}\Pi(\Re T_{1 + W_\alpha}Y, X_\alpha) + K.
\end{align*}
Summing the two contributions and observing cancellations, we obtain the source terms
\[
\half T_{\bar x}( \Pi (R, \bar R_\alpha) - \Pi(R_\alpha, \bar R))+K.
\]
Summing this with the remaining terms in $\mathcal{K}_{0, 1}$, recorded in $\mathcal{K}_{0, 1}'$, \eqref{eq:para-linearized4.1}, we obtain $\cK_{0, 2}$. 
\end{proof}

It remains to build a correction to match $\cK_{0, 2}$, consisting of errors in the second equation which are quadratic in $R$  and with $\bar x$ coefficients. Define 
\begin{align*}
\tilde{w}_3 &= iT_{1 + W_\alpha}(P[T_{\bar x} \bar X_\alpha X_\alpha] - iT_{\frac{1}{1 + a}}  P[T_{\bar x}  \bar R_\alpha R]) \\
&\quad + \frac{1}{2} T_{\bar x}\Re T_{1 + W_\alpha}(\Pi(\bar X_\alpha , X_\alpha) - iT_{\frac{1}{1 + a}} \Pi(\bar R_\alpha, R)) \\
&\quad +  \D_\alpha(T_{T_{\bar x} W_\alpha} X \\
&\quad+ \half T_{\bar x}\Pi( W_\alpha, X )) =:\tw_{3, 1} +\tw_{3, 2} +\tw_{3, 3} +\tw_{3, 4}, \\
\tilde{r}_3 &= iT_{1 + W_\alpha} P[T_{\bar x} \bar X_\alpha R]\\
&\quad + \frac{1}{2}T_{\bar x} \Re T_{1 + W_\alpha}\Pi( \bar X_\alpha, R) \\
&\quad +  T_{T_{\bar x}W_\alpha} R + T_{T_{\bar x}R_\alpha} W \\
&\quad + \half T_{\bar x} \D_\alpha \Pi(W, R) =: \tr_{3, 1} + \tr_{3, 2} + \tr_{3, 3} + \tr_{3, 4}.
\end{align*}
As usual, it is easy to verify \eqref{nf-bd} and \eqref{nf-bd2}.

\begin{lemma}\label{lem:cor6}
We have
$$\tilde \cG_3 = -\cG_{0, 2} + G = G, \qquad \tilde \cK_2 = -\cK_{0, 2} + K,$$
where $(G,K)$ satisfy the bound \eqref{nf-err-bd}.
\end{lemma}

\begin{proof}
It remains to match
\begin{align*}
\mathcal{K}_{0, 2}(w, r) &= T_{\bar x}(2ia +  2T_{1 + W_\alpha}T_{1 - \bar Y}T_{R_\alpha}R +\half \D_\alpha \Pi(T_{1 + W_\alpha} R, b) ) \\
&= T_{\bar x}(2ia + \half \D_\alpha \Pi(T_{1 + W_\alpha} R, T_{1 -  Y} \bar R)\\
&\quad +  2T_{1 + W_\alpha}T_{1 - \bar Y}T_{R_\alpha}R + \half \D_\alpha \Pi(T_{1 + W_\alpha} R, T_{1 - \bar Y} R)) + G \\
&=: \cK_{0, 2, 1} + \cK_{0, 2, 2} \\
&\quad + \cK_{0, 2, 3} + \cK_{0, 2, 4}.
\end{align*}
The Taylor coefficient term, and more generally the quadratic $\bar R_\alpha R$ and $\bar R R_\alpha$ terms collected in $\cK_{0,2,1}, \cK_{0, 2, 2}$, are corrected by $(\tw_{3, 1}, \tr_{3, 1}), (\tw_{3, 2}, \tr_{3, 2})$, respectively. The computations are easily adapted from the Taylor coefficient corrections of Lemma \ref{lem:cor3}, so we omit them here.

The remaining terms in $\cK_{0, 2, 3}, \cK_{0, 2, 4}$ are of the form $R_\alpha R$, and will be corrected by $(\tw_{3, 3}, \tr_{3, 3}), (\tw_{3, 4}, \tr_{3, 4})$, respectively. In the following, we demonstrate only the computations for the high-low interactions $\cK_{0, 2, 3}$, since the frequency balanced interactions in $\cK_{0, 2, 4}$ are similar. Thus, we are reduced to correcting
\begin{align*}
\mathcal{K}_{0, 2, 3} = 2T_{\bar x}T_{1 + W_\alpha}T_{1 - \bar Y}T_{R_\alpha}R
\end{align*}
with the corrections
\begin{align*}
\tilde{w}_{3, 3} &= \D_\alpha T_{T_{\bar x} W_\alpha} X, \\
\tilde{r}_{3, 3} &= T_{T_{\bar x}W_\alpha} R + T_{T_{\bar x}R_\alpha} W.
\end{align*}

\

1) We begin with the first equation, using as usual the para-Leibniz rule and the appropriate para-material derivative formulas, and observing that cases where the derivative falls on $\bar x$ are perturbative:
\begin{align*}
T_{D_t}\tw_{3, 3} &= - \D_\alpha(T_{T_{\bar x}(1 - \bar Y) R_\alpha} W + T_{T_{\bar x} W_\alpha }T_{1 - \bar Y} R) + G ,\\
T_{1 - \bar Y} \D_\alpha \tr_{3, 3} &= \D_\alpha(T_{1 - \bar Y}T_{T_{\bar x}R_\alpha} W + T_{1 - \bar Y}T_{T_{\bar x} W_\alpha} R) + G,
\end{align*}
so that we have no non-perturbative new contributions.

\

2) For the second equation, 
\begin{align*}
T_{D_t} \tr_{3, 3} &= -T_{T_{\bar x}T_{1 - \bar Y}T_{1 + W_\alpha}R_\alpha} R + iT_{T_{\bar x}W_\alpha}T_{1 + a} Y \\
&\quad + iT_{T_{\bar x}T_{1 + a}Y_\alpha} W - T_{T_{\bar x}R_\alpha} T_{1 - \bar Y}T_{1 + W_\alpha}R + G, \\
-iT_{1 - Y}T_{1 + a} \tw_{3, 3} &= -iT_{1 + a}(T_{T_{\bar x} W_\alpha}Y + T_{T_{\bar x} Y_\alpha}W)  + G. 
\end{align*}
We see the remaining terms, after perturbative modifications and cancellations, match $\mathcal{K}_{0, 2, 3}$.
\end{proof}

\section{Energy estimates for the full equation}
\label{s:ee}

The main goal of this section is to prove energy estimates for the $(\W,R)$ equation in $\dH^s$ for $s \geq 0$ with an $A_{\frac14}^2$ constant, as stated 
in Theorem~\ref{t:main}.

We are not only interested in having a good energy estimate, but also in understanding the structure of the differentiated evolution \eqref{ww2d-diff}.
 Thus, as part of the proof, we will also prove the renormalization result in Theorem~\ref{t:nf}. This will be useful in the next paper 
in our series, where we prove Strichartz estimates at low regularity.

We remark that, since $(\W,R)$ solves the linearized equation, 
the bounds for the linearized equation  in the previous section yield the case $s = \frac14$. We are interested in all $s$ but particularly the case $s \leq \frac34$.  The estimates for the linearized equation
will not work for $s \neq \frac14$, but here we can take advantage of the fact that $(w,r)=(\W,R)$ and balance things better.

\subsection{The main steps of the argument}

To begin the analysis, it is useful to recast \eqref{ww2d-diff} in a paradifferential form with a source term, based on the paradifferential equation \eqref{paralin(hwhr)}. Within the source term, we will peel off perturbative 
contributions, denoted by $(G,K)$ for the remainder of the section.
These are nonlinear expressions in $(\W,R)$ with the property that they satisfy 
favourable balanced cubic bounds of the form
\begin{equation}\label{good-err}
\|(G,K)\|_{\dH^s} \lesssim_A A_\frac14^2 \|(\W, R)\|_{\dH^s}.
\end{equation}
The remaining source terms are of two types:
\begin{enumerate}
    \item [(a)] quadratic terms in $(\W,R)$.
\item [(b)] unbalanced cubic terms, which can also be viewed as a quadratic 
expression in $(\W,R)$ but with a lower frequency coefficient which depends only on 
an undifferentiated $\W$.
\end{enumerate}
Implicitly, a similar analysis was carried out in \cite{HIT} by further differentiating the $(\W,R)$ equation. The difference is that in \cite{HIT}
all of the cubic source terms went into the perturbative box, whereas here we make a finer distinction.

We state the outcome of this computation in the following:

\begin{lemma}\label{l:para-rewrite}
The equation \eqref{ww2d-diff} can be rewritten as a paradifferential equation
for the variables $(\hat w,\hat r) = (\W,R)$, of the form
\begin{equation} \label{para-to-nonlin}
\left\{
\begin{aligned}
 &T_{D_t} \hat w  + T_{b_\alpha} \hat w + \D_\alpha T_{1 - \bar Y}T_{1 + W_\alpha} \hat r = \cG(\W, R) + G     \\
&T_{D_t} \hat r  + T_{b_\alpha} \hat r  -iT_{(1 - Y)^2}T_{1 + a}\hat w + T_M \hat r
= \cK(\W, R) + K,
\end{aligned}
\right.
\end{equation}
where
\begin{equation}\label{nonlin-source}
\begin{split}
\cG = & \ - \D_\alpha \left[ \Pi(\W, T_{1-\bar Y} R)+\Pi(\W, T_{1-Y} \bar R)
+ \Pi(\bar Y, T_{1 + \W} R)\right], \\
\cK = & \ - T_{1-\bar Y} \Pi(R_\alpha, R) - T_{1-Y} \Pi(R_\alpha, \bar R)
- iT_{1 - Y}T_{1 + a} \Pi(Y, \W)- T_{1-Y} \Pi(\bar R_\alpha, R).
\end{split}
\end{equation}
\end{lemma}

As in the previous section, here we include an implicit projection $P$ in $\Pi$. We note that in $\cG$ one can equivalently take out in front all the paraproducts. The above lemma is proved in Section~\ref{s:para-rewrite}.

We will divide the analysis of the equation \eqref{para-to-nonlin} into two steps.
We will first prove an $\dH^s$ well-posedness result for the 
paradifferential equation \eqref{paralin(hwhr)}. Fortunately we do not need to do this from scratch, and instead we can rely on the similar result proved in Proposition~\ref{p:paralin-wp} in the course of the analysis of the linearized equation. For this reason, we state the outcome of this step as follows:

\begin{proposition}\label{p:para2}
a) There exists a bounded linear transformation 
(independent of $s$) reducing \eqref{paralin(hwhr)} to the linearized paradifferential equation \eqref{paralin(wr)}. Precisely, given $(\hat w, \hat r)$ satisfying \eqref{paralin(hwhr)}, there exists $(w, r)$ satisfying \eqref{paralin(wr)FG} such that for any $s \in \R$ we have

(i) Invertibility:
\begin{align*}
\|(w_\alpha, r_\alpha) - (\hat w, \hat r) \|_{\dH^s} \lesssim_A A \|(w,r)\|_{\dH^s},
\end{align*}

(ii) Perturbative source term:
\begin{equation*}
\|(G, K)\|_{\dH^{s}} \lesssim_A A_{\frac14}^2 \|(w, r)\|_{\dH^{s}}.
\end{equation*}

b) As a consequence, Proposition \ref{p:paralin-wp} holds with \eqref{paralin(hwhr)} in place of \eqref{paralin(wr)}.
 \end{proposition}

This result is proved in Section~\ref{s:para2}; in particular, it yields a good $\dH^s$ energy functional for \eqref{paralin(hwhr)}. 

\

Our second step is to interpret the differentiated system \eqref{ww2d-diff}, expressed in the form \eqref{para-to-nonlin},
perturbatively based on the paradifferential equation \eqref{paralin(hwhr)}. However, the source terms $(\cG,\cK)$ are
not directly perturbative. Instead, we will show that there exists a favourable normal form transformation which largely eliminates the source terms:

\begin{proposition}\label{p:nf}
Given $(\W,R)$ satisfying \eqref{ww2d-diff}, there exist 
modified normal form variables $(\nfW, \nfR)$ satisfying an equation of the form
\begin{equation*}
\left\{
\begin{aligned}
 &T_{D_t} \nfW  + T_{b_\alpha} \nfW + \D_\alpha T_{1 - \bar Y}T_{1 + W_\alpha} \nfR = \tilde{\cG}(\W,R) \\
&T_{D_t} \nfR  + T_{b_\alpha} \nfR  -iT_{(1 - Y)^2}T_{1 + a}\nfW + T_M \nfR
= \tilde{\cK}(\W,R),
\end{aligned}
\right.
\end{equation*}
such that for any $s \geq 0$ we have

(i) Invertibility:
\begin{equation}
    \| (\nfW, \nfR) - (\W,R)\|_{\dH^s} \lesssim_A A \|(\W,R)\|_{\dH^s},
\end{equation}

(ii) Perturbative source term:
\begin{equation}
\| (\tilde{\cG}(\W,R),\tilde{\cK}(\W,R)\|_{\dH^s} \lesssim_A A_{\frac14}^2
\| (\W,R)\|_{\dH^s}.
\end{equation}
\end{proposition}
This result is proved in Section~\ref{s:nf-nonlin}.

\subsection{The paradifferential form of the differentiated equation}
\label{s:para-rewrite}

Our goal here is to prove Lemma~\ref{l:para-rewrite}. We succesively consider 
all terms in the two equations in \eqref{ww2d-diff}, which for convenience we recall here, written in a slightly more convenient algebraic fashion
\begin{equation*} 
\left\{
\begin{aligned}
 &D_t \W + (1-\bar Y)(1+\W) R_\alpha   =  (1+\W)M
\\
&D_t R = i\left(1-(1-Y)(1+a)\right).
\end{aligned}
\right.
\end{equation*}
We will use this equation in its projected form.

Along the way, we place perturbative terms in $(G,K)$. As a general rule,
all terms we place in $(G,K)$ are cubic and balanced, in the sense that the two highest frequencies are comparable and the low frequency variable is fully differentiated.

\medskip

1) For the first term in the first equation, expanding $b$ and discarding balanced terms, we have
\[
\begin{split}
P D_t \W =  & \ T_{D_t} \W + T_{\W_\alpha} P b +   \Pi(\W_\alpha,b)
\\ 
=  & \ T_{D_t} \W + T_{\W_\alpha} T_{1-\bar Y} R +   \Pi(\W_\alpha, T_{1-\bar Y} R
+ T_{1-Y} \bar R) + G.
\end{split}
\]

For the second term in the first equation we use a paraproduct expansion
discarding all balanced terms,
\[
\begin{split}
P[(1+\W)(1-\bar Y) R_\alpha] = & \  T_{1+\W} T_{1-\bar Y} R_\alpha
- PT_{1+\W} T_{R_\alpha} \bar Y + T_{1-\bar Y} T_{R_\alpha} \W 
\\ & \
 - T_{1+\W} \Pi(\bar Y, R_\alpha) + T_{1 - \bar Y} \Pi(\W,R_\alpha) + G.
\end{split}
\]

Finally for the source term in the first equation we use \eqref{M-def}
and Lemma~\ref{l:M} to get
\[
\begin{split}
P[(1+\W) M] = & \  T_{1+\W} PM + T_M \W + G = T_{1+\W}P [\bar Y_\alpha R - \bar R_\alpha Y ] + T_M \W+ G
\\
= & \ T_{1+\W} (T_{\bar Y_\alpha} R - T_{\bar R_\alpha} Y)
+ T_{1+\W} [\Pi(\bar Y_\alpha, R) - \Pi(\bar R_\alpha, Y) ] + G.
\end{split}
\]
We combine these terms, using  Lemmas~\ref{l:YtoW} and \ref{l:para-prod} to write 
\[
T_{1-\bar Y} T_{R_\alpha} \W + T_{1+\W} T_{\bar R_\alpha} Y  = 
T_{1-\bar Y} T_{R_\alpha} \W + T_{1-Y} T_{\bar R_\alpha} \W - T_M \W + G = T_{b_\alpha} \W + G,
\]
in order to obtain the first equation in \eqref{para-to-nonlin}.

\medskip

2) We now consider the two terms in the second equation, which we expand in a similar fashion:
\[
\begin{split}
P D_t R =  & \ T_{D_t} R + T_{R_\alpha}P b +   \Pi(R_\alpha,b)
\\ 
=  & \ T_{D_t} R + T_{R_\alpha} T_{1-\bar Y} R +   \Pi(R_\alpha, T_{1-\bar Y} R
+ T_{1-Y} \bar R) + K
\end{split}
\]
respectively, using Lemma~\ref{regularity for a},
\[
\begin{split}
P\left[\frac{\W-a}{1+\W}\right] = & \   P[ 1- (1+a)(1-Y)]  
\\
= & \ T_{1+a} Y - T_{1-Y} Pa + \Pi(a,Y)
\\
= & \ T_{1+a} Y - i T_{1-Y} [ \bar R_\alpha R] + K
\\
= & \ T_{1+a} Y - i T_{1-Y} T_{\bar R_\alpha} R - i T_{1-Y} \Pi(\bar R_\alpha, R) + K.
\end{split}
\]
Just as in the case of the first equation, we recombine
\[
T_{R_\alpha} T_{1-\bar Y} R + T_{1-Y} T_{\bar R_\alpha} R = T_{b_\alpha + M} R + K,
\]
and use Lemma~\ref{l:YtoW} to write
\[
T_{1+a} Y = T_{(1 - Y)^2}T_{1 + a}\W - T_{1 - Y} T_{1 + a}\Pi(Y, \W) + K,
\]
obtaining the second equation in \eqref{para-to-nonlin}.

\subsection{Well-posedness for the paradifferential flow}
\label{s:para2}

In this section we prove the paradifferential 
well-posedness result in Proposition~\ref{p:para2}.
One could also do this directly, but instead it is more 
efficient to transfer this result from the similar result 
in Proposition~\ref{p:paralin-wp}.

Consider a solution $(\hat w,\hat r)$ to \eqref{paralin(hwhr)}. Motivated by \eqref{full-relation}, define
\begin{equation}\label{para-relation}
 (w, r) = (\D_\alpha^{-1}\hatw, \D_\alpha^{-1}T_{1 + \W}\hatr - \D_\alpha^{-1}T_{R_\alpha} \D^{-1}_\alpha \hat w).
\end{equation}
It is easily seen that we have the norm equivalence
\begin{equation}
\|(w,r)\|_{\dH^{s+1}} \lesssim_A \| (\hat w,\hat r)\|_{\dH^{s}}.    
\end{equation}
We will prove that the equation for $(w,r)$ is a perturbation of 
the equation \eqref{paralin(wr)}, which will allow us to establish an equivalence 
between the $\dH^{s}$ well-posedness of 
\eqref{paralin(hwhr)} and the $\dH^{s+1}$ well-posedness of 
\eqref{paralin(wr)}.

We insert $(w,r)$  in \eqref{paralin(wr)} and compute the corresponding source terms.
For the first equation of the system for $(w, r)$ we have full cancellation
\begin{align*}
\D_\alpha(T_{D_t} w + T_{1 - \bar Y} \D_\alpha r + T_{1 - \bar Y}T_{R_\alpha} w) = T_{D_t} \hat w  + T_{b_\alpha} \hat w + \D_\alpha T_{1 - \bar Y}T_{1 + W_\alpha} \hat r = 0.
\end{align*}
For the second equation we use \eqref{ww2d-diff} and \eqref{paralin(hwhr)} to compute material derivatives, Lemma~\ref{l:para-L} to distribute para-material derivatives,
and Lemma~\ref{l:para-com} to commute paraproducts. Denoting by $K$ perturbative terms, i.e. which satisfy
\[
\|K\|_{\dot H^{s+\frac12}} \lesssim_A A_\frac14^2 \|(w, r)\|_{\dH^{s + 1}},
\]
we succesively evaluate the expression
\[
\cK = \D_\alpha(T_{D_t} r - i T_{1 - Y}T_{1 + a} w)
\]
as follows:
\begin{equation*}
\begin{aligned}
\cK &= T_{D_t} (T_{1 + \W}\hatr - T_{R_\alpha} \D_\alpha^{-1} \hat w) + T_{b_\alpha}(T_{1 + \W}\hatr - T_{R_\alpha} \D_\alpha^{-1} \hat w) - i\D_\alpha T_{1 - Y}T_{1 + a} w 
\\
&=  (T_{D_t \W} \hat r + T_{1 + \W}T_{D_t}\hatr -T_{\D_\alpha D_t R} w -T_{R_\alpha} T_{D_t}\D^{-1}_\alpha \hat w) + T_{1 + \W} T_{b_\alpha}\hatr\\
&\quad + i T_{1 + a} T_{Y_\alpha} w - i T_{1 - Y}T_{1 + a} w_\alpha + K
\\
&=  T_{1 + \W}(T_{D_t}+T_{b_\alpha}) \hatr - T_{R_\alpha} \partial^{-1} (T_{D_t}+T_{b_\alpha}) \hatw + T_{D_t \W} \hat r - T_{\D_\alpha D_t R} w
\\
&\quad + i T_{1 + a} T_{ Y_\alpha} w - i T_{1 - Y}T_{1 + a} w_\alpha + K
\\
&=  i T_{1 + \W}T_{(1+a)(1-Y)^2} \hatw - T_{1 + \W}T_M \hat r
+ T_{R_\alpha} T_{1-\bar Y}T_{1+\W} \hatr - T_{ (1+\W)(1-\bar Y) R_\alpha} \hat r + T_{(1+\W)M}\hat r 
\\
& \quad + i T_{\D_\alpha (1+a)(1-Y) } w
 + i T_{1 + a} T_{ Y_\alpha} w - i T_{1 - Y}T_{1 + a} \hatw + K
\\ 
& = K.
\end{aligned}
\end{equation*} 
Here we have harmlessly replaced the para-material derivatives of 
$\W$ and $R_\alpha$ with their full material derivatives, and also we have discarded 
the $\partial_\alpha a$ term using \eqref{a-bmo+}. Finally, at the last step we have manipulated 
paraproducts using Lemmas~\ref{l:para-com}, \ref{l:para-prod},~\ref{l:para-assoc}.

Thus, the $\dH^{s+ 1}$ well-posedness of \eqref{paralin(wr)} given by Proposition~\ref{p:paralin-wp} implies the $\dH^s$ well-posedness of \eqref{paralin(hwhr)}.

\subsection{The normal form transformation}
\label{s:nf-nonlin}
In this section we construct the normal form transformation
in Proposition~\ref{p:nf}. We recall the nonperturbative part of the source term,
\begin{equation*}
\begin{split}
\cG = & \ - \D_\alpha \left[ \Pi(\W, T_{1-\bar Y} R)+\Pi(\W, T_{1-Y} \bar R)
+ \Pi(\bar Y, T_{1 + \W} R)\right], \\
\cK = & \ - T_{1-\bar Y} \Pi(R_\alpha, R) - T_{1-Y} \Pi(R_\alpha, \bar R)
- iT_{1 - Y}T_{1 + a} \Pi(Y, \W)- T_{1-Y} \Pi(\bar R_\alpha, R),
\end{split}
\end{equation*}
where one could harmlessly commute outside the paraproducts in the first 
expression.


Then we define the correction
\begin{align*}
\tilde{\W} &= - \D_\alpha \Pi(\W,2\Re X), \\
\tilde{R} &= - \Pi(R_\alpha, 2\Re X) - \Pi(T_{1 - \bar Y}\bar \W,R).
\end{align*}
For a partial verification of these expressions, we remark that the quadratic part of this correction coincides with the balanced part of the derivative of the correction
in \eqref{nft1}, after switching to the good variables $(\W,R)$.

We now insert these corrections into the equation \eqref{paralin(hwhr)},
and verify that the generated source terms cancel $(\cG,\cK)$ modulo perturbative terms (i.e. which satisfy \eqref{good-err}).

 Note that for both equations, the contribution of the $T_{b_\alpha}$ term is perturbative. Also, in the first equation, the undifferentiated $\hat r$ terms are perturbative. Also recall that we use Lemma~\ref{l:para-L} to distribute the para-material derivative, and Lemma \ref{l:XY-mat} to compute the para-material derivative of $X$.

We compute the contributions to the first equation of \eqref{paralin(hwhr)}.
\begin{equation*}
\begin{aligned}
-T_{D_t} \D_\alpha \Pi(\W, 2\Re X) &= \D_\alpha( \Pi(T_{(1 - \bar Y)(1 + \W)} R_\alpha, 2\Re X) + \Pi(\W, T_{1-\bar Y} R + T_{1-Y} \bar R) + G, \\
-T_{1 -\bar Y} T_{1 + \W} \D_\alpha \Pi(R_\alpha,2\Re X) &= - \D_\alpha \Pi(T_{(1 - \bar Y)(1 + \W)} R_\alpha, 2\Re X) + G, \\
-T_{1 -\bar Y} T_{1 + \W} \D_\alpha\Pi(T_{1 - \bar Y}\bar \W, R) &= - \D_\alpha \Pi(\bar Y, T_{1 + \W} R) + G.
\end{aligned}
\end{equation*}
After cancellations, only the second term of the first contribution and the last contribution remain, modulo perturbative $G$ terms, matching $\cG$. 

Similarly, for the second equation of \eqref{paralin(hwhr)} we have
\begin{equation*}
\begin{aligned}
-T_{D_t}\Pi(R_\alpha, 2\Re X) &= -i\Pi(T_{1 + a}Y_\alpha, 2\Re X) + \Pi(R_\alpha, (T_{1-\bar Y} R + T_{1-Y} \bar R)) + K, \\
-T_{D_t} \Pi(T_{1 - \bar Y}\bar \W,R) &= T_{1 - Y}\Pi(\bar R_\alpha, R) -iT_{1 + a}\Pi(T_{1 - \bar Y}\bar \W, Y) + K, \\
iT_{(1 - Y)^2}T_{1 + a} \D_\alpha \Pi(\W,2\Re X) &= i\Pi(T_{1 + a}Y_\alpha, 2\Re X) \\
&\quad  +  iT_{1 + a} (T_{1 - Y} \Pi(\W, Y)  +\Pi(Y, \bar X_\alpha)) + K,
\\
T_M \tilde R &= K.
\end{aligned}
\end{equation*}
After cancellations, we match $\cK$ modulo perturbative $K$ terms. Setting
$$(\nfW, \nfR) = (\W + \tilde{\W}, R + \tilde{R})$$
yields the desired normal form variables.

\section{Local well-posedness}
\label{s:lwp}

In this section we prove the low regularity local well-posedness result 
stated in Theorem \ref{t:lwp}. The result from our previous paper \cite{HIT}  asserts that local well-posedness holds for more regular data. We use those more regular solutions to construct the rough solutions by a scheme 
which is similar to the one used there. Precisely, we truncate the data in frequency and then move through a continuous family of solutions, thereby estimating only a solution for a  linearized equation at each step. We implement this strategy directly on the equations \eqref{ww2d-diff}  
which are in terms of the diagonal variables  $(\W,R)$. 
We begin with the main well-posedness result in \cite{HIT}:
\begin{theorem}\label{t:lwp-hit}
  a) Let $n \geq 1$. Then the problem \eqref{ww2d1} is locally
  well-posed in for initial data $(\W_0,R_0)$ in $\H^n$.

b) (lifespan) There exists $T= T(\|(\W_0,R_0)\|_{\H^1}, \|Y_0\|_{L^\infty})$
so that the above solutions are well defined in $[0,T]$, with uniform bounds.

\end{theorem}

Our goal will be to obtain an $ \H^s$ version of the above theorem 
with $s = \frac34$.
We recall that the well-posedness result in part (a) carries different meanings
depending on $n$. If $n \geq 2$, then we obtain existence and uniqueness
in $C(\H^n)$ together with continuous dependence on the data.  On the other hand if $n =1$ then we only produce rough solutions $C(\H^1)$  as the unique strong limit of smooth solutions, again with continuous
dependence  on the data; however, for $n=1$ we do not establish a direct uniqueness result.

\subsection{\texorpdfstring{$\mathcal \dH^s $}\ \  bounds for regular solutions}

The solutions in the last theorem have a lifespan which depends on the $\H^1$
size of data. Here we prove that in effect the lifespan depends only on the
$\H^{\frac{3}{4}}$ size of data, and that we have uniform bounds for as long as the $\H^{\frac{3}{4}}$ size of the solutions is controlled.

Precisely, suppose that for some $n \geq 2$ we have an   $\H^n$ solution $(\W,R)$ which satisfies the bounds
\begin{equation*}
\| (\W,R)(0)\|_{\H^{\frac{3}{4}}} < \cM_0 \ll 1.
\end{equation*}
Then we claim that there exists $T = T(\cM_0)$ so that the solution
exists in $C([0,T];\H^n)$ and satisfies the bounds
\begin{equation}\label{h1bd}
\| (\W,R)\|_{L^\infty(0,T; \H^{\frac{3}{4}})} < \cM(\cM_0),
\end{equation}
as well as the $\H^n$ and $\dH^n $ bounds
\begin{equation*}
\| (\W,R)\|_{L^\infty(0,T; \H^n)} \leq  C (\cM_0)\| (\W,R)(0)\|_{ \H^n},
\end{equation*}
\begin{equation*}
\| (\W,R)\|_{L^\infty(0,T; \dH^n)} \leq  C (\cM_0)\| (\W,R)(0)\|_{ \dH^n}.
\end{equation*}

To prove this, we begin by making the bootstrap assumption
\begin{equation}
\label{h1boot}
\| (\W,R)\|_{L^\infty(0,T;\H^{\frac{3}{4}})} < 2\cM.
\end{equation}
We will show that for a suitable choice  $\cM(\cM_0)$ depending only on $\cM_0$, we can improve 
this to \eqref{h1bd}, provided that $T < T(\cM_0)$.

From the bootstrap assumption \eqref{h1boot} we bound the control parameter 
\[
A_{\frac{1}{4}}\lesssim \mathcal{M}.
\]
By applying the energy estimates obtained in Theorem~\ref{t:main}, and Gronwall's inequality
to  $(\W,R)$
\begin{equation}\label{h1-dh0}
\| (\W,R)(t)\|_{\H^{\frac{3}{4}}} \lesssim  e^{C t } \| (\W,R)(0)\|_{\H^{\frac{3}{4}}}\lesssim e^{Ct}\cM_{0}
, \qquad C = C(\cM).
\end{equation}
It is here that our smallness assumption on $A$ is used, as it guarantees that 
our energies $E_s$ are comparable with the $\dH^s$ norm of the solution.

The above estimate allows us to chose first $\cM$ large enough and then $T$ such that the bound in \eqref{h1bd} holds for $t\in [0.T]$.

 At the same time, applying the energy estimates proven in Theorem~ \ref{t:main} and Gronwall's inequality for higher norms the  pair
 $(\W,R)$ we get
\begin{equation}\label{h1-dh1}
\| (\W,R) (t)\|_{\H^n} \lesssim  e^{C  t }
\| (\W,R)(0)\|_{\H^n},
\end{equation}
which allow us to propagate $(\W, R)$ as classical solutions up to time $T$ using Theorem~\ref{t:lwp-hit}.

\subsection{Rough solutions}

Here we construct solutions for data in $\H^{\frac{3}{4}}$ as unique limits of
smooth solutions. Given a $\H^{\frac{3}{4}}$ initial data $(\W_0,R_0)$ as above
we regularize it to produce smooth approximate data $(\W_0^k,R_0^k) =
P_{< k} (\W_0,R_0) $.  We denote the corresponding solutions by
$(\W^k,R^k)$. By the previous analysis, these solutions  exist on a
$k$-independent time interval $[0,T]$ and  satisfy uniform $ \H^{\frac{3}{4}}$
bounds. Further, they are smooth and have a smooth dependence on $k$.

To better understand the evolution of the $\dH^{\frac34}$ norm of the solution
it is convenient to use the language of frequency envelopes. 

\begin{definition}
A sequence $c_k \in \ell^2$ is called an $\H^s$ frequency envelope for
$(\W,R) \in \H^s$ if 

(i) it is slowly varying, $c_j/c_k \leq
2^{-\delta|j-k|}$ with a small universal constant $\delta$, and 

(ii) it bounds the dyadic norms of $ (\W,R) $, namely $\|P_k (\W,R)
\|_{\H^s} \leq c_k$.
\end{definition}

Consider a frequency envelope $c_k$ for the initial data $(\W_0,R_0)$  in $\H^{\frac{3}{4}}$.
Then for the regularized data we have
\[
\| (\W^k_0,R^k_0)\|_{\mathcal{H}^{n}} \lesssim c_k 2^{\left(n-\frac{3}{4}\right)k}, \qquad n \geq 1.
\]
Hence, in the time interval $[0,T]$  we also have the estimates
\begin{equation}\label{high(W,K)}
\| (\W^k,R^k)\|_{\mathcal H^n} \lesssim c_k 2^{\left(n-\frac{3}{4}\right)k}, \qquad n \geq 1.
\end{equation}

We will use these for the high frequency part of the regularized solutions.

For the low frequency part, on the other hand, we 
view $k$ as a continuous rather than a discrete parameter,
differentiate $(\W^k,R^k)$ with respect to $k$ and use the estimates
for the linearized equation. One minor difficulty is that the linearized 
equation \eqref{lin(wr)0} arises from the linearization of 
the $(W,Q)$ system in \eqref{ww2d1} rather than the differentiated
$(\W,R)$ system in \eqref{ww2d-diff}. 
Assuming that $(W^k,Q^k)$ were also defined, we formally  
denote  
\[
(w^k,r^k) = (\partial_k
W^k, \partial_k Q^k - R^k \partial_k W^k).
\]
These would solve the linearized equation around the $(\W^k,R^k)$ solution. 
For our analysis we want to refer only to the differentiated variables, so we 
we compute
\[
\begin{split}
\partial_\alpha w^k = & \  \partial_k \W^k, \\
\partial_\alpha r^k = & \ (1+\W^k) \partial_k R^k - R_\alpha^k w^k.
\end{split}
\]
We take these formulas as the definition of $(w^k,r^k)$, and observe 
that inverting the $\partial_\alpha$ operator is straightforward 
since the above multiplications involve only holomorphic factors 
therefore the products are at frequency $2^k$ and higher.
To take advantage of the bounds in Theorem~\ref{t:balancedenergy} for
the linearized equation, we need a $\dH^{\frac{1}{4}}$ bound for
$(w^k(0),r^k(0))$, namely
\begin{equation}
\| (w^k(0),r^k(0)) \|_{\dH^{\frac{1}{4}}} \lesssim c_k 2^{-\frac{3}{2}k}.
\end{equation}
The bound for $w^k(0)$ is straightforward, but some work is required
for $r^k(0)$.  This follows via the usual Littlewood-Paley trichotomy
and Bernstein's inequality for the low frequency factor, with the
twist that, since both factors are holomorphic, no high-high to low
interactions occur.

In view of the uniform $\H^{\frac{3}{4}}$ bound for $(\W^k,R^k)$,
Theorem~\ref{t:balancedenergy} shows that in $[0,T]$ we have the uniform
estimate
\begin{equation}
\| (w^k,r^k) \|_{\dH^{\frac{1}{4}}} \lesssim  c_k 2^{-\frac{3}{2}k}.
\end{equation}
Now, we return to $(\W^k, R^k)$ and claim the bound
\begin{equation}\label{diff(W,K)}
\| P_{\leq k} (\partial_k \W^k, \partial_k R^k) \|_{\dH^{\frac{1}{4}}} \lesssim  c_k 2^{-\frac{1}{2}k}.
\end{equation}
Again the $\W^k$ bound is straightforward. For $\partial_k R^k$ we write
\[
\partial_k R^k = (1- Y^k)(\partial_\alpha r^k + R_\alpha^k \partial_k W^k),
\]
where again all factors are holomorphic. Then applying $P_{\leq k}$ restricts all frequencies
to $\lesssim 2^k$, and the  Littlewood-Paley trichotomy and Bernstein's
inequality again apply.

Now we integrate \eqref{diff(W,K)} over unit $k$ intervals and use it
to estimate the differences. Combining the result  with \eqref{high(W,K)} we obtain
\begin{equation}
\label{h14}
\begin{split}
&\| (\W^{k+1} - \W^k, R^{k+1}-R^k)\|_{\dH^{\frac{1}{4}}} \lesssim   \ c_k 2^{-\frac{1}{2}k},
\\
& \| \partial_\alpha^2(\W^{k+1} - \W^k, 
 R^{k+1}-R^k)\|_{\dH^{\frac{1}{4}}} \lesssim  \ c_k 2^{\frac{3}{2}k}.
\end{split}
\end{equation}

Summing up with respect to $k$ it follows that the sequence
 $(\W^k,R^k)$ converges uniformly in $\dH^{\frac{3}{4}} \cap \dH^{\frac14}$ in the time interval $[0,T]$
 to a solution $(\W,R)$. As the sequence $(\W^k,R^k)$ is uniformly bounded
 in $\H^{\frac34}$, it follows that we also have $(\W,R) \in \H^\frac34$.
 Furthermore, it is easily seen that the solution $(\W,R)$
 inherits the frequency envelope bounds from the data.
 
 For the continuous dependence  on the initial data argument, we need to show that the sequence of solutions $(\W^k,R^k)$  converges to $(\W,R)$ in the non-homogeneous $\H^{\frac{3}{4}}$ topology. For the computations below, we note that the solutions $(\W^k,R^k)$ and $(\W,R)$ have uniformly bounded control parameters $A$, $A_\frac14$ and $\As$, so we will use these notations without any $k$ indices. In addition, we introduce 
 \[
 \Ass = \| (\W_0,R_0)\|_{\dH^\frac34},
 \]
 which uniformly controls the $\dH^\frac34$ norm of all the solutions $(\W^k,R^k)$ and $(\W,R)$. We also  introduce the notation 
 \[
 \delta \W^k := \W^k-\W, \qquad \delta R^k :=  R^k -R.
 \]
 
 From  \eqref{h14} we get the $\dH^{\frac{1}{4}}$ convergence of the approximating sequence
 \begin{equation}\label{h14-diff}
 \Vert (\delta \W^k, \delta R^k)\Vert_{\dH^{\frac{1}{4}}}\lesssim 2^{-\frac{k}{2}}.
 \end{equation}
 We will use this, together with  the equations \eqref{ww2d-diff} for $(\W^k,R^k)$ and $(\W,R)$, in order to supplement this with a low frequency bound
 \begin{equation}\label{h0-diff}
 \Vert (\delta \W^k, \delta R^k)\Vert_{\dH^{0}}\lesssim_{\As,\Ass,T} 2^{-\frac{k}{4}}.
 \end{equation}
 By construction this holds at the initial time, so we need to propagate it.
 
 We first claim that from \eqref{h14-diff} we obtain the following difference
 bounds:
 \begin{equation}
 \label{differences}
 \Vert \delta Y^k\Vert_{\dot{H}^{\frac{1}{4}}} + \Vert \delta a^k\Vert_{\dot{H}^{\frac{1}{4}}}+ \Vert \delta b^k\Vert_{\dot{H}^{\frac{1}{4}}} + \Vert \delta M^k\Vert_{\dot H^{-\frac14}} \lesssim_{\As} 2^{-\frac{k}2}
\end{equation}
where $a^k$ and $b^k$ are the corresponding $a$ and  $b$ where $(\W, R)$ is replaced with $(\W^k, R^k)$.  These are all  relatively straightforward balanced $L^2$ bounds, whose proofs are left for the reader.
The bounds for $\delta \W^k$ and $\delta b^k$ are too low in regularity for our purposes, but we can interpolate with the uniform $\dH^\frac34$ bounds to  get 
 \begin{equation}
 \label{differences1}
   \Vert \delta b^k\Vert_{\dot{H}^{\frac{1}{2}} \cap L^\infty}+ \Vert \delta \W^k\Vert_{\dot{H}^{\frac{1}{2}} \cap L^\infty}  \lesssim_{\As,\Ass} 2^{-\frac{k}4}.
\end{equation}

Subtracting the equations \eqref{ww2d-diff} for $(\W^k,R^k)$ and $(\W,R)$ we obtain 
\begin{equation*}
\left\{
\begin{aligned}
  D_t \delta \W = & \  \delta b^k  \W^k_\alpha   -(1+\W)(1-\bar Y)\delta R^k_{\alpha} + (1-\bar Y) R^k_{\alpha}\delta \W^k - (1+\W^k) R^k_{\alpha}\delta \bar Y^k\\
  & +(1+ \W^k)\delta M^k + M\delta \W^k \\
  D_t\delta R^k = & \ \delta b^k  R^k_{\alpha}-i(1-Y)\delta a^k +i(1+a^k)\delta Y^k.
\end{aligned}    
\right.
\end{equation*}
Here we use the relations \eqref{differences}, \eqref{differences1}
to estimate all the terms on the right in $\dH^{-\frac14}$, obtaining
\begin{equation*}
\| D_t \delta \W^k  \|_{\dot H^{-\frac14}} + \| D_t \delta R^k  \|_{\dot H^{\frac14}}
\lesssim_{\As,\Ass} 2^{-\frac{k}4}
\end{equation*}
with $D_t = \partial_t + b \partial_\alpha$.
In the last bound we can harmlessly convert $D_t$ to $T_{D_t}$ modulo errors that can be placed on the right,
\begin{equation*}
\| T_{D_t} \delta \W^k \|_{\dot H^{-\frac14}} + \| T_{D_t} \delta R^k \|_{\dot H^{\frac14}}
\lesssim_{\As,\Ass} 2^{-\frac{k}4}.
\end{equation*}
These two bounds allow us to compute energy estimates
for $\delta \W_k$ and $\delta R_k$ as follows:
\[
\begin{split}
\frac{d}{dt} \| \delta \W^k\|_{L^2}^2 & =   \int T_{b_\alpha} \delta \W^k \cdot \delta \bar \W^k d\alpha + 2 \Re \int T_{D_t} \delta \W^k \cdot \delta \bar \W^k d\alpha \\ 
& \lesssim_A  \  \|   \delta \W^k \|_{\dot H^\frac14}^2  + 
\|   \delta \W^k \|_{\dot H^\frac14} \| T_{D_t}  \delta \W^k \|_{\dot H^{-\frac14}} \\
& \lesssim_{\As,\Ass} 2^{-\frac{k}4},
\end{split}
\]
and a nearly identical computation for $\delta R^k$.
This yields the desired bound \eqref{h0-diff} for the low frequencies of the difference.

Now we return to the last step of the proof of the well-posedness result, which is the continuous dependence on the initial data in the $\H^{\frac{3}{4}}$ topology. The frequency envelope bounds are very useful
in this proof. This is standard, but we
briefly outline the argument. Suppose that $(\W_j,R_j)(0) \in \H^{\frac{3}{4}}$ and
$(\W_j,R_j)(0)-(\W,R)(0) \to 0 $ in $\H^{\frac{3}{4}}$. We consider the approximate
solutions $(\W_j^k,R_j^k)$, respectively $(\W^k,R^k)$. According to
our result for more regular solutions, we have 
\begin{equation}\label{reg-lim}
(\W_j^k,R_j^k) - (\W^k,R^k) \to 0 \qquad \text{in} \ \ \H^n.
\end{equation}
On the other hand, from the $\H^{\frac{3}{4}}$ data convergence we get
\[
(\W_j^k,R_j^k)(0) - (\W_j,R_j)(0) \to 0  \qquad \text{in} \ \ \H^{\frac{3}{4}} \ \ \text{uniformly in} \ \ j.
\]
Then the above frequency envelope analysis,  shows that
\[
(\W_j^k,R_j^k) - (\W_j,R_j) \to 0 \qquad \text{in} \ \ \H^{\frac{3}{4}} \ \ \ \text{uniformly in} \ \ j.
\]
Hence we can let $k$ go to infinity in \eqref{reg-lim}
and conclude that
\[
(\W_j,R_j) - (\W,R) \to 0 \qquad \text{in} \ \ \H^{\frac{3}{4}}.
\]

\subsection{Enhanced cubic lifespan bounds}
\label{s:cubic}
In this section we prove Theorem~\ref{t:cubic}. Given initial data $(\W,R)$ for \eqref{ww2d-diff} 
satisfying the conditions in the theorem, namely 
\[
\| (\W,R)(0)\|_{\dH^{\frac{3}{4}}}  \leq \epsilon,
\qquad
\| (\W,R)(0)\|_{\dH^{\frac12-\delta}}^\frac14  
\| (\W,R)(0)\|_{\dH^{\frac{3}{4}}}^\delta  \leq c_0^{\frac14 + \delta},
\]
we consider the solutions on a time interval $\left[0, T \right] $ and seek to prove the estimate 
\begin{equation}
\label{miki}
\| (\W,R)(t)\|_{\dH^{\frac{3}{4}}}  \leq C \epsilon,
\quad
\| (\W,R)(t)\|_{\dH^{\frac12-\delta}}^\frac14  
\| (\W,R)(t)\|_{\dH^{\frac{3}{4}}}^\delta  \leq (C c_0)^{\frac14+\delta},
\quad t\in \left[0, T \right], \!\!\!\!
\end{equation}
provided that  $T\ll \epsilon^{-2}$.  In view of our local well-posedness result this shows that the solutions can be extended up to time $T_{\epsilon}=c\epsilon^{-2}$ concluding the proof of the theorem.

In order to prove \eqref{miki} we can harmlessly make the bootstrap assumption
\begin{equation}
\label{miki2}
\| (\W,R)(t)\|_{\dH^{\frac{3}{4}}}  \leq 2C \epsilon,
\quad
\| (\W,R)(t)\|_{\dH^{\frac12-\delta}}^\frac14  
\| (\W,R)(t)\|_{\dH^{\frac{3}{4}}}^\delta  \leq (2 C c_0)^{\frac14+\delta},
\quad t\in \left[0, T \right].\!\!\!\!
\end{equation}

By Sobolev embeddings and interpolation, from \eqref{miki2} we obtain the bounds
\begin{equation*}
A \lesssim C c_0 \ll 1, \qquad  A_{\frac{1}{4}}\lesssim C\epsilon.
\end{equation*}
Hence, by the energy estimates in Theorem~\ref{t:main} applied to $(\W, R)$ with $s = \frac34$, we obtain
\[
\begin{split}
\Vert (\W, R)\Vert_{L^{\infty}(0,T; \dH^{\frac{3}{4}})} 
\lesssim & \ \Vert (\W, R)(0)\Vert_{ \dH^{\frac{3}{4}}}+TA^2_{\frac{1}{4}}\Vert (\W, R)\Vert_{L^{\infty}(0,T;\dH^{\frac{3}{4}})}\\
\lesssim & \ \epsilon +TC^3\epsilon^3.
\end{split}
\]
Combining this with  the same energy estimates in Theorem~\ref{t:main} but for $s = \frac12-\delta$, we also get
\[
\| (\W,R)(t)\|_{\dH^{\frac12-\delta}}^\frac14  
\| (\W,R)(t)\|_{\dH^{\frac{3}{4}}}^\delta  \lesssim (c_0(1+T\epsilon^2))^{\frac14+\delta}.
\]
Hence, the desired estimate \eqref{miki} follows provided that $T\ll (C\epsilon)^{-2}$.

\bibliography{bib-holo-gw}
\bibliographystyle{plain}

\end{document}